\documentclass[draft]{amsart} 
\usepackage{amsthm}
\usepackage{amsmath}
\usepackage{amsfonts}
\usepackage{amssymb}
\usepackage[mathscr]{euscript}
\usepackage{color}
\usepackage[arrow,curve,frame,color]{xy}
\usepackage[final,hidelinks]{hyperref} 
\def\ball#1,#2.{B(#1,#2)} 
\def\epsi{\varepsilon}
\def\metric{d} 
\def\dist#1,#2.{\metric(#1,#2)} 
\def\setdist#1,#2.{{\rm dist}(#1,#2)}
\def\hdist#1,#2.{\metric_H(#1,#2)}
\def\natural{{\mathbb N}} 
\def\clball#1,#2.{\bar B(#1,#2)} 
\def\albrep#1.{{\mathcal A}_{#1}} 
\def\mpush#1.{{#1}_{\sharp}} 
\def\glip#1.{{\bf L}(#1)} 
\def\elleone#1.{{ L}^1(#1)} 
\def\elleinfty#1.{{L}^\infty(#1)} 
\def\bborel#1.{{\mathcal B}^{\infty}(#1)} 
\def\lipalg#1.{{\rm Lip}_{\text{\normalfont b}}(#1)} 
\def\lipfun#1.{{\rm Lip}(#1)} 
\def\onenorm#1.{\left\|#1\right\|_{\elleone\mu.}} 
\def\real{{\mathbb{R}}}
\def\mrest{{\,\evansgariepyrest}}
\def\arenseels#1.{{\rm AE}[#1]}
\def\cone{{\mathcal C}} 
\def\lebmeas{{\mathcal L}^1} 
\def\nlebmeas#1.{\mathscr{L}^{\setbox0=\hbox{#1\unskip}\ifdim\wd0=0pt 1
    \else #1 \fi}} 
\def\pathfrag#1.{{\rm Frag}(#1)} 
\def\hmeas#1.{\mathscr{H}^{#1}} 
\def\pullb#1.{{#1}^{\sharp}} 
\def\cmass#1.{{\|#1\|}}
\def\crest{{\,\evansgariepyrest}}
\def\mcurr#1,#2.{{\bf M}_{#1}(#2)} 
\def\mcurr#1,#2.{{\bf M}_{#1}(#2)} 
\def\ncurr#1,#2.{{\bf N}_{#1}(#2)} 
\def\pcurr#1,#2.{{\bf P}_{#1}(#2)} 
\def\mfunc#1,#2.{{\rm MF}_{#1}(#2)} 
\def\convgeo{{\mathscr{K}}}
\def\QG#1,#2.{{\rm QG}(#1,#2)}
\def\tnorm#1.{{{\left\|#1\right\|_{TX}}}}
\def\cotnorm#1.{{{\left\|#1\right\|_{T^*X}}}}
\def\metuple#1,#2.{{\mathscr{D}}^{#1}({#2})}
\def\frest#1.{{\rm Red}_{#1}}
\newcommand{\chartfun}[1][j]{x_\alpha^j}
\DeclareMathOperator\spt{spt} 
\DeclareMathOperator\dom{dom} 
\DeclareMathOperator\metdiff{md} 
\DeclareMathOperator\paths{Path} 
\DeclareMathOperator\frags{Frag} 
\DeclareMathOperator\diam{diam}
\DeclareMathOperator\sgn{sgn}
\DeclareMathOperator\curves{Curves}
\DeclareMathOperator\pieces{Pieces}
\DeclareMathOperator\rprm{Rep}
\def\balt#1,#2.{{\mathcal{BA}\left(#1,#2\right)}}
\def\balta#1,#2,#3.{{\mathcal{BA}_{(#3)}\left(#1,#2\right)}}
\def\boundlin#1,#2.{{\mathcal{BL}\left(#1,#2\right)}}
\def\extpow#1.{\widehat\bigwedge^{#1}}
\def\extpowm#1,#2.{\widehat\bigwedge^{#1}_{#2}}
\def\frest#1.{{\rm Red}_{#1}}
\def\dualmod#1.{\hom\left({#1},L^\infty(\mu)\right)}
\def\boundlin#1,#2.{{\mathcal{BL}\left(#1,#2\right)}}
\def\locdernorm#1,#2.{{\left|{#1}\right|_{{\rm  Der}({#2}),\text{loc}}}}
\def\grass#1,#2.{{\normalfont\rm Gr}({#1},{#2})}
\newcommand{\evansgariepyrest}{\mathbin{\vrule height 1.3ex depth%
    0pt width 0.08ex\vrule height 0.08ex depth 0pt width 1.0ex}}
\def\balt#1,#2.{{\mathcal{BA}\left(#1,#2\right)}}
\def\balta#1,#2,#3.{{\mathcal{BA}_{(#3)}\left(#1,#2\right)}}
\def\boundlin#1,#2.{{\mathcal{BL}\left(#1,#2\right)}}
\def\extpow#1.{{\widehat\bigwedge^{#1}}}
\def\extpowm#1,#2.{\widehat\bigwedge^{#1}_{#2}}
\def\strip#1,#2.{{\rm St}\left({#1},{#2}\right)}
\def\wder#1.{{\mathscr{X}}({#1})}
\def\wform#1.{{\mathscr{E}}({#1})}
\def\hwder#1,#2.{{\mathscr{X}^{#1}}({#2})} 
\def\hwform#1,#2.{{\mathscr{E}^{#1}}({#2})} 
\def\locnorm#1,#2.{\left|{#1}\right|_{{#2},\text{\normalfont loc}}}
\def\elocnorm#1,#2.{\left|{#1}\right|^{(\varepsilon)}_{{#2},\text{\normalfont loc}}}
\def\rational{{\mathbb Q}}
\def\ccgengroup#1.{{\mathbb {#1}}}
\def\liegenalg#1.{{\mathfrak{#1}}}
\def\Der#1.{\text{\normalfont Der}_{#1}}
\def\Cur#1.{\text{\normalfont Cur}_{#1}}
\def\alt#1,#2,#3.{\text{\normalfont Alt}_{#1}(#2; #3)} 
\def\createcat#1#2{\expandafter\def\csname
  #1cat\endcsname{#2}} 
\createcat{ban}{{\normalfont\textbf{Ban}}}
\createcat{lmod}{{}_{\mu}^\infty{\normalfont\textbf{Mod}}}
\createcat{lnmod}{{}_{\mu}^\infty{\normalfont\textbf{Mod}_{\text{\normalfont
      loc}}}}
\def\pcreatedst#1#2{\expandafter\def\csname
  #1dst\endcsname##1,##2.{\setbox0=\hbox{##1\unskip}\ifdim\wd0=0pt {#2}
  \else {#2(##1,##2)}\fi}}
\pcreatedst{e}{d^{(\epsi)}}
\pcreatedst{em}{d_M^{(\epsi)}}
\DeclareMathOperator{\Ext}{Ext}
\def\banext#1,#2.{\Ext^{#1}#2} 
\def\lmodext#1,#2,#3.{\Ext^{#1}_{#2}#3} 
\def\lnmodext#1,#2,#3.{\Ext^{#1}_{#2,\text{\normalfont loc}}#3}  
\newcount\albglue
\newcount\albertirepprod
\newcount\lemmeasfix
\newcount\partialderivatives
\def\syncref#1{\ifcase#1 2.46 \or 2.64 \or 3.1 \or 4.1\fi}
\def\makeunderscoreletter{\catcode`_=11}
\def\makecolonletter{\catcode`:=11}
\def\unmakecolonletter{\catcode`:=12}
\def\makeunderscoresub{\catcode`_=8}
\def\sync #1.{\makeunderscoreletter\makecolonletter\expandafter\ifx\csname
  sync#1\endcsname\relax Undefined\else \csname
  sync#1\endcsname\fi\makeunderscoresub\unmakecolonletter}
\makeunderscoreletter\def\syncalberti_rep_prod{Theorem
  2.67}\makeunderscoresub
\makeunderscoreletter\def\syncalb_glue{Theorem
  2.49}\makeunderscoresub
\makeunderscoreletter\makecolonletter\def\syncrem:derivation_extension{Remark
  2.115}\makeunderscoresub\unmakecolonletter
\makeunderscoreletter\makecolonletter\def\synconedimapprox_multi{Theorem
  3.66}\makeunderscoresub\unmakecolonletter
\makeunderscoreletter\makecolonletter\def\synclem:loc_estimate_dist{Lemma
  3.69}\makeunderscoresub\unmakecolonletter
\makeunderscoreletter\makecolonletter\def\synclem:loc_estimate_fnull{Lemma
  3.76}\makeunderscoresub\unmakecolonletter
\makeunderscoreletter\makecolonletter\def\syncthm:alb_derivation{Theorem
  3.11}\makeunderscoresub\unmakecolonletter
\makeunderscoreletter\makecolonletter\def\syncthm:taira{Theorem
  3.60}\makeunderscoresub\unmakecolonletter
\makeunderscoreletter\makecolonletter\def\synccor:mu_arb_cone{Corollary
  3.95}\makeunderscoresub\unmakecolonletter
\makeunderscoreletter\makecolonletter\def\syncbiLip_dis{Lemma
  2.59}\makeunderscoresub\unmakecolonletter
\makeunderscoreletter\makecolonletter\def\synccor:assouad_bound{Corollary
  4.6}\makeunderscoresub\unmakecolonletter
\makeunderscoreletter\makecolonletter\makeunderscoresub\unmakecolonletter
\makeunderscoreletter\makecolonletter\def\synclem:meas_fix{Lemma
  3.1}\makeunderscoresub\unmakecolonletter
\makeunderscoreletter\makecolonletter\catcode`*=11\def\syncthm:weak*density{Theorem
  3.97}\catcode`*=12\makeunderscoresub\unmakecolonletter
\makeunderscoreletter\makecolonletter\def\syncborel_rest{Lemma
  2.22}\makeunderscoresub\unmakecolonletter
\makeunderscoreletter\makecolonletter\def\synccompact_reduction{Theorem
  2.15}\makeunderscoresub\unmakecolonletter
\makeunderscoreletter\makecolonletter\def\synclem:vector_alberti{Lemma
  3.125}\makeunderscoresub\unmakecolonletter
\numberwithin{equation}{section}
\theoremstyle{plain} 
\newtheorem{lem}[equation]{Lemma}

\newtheorem{thm}[equation]{Theorem}
\newtheorem{cor}[equation]{Corollary}
\theoremstyle{definition}
\newtheorem{defn}[equation]{Definition}
\theoremstyle{remark}
\newtheorem{exa}[equation]{Example}
\newtheorem{rem}[equation]{Remark}
\begin{document}
\title{Metric Currents and Alberti representations}
\author{Andrea Schioppa}
\address{ETHZ}
\email{andrea.schioppa@math.ethz.ch}
\begin{abstract}
  We relate Ambrosio-Kirchheim metric currents to Alberti
  representations and Weaver derivations. In particular, given a
  metric current $T$, we show that if the module $\wder{\cmass T.}.$
  of Weaver derivations is finitely generated, then $T$ can be
  represented in terms of derivations; this extends previous results
  of Williams. Applications of this theory include an approximation of
  $1$-dimensional metric currents in terms of normal currents and the
  construction of Alberti representations in the directions of vector
  fields.
\end{abstract}
\maketitle
\tableofcontents 

\section{Introduction}
\label{sec:intro}
\subsection*{Overview}
\label{subsec:overview}
The goal of this paper is to relate metric currents to Alberti
representations and Weaver derivations. In particular, it seems that
metric currents carry a weak notion of a differentiable structure
which we try to describe by using Alberti representations and Weaver
derivations. As a first application we prove an approximation result
in which a $1$-dimensional metric current is approximated by a
sequence of normal currents. As a second application we show how to
use $1$-dimensional normal currents to produce Alberti representations
in the directions of vector fields.
\subsection*{Metric currents}
\label{subsec:met_currs}
\par Federer and Fleming \cite{federer_fleming_currents} introduced the
theory of currents to study the Plateau problem in Euclidean spaces of
dimension higher than $2$, and overtime currents have proven useful to
attack a wide range of problems, see
\cite{almgren_liquid_crystals,lin_gradient_harmonic,giaquinta_modica_dirichlet}
to cite some examples. In order
to study similar problems in general metric spaces, it became
desirable to have an analogue of the Federer-Fleming currents and a
major obstacle was that the classical definition of currents uses the
differentiable structure of $\real^N$. In \cite{ambrosio-kirch}
Ambrosio and Kirchheim, inspired by an idea of de Giorgi
\cite{de_giorgi_metric_currents}, developed a theory of metric currents starting
by circumventing the lack of a differentiable structure.
Essentially, $k$-dimensional metric currents are defined by duality
with $(k+1)$-tuples of Lipschitz functions $(f,\pi_1,\cdots,\pi_k)$,
where the first function $f$ is also bounded. The axioms that currents
satisfy are then designed so that one can formally treat, to some extent, the
$(k+1)$-tuple $(f,\pi_1,\cdots,\pi_k)$ as a \emph{$k$-dimensional differential form}
$fd\pi_1\wedge\cdots\wedge d\pi_k$. In \cite{williams-currents} Williams
showed that in a differentiability space $(X,\mu)$, those metric currents whose masses are absolutely continuous
with respect to $\mu$ are dual to the differential $k$-forms defined
using the differentiable structure. This result was the starting
point of the present work in which, roughly speaking, we remove the
assumption that $(X,\mu)$ is a differentiability space.
\par For a treatment of metric currents we refer the reader to
\cite{ambrosio-kirch}; some basic facts are recalled in Subsection
\ref{subsec:metcurr}. Note that Lang \cite{lang_local_currents} has formulated an alternative
theory of metric currents in which the finite mass axiom is removed;
our results have natural counterparts in that setting.
\subsection*{Alberti representations}
\label{subsec:alb_rep}
\par Alberti representations were introduced in
\cite{alberti_rank_one} to prove the rank-one property for BV
functions; they were later applied to study the differentiability
properties of Lipschitz functions $f:\real^N\to\real$
\cite{acp_plane,acp_proceedings} and have recently been used to obtain a
description of measures in differentiability spaces
\cite{bate-diff}. We give here an informal definition and refer the
reader to \cite{bate-diff,deralb} and Subsection \ref{subsec:alberti} for further details.
\par An \textbf{Alberti representation} of a Radon measure $\mu$ is a
generalized Lebesgue decomposition of $\mu$ in terms of rectifiable
measures supported on path fragments; a \textbf{path fragment} in $X$
is a Lipschitz map $\gamma:K\to X$ where $K\subset\real$ is compact;
the set of fragments in $X$ will be denoted by $\frags(X)$ and
topologized as a subspace of $K(X)$, the set of compact subsets of $X$
with the {topology induced by the Hausdorff metric. An Alberti
representation of $\mu$ is then a decomposition:
\begin{equation}
\mu=\int_{\frags(X)}\nu_\gamma\,dP(\gamma),
\end{equation}
where $P$ is a regular Borel probability measure on $\frags(X)$, and
$\nu$ associates to each fragment $\gamma$ a finite Radon measure
$\nu_\gamma$ which is absolutely continuous with respect to the
$1$-dimensional Hausdorff measure $\hmeas 1._\gamma$ on the image of
$\gamma$. Examples of an Alberti representation are offered by
Fubini's Theorem; however, in general it is necessary to work with
path fragments instead of Lipschitz curves because the space $X$ on
which $\mu$ is defined might lack any rectifiable curve.
\subsection*{Weaver derivations and their relationship with Alberti representations}
\label{subsec:weav_der}
\par Weaver derivations, hereafter simply called derivations, were
introduced in \cite{weaver00} and provide a quite broad framework to
formulate a notion of differentiability on metric measure spaces. To
fix the ideas, let $\lipfun X.$ denote the set of real-valued
Lipschitz functions defined on $X$ and let $\lipalg X.\subset\lipfun
X.$ denote the subset of bounded Lipschitz functions. The vector space
$\lipalg X.$ becomes a Banach algebra with norm:
\begin{equation}
  \label{eq:banlip}
  \|f\|_{\lipalg X.}=\max(\|f\|_\infty,\glip f.),
\end{equation}
where $\glip f.$ denotes the Lipschitz constant of $f$. It is a fact
\cite[Ch.~2]{weaver_book99} that the Banach algebra $\lipalg X.$ is a dual
Banach space and so it has a weak* topology; for the present work, it
is sufficient to consider sequential convergence which is
characterized as follows: $f_n\xrightarrow{\text{w*}}f$ if and only if
the global Lipschitz constants of the $f_n$ are uniformly bounded and
$f_n\to f$ pointwise.
\par Having fixed a Radon measure $\mu$ on $X$, derivations are weak*
continuous bounded linear maps $D:\lipalg X.\to L^\infty(\mu)$ which
satisfy the product rule $D(fg)=fDg+gDf$. Intuitively, derivations can
be interpreted as \emph{measurable vector fields} and depend only on the
measure class of $\mu$. For example, if $\nlebmeas n.$ denotes the
Lebesgue measure on $\real^n$, one obtains a derivation
$\frac{\partial}{\partial x_i}:\lipalg X.\to L^\infty(\nlebmeas n.)$
by taking the partial derivatives of Lipschitz functions in the
$x_i$-direction. Note that the set of derivations is an $L^\infty(\mu)$-module.
\par Even for metric measure spaces $(X,\mu)$ which cannot admit a
differentiable structure the module $\wder\mu.$ can be
nontrivial. Moreover, one can also study the modules $\wder\mu.$ and
$\wder\mu'.$ for mutually singular measures $\mu$ and $\mu'$ on the
\emph{same space} $X$. Derivations provide thus a broad definition of
differentiability for Lipschitz functions and it is desirable to
obtain a characterization of derivations for general metric measure
spaces. In \cite{deralb} the author showed that there is a
correspondence between Alberti representations and Weaver derivations
which implies, roughly speaking, that derivations are obtained by
taking derivatives along fragments. Some results in \cite{deralb}
relevant for the present work are recalled in Subsection \ref{subsec:corresp}.
\subsection*{Main results}
\label{subsec:main_res}
\par We now describe the main results of this paper and refer the
reader to the following sections for an explanation of the
terminology; we denote by $\mcurr k,X.$ the Banach space of
$k$-dimensional metric currents in the metric space $X$.
\par It is an observation\footnote{Gong \cite[pg.~3]{gong_plane}
  attributes it to Wenger} that there is a close similarity between
Weaver derivations and $1$-dimensional metric currents (see
Sec.~\ref{sec:onedimcurr}). In the light of \cite{deralb}
it is thus natural to ask how this similarity relates to the existence
of Alberti representations. We show that the mass $\cmass T.$ of a
$k$-dimensional metric current $T$ posseses Alberti representations in
the directions of $k$-dimensional cone fields. Specifically, in Section
\ref{sec:currents_and_alberti} we prove the following:
\begin{thm}\label{thm:currents_arbitrary_cones}
  Let $X$ be a complete separable metric space and let $T\in\mcurr
  k,X.\setminus\{0\}$ for $k>0$. Then there are
  disjoint Borel sets $\{V_j\}_j$ and $1$-Lipschitz functions
  $\pi^j:X\to\real^k$ (on $\real^k$ we consider the $l^\infty$ norm) such that:
  \begin{enumerate}
  \item $\cmass T.\left(X\setminus\bigcup_jV_j\right)=0$.
  \item For all $\epsi>0$ and for any $k$-dimensional cone field $\cone$, the measure $\cmass T.$
    admits a $(1,1+\epsi)$-biLipschitz Alberti representation
    $\albrep.$ with $\albrep.\mrest V_j$ in the $\pi^j$-direction of
    $\cone$.
  \end{enumerate}
  \par In particular, the module $\wder{\cmass T.}.$ contains
  $k$ independent derivations.
\end{thm}
\par Note that the proof of Theorem \ref{thm:currents_arbitrary_cones}
actually does not take full advantage of the \emph{joint continuity} of
$T$ in its last arguments $(\pi_1,\cdots,\pi_k)$ and so applies to a
larger class of metric functionals. It might be worth mentioning a
connection between Theorem \ref{thm:currents_arbitrary_cones} and the
classical Rademacher Theorem, which asserts that a Lipschitz function
$f:\real^n\to\real$ is differentiable at $\hmeas n.$-a.e.~point, where
$\hmeas n.$ denotes the Lebesgue measure. Given a top dimensional
current $T\in\mcurr n,\real^n.$, Theorem
\ref{thm:currents_arbitrary_cones} implies that  the mass measure $\cmass T.$ 
posseses $n$-independent Alberti representations, and then it follows
that the conclusion of Rademacher's Theorem holds for the measure
$\cmass T.$. A detailed argument which uses normal currents can be
found in the recent work of Alberti and Marchese~\cite{alberti_marchese}. However, we provide the
sketch of two alternative arguments.
First of all, having fixed a real-valued
  Lipschitz function $f$, one can use
  the $n$-independent Alberti representations to show that at $\cmass
  T.$-a.e.~point $p$ the function $f$ has partial derivatives in $n$-independent
  directions $\{e_i(p)\}_{i=1}^n$. From this one can proceed in two
  different ways.
The first uses a porosity argument like
  in~\cite[Sec.~9]{bate-diff} by showing that the partial
  derivatives constructed above give a linearization of $f$ at $p$. An alternative 
  geometric argument uses the fact that, if $\cmass T.$
  has $n$-independent Alberti representations, then at a generic point $p$ one
  can follow the fragments in $n$-independent directions to get close to any point in
  $B(p,r)$ like in \cite[Subsec.~5.2]{deralb}. Specifically, for any $q\in B(p,r)$ one can follow
  $n$-fragments $\gamma_1,\cdots,\gamma_n$ such that $\gamma_1$ starts
  at $p$, $\gamma_{i+1}$ starts at the end point of $\gamma_i$, the
  end point of $\gamma_n$ is $q'$ with $d(q,q')=o(r)$, along
  $\gamma_i$ the unit tangent vector is at distance $O(r)$ from
  $e_i(p)$, and the total
  length of the fragments $\gamma_i$'s is $\le Cd(p,q)$. Moreover, one can
  assume that for $r$ sufficiently small the fragments $\gamma_i$ are almost
  paths, i.e.~that there are paths $\tilde\gamma_i$ which extend
  $\gamma_i$, and such that the domain of $\gamma_i$ has Lebesgue
  measure at least $(1-O(r))$ times that of the domain of
  $\gamma_i$. Choosing $p$ to be also an approximate continuity point
  of the partial derivatives of $f$ in the directions given by the
  vector fields $p\mapsto e_i(p)$, 
  one concludes that $f$ is differentiable at $p$.
 \par In 2011 M.~Cs\"ornyei and P.~Jones have announced very deep and
 very difficult results in
 Geometric Measure Theory and Harmonic Analysis which imply that 
 Rademacher's Theorem is \emph{sharp} in the sense that, if
its \emph{conclusion} holds for the metric measure space
$(\real^n,\mu)$, then $\mu$ must be absolutely continuous with respect
to the Lebesgue measure. One can then conclude that an
$n$-dimensional metric current $T$ in $\real^n$ must have $\cmass
T.\ll{\mathcal L}^n$. Since the first version of this preprint
appeared in April 2014, G.~De Philippis and F.~Rindler~\cite{rindler-afree} have provided a nice
and elegant proof of the sharpeness of Rademacher's Theorem that follows from
remarkable and deep results on the structure of ${\mathscr A}$-free measures.
\par Note also that Theorem \ref{thm:currents_arbitrary_cones}
suggests that metric currents come with some weak notion of a
\emph{differentiable structure}. To make this intuition precise, we prove a
representation formula for metric currents in terms of Weaver
derivations; essentially, a $k$-dimensional metric current $T$ is of
the form $\omega_T\,\cmass T.$, where $\omega_T$ is a measurable
$k$-dimensional vector field (see the next Subsection) and the formal
$k$-form $(f,\pi_1,\cdots,\pi_k)$ can be interpreted as a $k$-form in
the $k$-th exterior power of the Weaver's cotangent bundle (see also
the next
Subsection). Specifically, in Section \ref{sec:rep_formula} we prove:
\begin{thm}\label{thm:gl_curr_rep}
  Let $T\in\mcurr k,X.$ and assume that $\wder{\cmass T.}.$ is
  finitely generated with $N$ generators. Then there is $\omega_T\in\hwder k,{\cmass T.}.$
  such that:
  \begin{equation}
    \label{eq:gl_curr_rep_s1}
    T(f,\pi_1,\cdots,\pi_k)=\int_Xf\langle\omega_T,d\pi_1\wedge\cdots\wedge d\pi_k\rangle\,d{\cmass T.}.
  \end{equation} 
  Alternatively, one might take $\omega_T$ to be an element of $\lmodext k,{\cmass T.},{\wder{\cmass T.}.}.$ or
  $\banext k,{\wder{\cmass T.}.}.$, see
  Subsection~\ref{subsec:ext_prod} for different definitions of
  exterior products.
  \par Moreover, $\omega_T$ has norm at most $(C(N))^k\binom{N}{k}$.
\end{thm}
\par Note that the assumption that $\wder\mu.$ is finitely generated
is not very restrictive as it holds if the restriction of $\cmass T.$
to its support is doubling or if the support of $\cmass T.$ is
doubling \cite{deralb}. Note also how Theorem \ref{thm:gl_curr_rep}
parallels the representation of classical currents
(\cite[Sec.~7.2]{krantz_geometric_integration},
\cite[Sec.~4.1]{federer_gmt}).
\par In Section \ref{sec:applications} we provide two applications of
this theory. The first application provides an approximation of
$1$-dimensional metric currents in terms of normal currents:
\begin{thm}\label{thm:amb_kirch_one}
  If $T\in\mcurr 1,Z.$ where $Z$ is a Banach space and if the module $\wder{\cmass
    T.}.$ is finitely generated, then there is a sequence of normal
  currents $\{N_n\}\subset\ncurr 1,Z.$ such that:
  \begin{equation}
    \label{eq:amb_kirch_one_s1}
    \lim_{n\to\infty}\|T-N_n\|_{\mcurr 1,Z.}=0.
  \end{equation}
\end{thm}
This provides an affirmative answer to the $1$-dimensional case of a
question raised in \cite[pg.~68]{ambrosio-kirch}. The question of
Ambrosio and Kirchheim is whether their metric currents coincide, in
$\real^N$, with the Federer-Fleming flat chains of finite mass. We
answer this question affirmatively for $1$-dimensional current, but
our result is more general. In fact,
even though we prove the result in Banach spaces, the proof can be
adapted to spaces where fragments can be \emph{filled-in} to give
Lipschitz curves. In particular, the structure of $1$-dimensional
metric currents seems very close to that of normal currents. Note that
this is not the case for classical currents. We also
  mention that since the preprint of this work appeared in April 2014,
  recently also $N$-dimensional metric currents in $\real^N$ have
  been shown to be Federer-Fleming flat chains. This follows from the
  recent beautiful work of G.~De Philippis and F.~Rindler~\cite{rindler-afree}
  combined with Theorem~\ref{thm:currents_arbitrary_cones}, see
  also the previous discussion on the sharpeness
  of Rademacher's Theorem.
\par As a second application we provide a different method to produce
Alberti representations which is based on results of Paolini and
Stepanov \cite{paolini_acyclic,paolini_one_normal} on the structure of
$1$-dimensional normal currents. This approach allows to gain a better
control on the direction of the Alberti representations; in fact,
instead of obtaining Alberti representations in the $\psi$-direction
of a finite dimensional cone field $\cone$, one obtains Alberti
representations in ther $\psi$-direction of a \emph{vector field}
$v$. Moreover, the Lipschitz function $\psi$ can be taken to be
$l^2$-valued, allowing to control countably many functions. The
precise result is Theorem \ref{thm:const_dir}, which is proved in
Subsection \ref{subsec:const_dir}. This result is based on identifying
a special class of derivations, which we call \textbf{normal
  derivations}, which have properties closely related to those of
normal currents. A further direction related to this result is to
extend the action of derivations to Lipschitz functions which take
values in Banach spaces with the Radon-Nikodym property: this will be
pursued elsewhere.
\subsection*{Technical tools}
\label{subsec:tools}
Section \ref{sec:technical} contains some technical results. In
Subsection \ref{subsec:ext_prod} we discuss exterior powers in the
categories of Banach spaces, $L^\infty(\mu)$-modules and
$L^\infty(\mu)$-normed modules. This material is just an adaptation of
the treatment in \cite[Ch.~2 and 3]{cigler_banach_modules} of tensor
products. The motivation is to give a precise meaning to an exterior
product of derivations $D_1\wedge\cdots\wedge D_k$; as $\wder\mu.$ is
an $L^\infty(\mu)$-normed module, the construction can be done in the
three aforementioned categories and the results are different. In the
author's opinion, the most natural choice is probably that of
$L^\infty(\mu)$-normed modules.
\par In Subsection \ref{subsec:alberti_banach} we prove Theorem
\ref{thm:alberti_banach} which is a criterion to produce Alberti
representations for measures in Banach spaces when the direction and
the speed are specified by linear maps. This result is used in the
proof of Theorem \ref{thm:amb_kirch_one}.
\par In Subsection \ref{subsec:renorming} we discuss Theorem
\ref{thm:renorm}, which is a renorming trick which allows to obtain a
strictly convex local norm on $\wder\mu.$ by taking a biLipschitz
deformation of the metric on the ambient metric space. This result is
used in the proof of Theorem \ref{thm:const_dir} and might be of
independent interest. It is worth to point out that Theorem
\ref{thm:renorm}, when specialized to the context of differentiability
spaces, gives a stronger conclusion than Cheeger's renorming Theorem
\cite[Sec.~12]{cheeger99} for PI-spaces. In fact, Theorem
\ref{thm:renorm} works in general differentiability spaces, does
require only a small perturbation of the distance function, and works
globally (while Cheeger's argument works only on a single chart).
\subsection*{Acknowledgements} The author thanks J.~Cheeger for
comments on the renorming trick, B.~Kleiner for
raising the question of whether it is possible to obtain Alberti
representations in the directions of vector fields, and U.~Lang for
pointing out Z\"ust's bound \cite{zust_phd}.
\par Finally, I want to express my gratitude to the referee for
carefully reading the manuscript at suggesting a simplification for
the proof of Lemma~\ref{lem:loc_curr_rep}.
\par The mathematical content of the paper originated when the author
was a PhD student at NYU. During the revision phase of the paper the
author was supported by the 
``ETH Zurich Postdoctoral Fellowship Program and the Marie Curie Actions
 for People COFUND Program''.
\section{Preliminaries}
\label{sec:prelim}

\subsection{Metric currents}
\label{subsec:metcurr}
We recall here some definitions and facts about metric currents and
refer the reader to \cite{ambrosio-kirch,lang_local_currents} for more
information. 
\par Let $\metuple k,X.$ denote the set of $\lipalg X.\times\left(\lipfun
    X.\right)^k$\footnote{for $k=0$ we let $\metuple 0,X.=\lipalg X.$} of $(k+1)$-tuples of Lipschitz
  functions where the first one is bounded.
Intuitively, we want to think of a $(k+1)$-tuple
$(f,\pi_1,\cdots,\pi_k)$ as a $k$-differential form
$fd\pi_1\wedge\cdots\wedge d\pi_k$. A map $T:V\to\real$, where $V$ is
a vector space over $\real$ is called \textbf{subadditive} if for each
$v_1,v_2\in V$ one has:
\begin{equation}
  \label{eq:subadd}
  \left|T(v_1+v_2)\right|\le\left|T(v_1)\right|+\left|T(v_2)\right|;
\end{equation}
the map $T$ is called \textbf{positively $1$-homogeneous} if for all $(v,\lambda)\in
V\times[0,\infty)$ one has:
\begin{equation}
  \label{eq:1hom}
  \left|T(\lambda v)\right|=\lambda\left|T(v)\right|.
\end{equation}
\begin{defn}
  \label{defn:metric_functional}
  A \textbf{$k$-dimensional metric functional} $T$ on the metric space $X$ is a
  map $T:\metuple k,X.\to\real$ which is subadditive and positively
  $1$-homogeneous in each of its arguments
  $(f,\pi_1,\cdots,\pi_k)$. The \textbf{boundary} $\partial T$ of a
  $k$-dimensional metric functional ($k\ge1$) is the $(k-1)$-dimensional metric
  functional defined by:
  \begin{equation}
    \label{eq:boundary}
    \partial T(f,\pi_1,\cdots,\pi_{k-1})=T(1,f,\pi_1,\cdots,\pi_{k-1}).
  \end{equation}
  For $0$-dimensional metric functionals we convene that the boundary
  is $0$.
\end{defn}
\begin{defn}
  \label{defn:mass}
  A \textbf{$k$-dimensional metric functional} $T$ has finite mass if there is
  a finite Radon measure $\mu$ such that for each
  $(f,\pi_1,\cdots,\pi_{k})\in\metuple k,X.$:
  \begin{equation}
    \label{eq:massdef}
    \left|T(f,\pi_{1},\cdots,\pi_{k})\right|\le\prod_{i=1}^k\glip\pi_i.\,\int_X|f|\,d\mu.
  \end{equation}
  In this case there is a \emph{minimal} $\mu$ satisfying \eqref{eq:massdef}, called the \textbf{mass} of $T$ and denoted by
$\cmass T.$.
\end{defn}
\begin{rem}
  Note that any metric functional $T$ with finite mass can be uniquely
  extended to a map $T:\bborel X.\times\left(\lipfun X.\right)^k$ so
  that the first argument $f$ can be taken to be a bounded Borel function.
\end{rem}
\begin{defn}
  \label{defn:restriction}
  Let $T$ be a $k$-dimensional metric functional with finite
  mass. Suppose that $l\le k$ and that
  \begin{equation}
    \label{eq:restriction1}
    \omega=(\psi,\pi_1,\cdots,\pi_l)\in\bborel X.\times\left(\lipfun X.\right)^l;
  \end{equation}
  the \textbf{restriction} $T\crest\omega$ is the $(k-l)$-dimensional metric
  functional defined by:
  \begin{equation}
    \label{eq:restriction2}
    T\crest\omega(f,\tilde\pi_1,\cdots,\tilde\pi_{k-l})=T(f\psi,\pi_1,\cdots,\pi_l,\tilde\pi_1,\cdots,\tilde\pi_{k-l}).
  \end{equation}
\end{defn}
\par In the Introduction we recalled the notion of weak* convergence
for sequences in $\lipalg X.$. We now introduce a notion of
convergence for sequences in $\lipfun X.$ which plays a fundamental
r\^ole in the definition of metric currents: if $\{f_n\}\subset\lipfun
X.$ and $f\in\lipfun X.$, we write $f_n\xrightarrow{\text{w*}}f$ if
$f_n\to f$ pointwise and $\sup_n\glip f_n.<\infty$.
\begin{defn}
  \label{defn:metric_currents}
  A $k$-dimensional metric functional $T$ of finite mass is called a
  \textbf{metric current} if it satisfies the following additional
  properties\footnote{in this formulation some axioms are redundant, see
  \cite[Sec.~3]{ambrosio-kirch}.}:
  \begin{enumerate}
  \item $T$ is multilinear in its arguments $f,\pi_1,\cdots,\pi_k$;
  \item $T$ is alternating in its last $k$-arguments $\pi_1,\cdots,\pi_k$;
  \item $T$ is \textbf{local} in the sense that if some $\pi_i$ is
    constant on the set $\{x: f(x)\ne0\}$, then
    \begin{equation}
      \label{eq:currloc}
      T(f,\pi_{1},\cdots,\pi_{k})=0;
    \end{equation}
  \item if one has that $f_n\xrightarrow{\text{w*}}f$ and for $i\in\{1,\cdots,k\}$
    one also has $\pi_{i,n}\xrightarrow{\text{w*}}\pi_i$, then, under
    the assumption that $\sup_n\|f_n\|_\infty,\|f\|_\infty<\infty$, it
    follows that:
    \begin{equation}
      \label{eq:currcont}
      \lim_{n\to\infty}T(f_n,\pi_{i,1},\cdots,\pi_{i,k})=T(f,\pi_{1},\cdots,\pi_{k}).
    \end{equation}
  \end{enumerate}
\end{defn}
\par The set of $k$-dimensional metric currents is denoted by $\mcurr
k,X.$ and is a Banach space with norm $\|T\|_{\mcurr k,X.}=\cmass
T.(X)$. An important class of metric currents consists of the normal
currents:
\begin{defn}
  \label{defn:normcurrs}
  A $k$-dimensional metric current is a \textbf{normal current} if the
  boundary $\partial T$ is a metric current. The set of
  $k$-dimensional normal currents is denoted by $\ncurr k,X.$ and is a
  Banach space with norm:
  \begin{equation}
    \label{eq:normnorm}
    \|T\|_{\ncurr k,X.}=\cmass T.(X)+\cmass\partial T.(X).
  \end{equation}
\end{defn}
\subsection{Alberti representations}
\label{subsec:alberti}
\par In this Subsection we recall some facts about Alberti
representations. We next give the definition of Alberti representation
after elaborating on the definition of fragment given in the introduction.
\begin{defn}\label{defn:Alberti_rep}
  A fragment in $X$ is a Lipschitz map $\gamma:K\to X$ where
  $\dom\gamma=K$ is a nonempty compact subset of $\real$.
  We denote the set of fragments by $\frags(X)$ and
  topologize it with the Hausdorff distance between their graphs:
  $d(\gamma_1,\gamma_2)$ is the infimum of those $\varepsilon>0$ such
  that for each $(i,j)\in\{(1,2), (2,1)\}$ one has that for each 
  $t_i\in\dom\gamma_i$ there is a $t_j\in\dom\gamma_j$ with
  $d(\gamma_i(t_i),\gamma_j(t_j))\le\varepsilon$ and $|t_i-t_j|\le\varepsilon$.
  \par  Let $\mu$ be a Radon measure on a metric space $X$ and $M(X)$ denote
  the set of finite Radon measures on $X$; an Alberti
  representation of $\mu$ is a pair $(P,\nu)$:
  \begin{enumerate}
  \item The measure $P$ is a regular Borel probability measure on $\frags(X)$;
  \item The map $\nu:\frags(X)\to M(X)$ is Borel \footnote{on $M(X)$
one takes the weak* topology} and
    $\nu_\gamma\ll\hmeas 1._\gamma$, where $\hmeas 1._\gamma$ denotes
    the $1$-dimensional Hausdorff measure on the image of $\gamma$;
  \item  The measure $\mu$ can be represented as
    $\mu=\int_{\frags(X)}\nu_\gamma\,dP(\gamma)$;
  \item For each Borel set $A\subset X$ and for all real numbers $a\le b$,
    the map
    $\gamma\mapsto\nu_\gamma\left(A\cap\gamma(\dom\gamma\cap[a,b])\right)$
    is Borel.
  \end{enumerate}
Finally  to deal with the operation of restriction, one is led to introduce the \textbf{restriction of
  $\albrep.=(P,\nu)$ to a Borel set $U$}: $\albrep.\mrest
U=(P,\nu\mrest U)$ \cite[Lem.~2.4]{bate-diff}. Note that  $\albrep.\mrest
U$ yields an Alberti representation of $\mu\mrest U$.
\end{defn}
\begin{rem}
  \label{rem:new_frags}
  Note that in this paper the definition of fragments is different
  from that used in \cite{deralb} because, for a frament
  $\gamma:K\to X$, we do not require $\gamma$ to be biLipschitz or
  $\dom\gamma$ to have positive Lebesgue measure. However, an
  application of the area formula \cite[Cor.~8]{kirchheim_metric_diff} shows
  that the results that we cite from \cite{deralb} are still
  valid in this setting. For the reader's convenience we provide more
  details here.
  \par First, note that if $\dom\gamma$ has $0$ Lebesgue measure, then
  the image of $\gamma$ has $1$-dimensional Hausdorff measure equal to
  $0$ and so $\nu_\gamma=0$ by Axiom (2) in the definition of an
  Alberti representation. Thus, we can just consider fragments where
  $K=\dom\gamma$ has positive Lebesgue measure. Now we can partition
  $K=K_{-1}\cup K_0 \cup K_{1}$ where (a) $\lebmeas(K_{-1})=0$, (b)
  the metric differential (see Definition~\ref{def:met_diff})
  $\metdiff \gamma$ exists and is approximately continuous on $K_{0}$
  and $K_{1}$, and (c) $\metdiff\gamma = 0$ on $K_{0}$ and
  $\metdiff\gamma>0$ on $K_1$. Then $\hmeas 1._\gamma(\gamma(K_0))=0$
  by the area formula and by metric differentiation one can find a
  countable partition $K_{1}=\bigcup S_\alpha$ such that
  $\gamma|S_\alpha$ is biLipschitz onto $\gamma(S_\alpha)$. In this
  way the part of $\gamma$ that contributes to the Alberti
  representations can be represented as a countable union of
  biLipschitz fragments.
\end{rem}
In order to define notions of speed and direction for Alberti
representations we recall the definitions of Euclidean cone and of the
metric differential of a fragment.
\begin{defn}\label{defn:cone}
  Let $\alpha\in(0,\pi/2)$, $w\in{\mathbb S}^{n-1}$; the \textbf{open
    cone} $\cone(w,\alpha)\subset\real^n$ with axis $w$ and opening
  angle $\alpha$ is:
\begin{equation}
    \cone(w,\alpha)=\{u\in\real^q: \tan\alpha\langle
    w,u\rangle>\|\pi_w^\perp u\|_2\},
  \end{equation} where $\pi_w^\perp$ denotes the orthogonal projection
  on the orthogonal complement of the line $\real w$.
\end{defn}
\begin{defn}\label{def:met_diff}
  For a fragment $\gamma\in\frags(X)$, the metric
  differential $\metdiff\gamma(t)$ of $\gamma$ at $t\in\dom\gamma$ is the limit
  \begin{equation}
    \lim_{\dom\gamma\ni t'\to t}\frac{\dist\gamma(t'),\gamma(t).}{|t'-t|}
  \end{equation}
  whenever it exists; if  $t$ is an isolated
    point of $\dom\gamma$ we convene that the limit is $0$.
\end{defn}
\par In order to measure the direction of a fragment $\gamma$, one uses a 
 Lipschitz function $f:X\to\real^q$ and studies the direction of
 $(f\circ\gamma)'$ using cones.
 \begin{defn}
   \label{defn:alberti_direction}
   An $n$-dimensional \textbf{cone field} $\cone$ is a Borel map from
   $X$ to the set of open cones in $\real^n$. Alternatively, an $n$-dimensional
   cone-field $\cone$ is specified by a pair of Borel maps
   $\alpha:X\to(0,\pi/2)$ and $w:X\to{\mathbb S}^{n-1}$ by letting
   $\cone(x)=\cone(\alpha(x),w(x))$.
   \par Given a Lipschitz function $f:X\to\real^n$, an Alberti
   representation $\albrep.=(P,\nu)$ is said to be \textbf{in the
     $f$-direction of the $n$-dimensional cone-field $\cone$} if for
   $P$-a.e.~$\gamma\in\frags(X)$ and
   $\lebmeas\mrest\dom\gamma$-a.e.~$t$ one has $(f\circ\gamma)'(t)\in\cone(\gamma(t))$.
 \end{defn}
 \begin{defn}
   \label{defn:alberti_speed}
   Let $\sigma:X\to[0,\infty)$ be Borel and $f:X\to\real$ be
   Lipschitz. An Alberti representation $\albrep.=(P,\nu)$  is said to
   be \textbf{have $f$-speed $\ge\sigma$ (resp.~$>\sigma$)} if for $P$-a.e.~$\gamma\in\frags(X)$ and
   $\lebmeas\mrest\dom\gamma$-a.e.~$t$ one has
   $(f\circ\gamma)'(t)\ge\sigma(\gamma(t))\metdiff\gamma(t)$ (resp.~$(f\circ\gamma)'(t)>\sigma(\gamma(t))\metdiff\gamma(t)$).
 \end{defn}
 \par One finally needs also to control the Lipschitz constant of the
 fragments used to produce Alberti representations.
 \begin{defn}
   \label{defn:alberti_constants}
   An Alberti representation $\albrep.=(P,\nu)$ is said to be
   \textbf{$C$-Lipschitz (resp.~$(C,D)$-biLipschitz)} if
   $P$-a.e.~$\gamma$ is $C$-Lispchitz (resp.~$(C,D)$-biLipschitz).
 \end{defn}
\par Alberti representations are produced using Rainwater's Lemma
\cite{rainwater-note}, which can be regarded as a generalization of the
 Radon-Nikodym Theorem. In particular, one studies a
 notion of \emph{nullity for sets} with respects to a \emph{family of measures}.
 
 \begin{defn}
   \label{defn:frags_null}
   Let $S\subset X$ and $\Omega\subset\frags(X)$. The set $S$ is said
   to be \textbf{$\Omega$}-null if for each $\gamma\in\Omega$ one has $\hmeas 1._\gamma(S)=0$.
 \end{defn}
 We will use the previous notion of nullity mainly for the following
 families of fragments:
 \begin{defn}
   \label{defn:frags_sets}
   Let $f:X\to\real^n$ and $g:X\to\real$ be Lipschitz functions,
   $\sigma:X\to[0,\infty)$ a Borel function and $\cone$ an
   $n$-dimensional cone field. We denote by
   $\frags(X,f,\cone,g,>\sigma)$ the set of those $\gamma\in\frags(X)$
   satisfying:
   \begin{align}
     \label{frags_sets1}
     (f\circ\gamma)'(t)&\in\cone(\gamma(t))\quad\text{for
       $\lebmeas\mrest\dom\gamma$-a.e~$t$;}\\
     \label{frags_sets2}
     (g\circ\gamma)'(t)&>\sigma(\gamma(t))\metdiff\gamma(t)\quad\text{for
       $\lebmeas\mrest\dom\gamma$-a.e~$t$;}
   \end{align}
   the set $\frags(X,f,\cone,g,\ge\sigma)$ is defined by changing the
   strict inequality in \eqref{frags_sets2} to a non-strict inequality.
 \end{defn}
 The following Theorem (\sync alberti_rep_prod.\ in \cite{deralb}) is a
 standard criterion to produce Alberti representations:
 \begin{thm}
   \label{thm:alberti_rep_prod}
   Let $X$ be a complete separable metric space and $\mu$ a Radon
   measure on $X$. Then the following are equivalent:
   \begin{enumerate}
   \item The measure $\mu$ admits an Alberti representation in the
     $f$-direction of $\cone$ with $g$-speed $>\sigma$;
   \item For each $\varepsilon>0$ the measure $\mu$ admits a
     $(1,1+\varepsilon)$-biLipschitz Alberti representation in the
     $f$-direction of $\cone$ with $g$-speed $>\sigma$;
   \item Any Borel set $S\subset X$ which is
     $\frags(X,f,\cone,g,>\sigma)$-null is also $\mu$-null.
   \end{enumerate}
 \end{thm}
\par In the following we will also use a gluing principle for Alberti
 representations (compare \sync alb_glue.\ in \cite{deralb}). 
 \begin{defn}\label{defn:Linfty_partitions}
  A countable collection $\{U_\alpha\}$ of $\mu$-measurable and
  pairwise disjoint sets with
  positive $\mu$-measure is called an \textbf{$L^\infty(\mu)$-partition of
  unity} if\\ $\mu\bigl(\bigl(\bigcup_\alpha
      U_\alpha\bigr)^{c}\bigr)=0$; note that in this case
  \begin{equation}
    \sum_{\alpha}\chi_{U_\alpha}=1
  \end{equation} where convergence of the series is understood in the
  weak* sense. If the sets $U_\alpha$ are Borel (resp.~compact) the
  $L^\infty(\mu)$-partition of unity is called \textbf{Borel (resp.~compact)}.
\end{defn}
 \begin{thm}
   \label{thm:alb_glue}
   Let $X$ be a complete separable metric space and $\mu$ a Radon
   measure on $X$ and $\{U_\alpha\}$ a Borel $L^\infty(\mu)$-partition
   of unity. If for each $\alpha$ the measure $\mu\mrest U_\alpha$
   admits an Alberti representation in the $f_\alpha$-direction of an
   $N_\alpha$-dimensional cone field $\cone_\alpha$ with $g_\alpha$-speed $\ge\sigma_\alpha$, then $\mu$ admits an Alberti
   representation $\albrep.$ such that each restriction $\albrep.\mrest U_\alpha$
   is  in the $f_\alpha$-direction of an
   $N_\alpha$-dimensional cone field $\cone_\alpha$ with $g_\alpha$-speed $\ge\sigma_\alpha$. Moreover, for each $\varepsilon>0$ the
   Alberti representation $\albrep.$ can be assumed to be $(1,1+\varepsilon)$-biLipschitz.
 \end{thm}

 \subsection{Derivations}
 \label{subsubsec:derivations}
\par An \textbf{$L^\infty(\mu)$-module} $M$ is a Banach space $M$ which
is also an $L^\infty(\mu)$-module and such that for all
$(m,\lambda)\in M\times L^\infty(\mu)$ one has:
\begin{equation}
  \label{eq:boundedaction}
  \|\lambda m\|_M\le\|\lambda\|_{L^\infty(\mu)}\,\|m\|_M.
\end{equation}
Among $L^\infty(\mu)$-modules a special r\^ole is played by
\textbf{$L^\infty(\mu)$-normed modules}:
\begin{defn}\label{defn:local_norm}
  An $L^\infty(\mu)$-module $M$ is said to be an
  \textbf{$L^\infty(\mu)$-normed module} if there is a map
  \begin{equation}
    |\cdot|_{M,{\text{loc}}}:M\to L^\infty(\mu) 
  \end{equation}
such that:
\begin{enumerate}
\item For each $m\in M$ one has $|m|_{M,{\text{loc}}}\ge0$;
\item For all
  $c_1,c_2\in\real$ and $m_1,m_2\in M$ one has:
  \begin{equation}
        |c_1m_1+c_2m_2|_{M,{\text{loc}}}\le|c_1||m_1|_{M,{\text{loc}}}+|c_2||m_2|_{M,{\text{loc}}};
  \end{equation}
\item For each $\lambda\in L^\infty(\mu)$ and each $m\in M$, one has:
  \begin{equation}
    |\lambda m|_{M,{\text{loc}}}=|\lambda|\,|m|_{M,{\text{loc}}};
  \end{equation}
\item The local seminorm $|\cdot|_{M,{\text{loc}}}$ can be used to
  reconstruct the norm of any $m\in M$:
  \begin{equation}
    \|m\|_M=\|\,|m|_{M,{\text{loc}}}\,\|_{L^\infty(\mu)}.
  \end{equation}
\end{enumerate}
\end{defn}
\par Let $\mu$ be a Radon measure on the metric space $X$ and denote
by $\mcurr k,\mu.$ the set of $k$-dimensional metric currents whose
mass in absolutely continuous with respect to $\mu$.
\begin{lem}\label{lem:mcurr_module}
  The set $\mcurr k,\mu.$ is a Banach space  and an $L^\infty(\mu)$-module. It is
  not an $L^\infty(\mu)$-normed module if
  \begin{enumerate}
  \item $k>0$ and $\mcurr k,\mu.\ne\{0\}$;
  \item $k=0$ and $\mu$ is not a Dirac measure.
  \end{enumerate}
\end{lem}
\begin{proof}
  The space $\mcurr k,X.$ is a Banach space with the mass
  norm. Suppose that
  \begin{equation}\lim_{k\to\infty}\cmass T_k-T.(X)=0,
  \end{equation} and that
 for each $k$ one has $\cmass
  T_k.(A)=0$; then one has $\cmass T.(A)=0$. Thus, $\mcurr k,\mu.$ is a closed
  subspace of $\mcurr k,X.$ and hence a Banach space.
  \par The action of $L^\infty(\mu)$ on $\mcurr
  k,\mu.$ is given by
  \begin{equation}
    \lambda.T=T\crest\lambda,
  \end{equation}
  and $\cmass T\crest\lambda.(X)\le\|\lambda\|_{L^\infty(\mu)}\,\cmass
  T.(X)$; thus $\mcurr k,\mu.$ is an $L^\infty(\mu)$-module.
  \par Let $\delta_x$ denote the Dirac measure concentrated at
  $x$. Using \cite[(iii) in Thm.~3.5]{ambrosio-kirch} it follows that
  $\mcurr k,\delta_x.=0$ for $k>0$. Thus, if $T\in\mcurr k,\mu.$ is
  nontrivial, there is a Borel $U\subset X$ with
  \begin{equation}
    \cmass T.(U),\cmass T.(X\setminus U)>0;
  \end{equation}
  in particular,
  \begin{equation}
    \cmass T.(X)>\max(\cmass T\crest\chi_U.(X),\cmass T\crest(1-\chi_U).(X))
  \end{equation}
  and so $\mcurr k,\mu.$ is not an $L^\infty(\mu)$-normed module.
  \par The same argument can be applied if $k=0$ and $\mu$ is not a
  Dirac measure.
\end{proof}
We now introduce the notion of derivations. In the Introduction we
described sequential convergence for the weak* topology on $\lipalg
X.$; for further information we refer the reader to
\cite[Ch.~2]{weaver_book99}.
\begin{defn}\label{defn:derivations}
    A \textbf{derivation $D:\lipalg X.\to L^\infty(\mu)$} is a weak*
    continuous, bounded linear map satisfying the product rule:
    \begin{equation}
      D(fg)=fDg+gDf.
    \end{equation}
  \end{defn}
  \par Note that the product rule implies that $Df=0$ if $f$ is
  constant. The collection of all derivations $\wder\mu.$ is an
  $L^\infty(\mu)$-normed module \cite[Thm.~2]{weaver00} and the
corresponding  local norm will be denoted by
$\locnorm\,\cdot\,,{\wder\mu.}.$. Note also that $\wder\mu.$ depends
only on the measure class of $\mu$.
\par Observe that the norm of $D\in{\wder\mu.}$ is the supremum of
$\|Df\|_{L^\infty(\mu)}$ for $f$ bounded and $1$-Lipschitz. One can
then give a ``variational'' characterization of $\locnorm
D,{\wder\mu.}.$ as the infimum of $\lambda\in L^\infty(\mu)$ with
$\lambda\ge0$ and such that for each bounded $1$-Lipschitz function
$g$ one has $|Dg|\le\lambda$ (note that in $L^\infty(\mu)$
comparisons like $f_1\le f_2$ mean $f_1(x)\le f_2(x)$ for
$\mu$-a.e.~$x$).
\par  Finally recall that a free module is a module that has a basis,
i.e.~a generating set consisting of linearly independent elements.
\begin{rem}
  \label{rem:restr}
  Consider a Borel set $U\subset X$ and a derivation
  $D\in\wder\mu\mrest U.$. The derivation $D$ can be also regarded as
  an element of $\wder\mu.$ by extending $Df$ to be $0$ on $X\setminus
  U$ (compare Lemma \ref{lem:locality_derivations}). In particular,
  the module $\wder\mu\mrest U.$ can be naturally identified with the
  submodule $\chi_U\wder\mu.$ of $\wder\mu.$.
\end{rem}
\par Derivations are local in the following sense (\cite[Lem.~27]{weaver00}):
\begin{lem}\label{lem:locality_derivations}
  If $U$ is $\mu$-measurable and if $f,g\in\lipalg X.$ agree on $U$,
  then for each $D\in\wder\mu.$, $\chi_UDf=\chi_UDg$.
\end{lem}
Note that locality allows to extend the action of derivations on
Lipschitz functions so that if $f\in\lipfun X.$ and $D\in\wder\mu.$,
$Df$ is well-defined (see \sync rem:derivation_extension.\ in \cite{deralb}). We now
pass to consider some algebraic properties of $\wder\mu.$.
\par In general, even if the module $\wder\mu.$ is finitely generated,
it is not free. Nevertheless, it is possible to obtain a decomposition
into free modules over \emph{smaller rings} \cite{weaver00,derivdiff}:
\begin{thm}\label{thm:free_dec}
  Suppose that the module $\wder\mu.$ is finitely generated with $N$ generators. Then
  there is a Borel partition $X=\bigcup_{i=0}^N X_i$ such that, if
  $\mu(X_i)>0$, then $\wder\mu\mrest X_i.$ is free of rank $i$ as an
  $L^\infty(\mu\mrest X_i)$-module. A
  basis of $\wder\mu\mrest X_i.$ will be called a {\normalfont\textbf{local basis of
    derivations}}.
\end{thm}
\par In many applications in Analysis on metric spaces the assumption
that $\wder\mu.$ is finitely generated is not restrictive: for example
it holds if either $\mu$ or $X$ are doubling (\sync derbound.\ in \cite{deralb}).
\par In practice, to explicitly use the linear independence of some derivations
it is useful to construct \emph{pseudodual} Lipschitz functions:
\begin{defn}\label{defn:pseudodual_funcs}
  We say that
  Lipschitz functions $\{g_j\}_{j=1}^k\subset\lipalg X.$ are
  \textbf{pseudodual to  $\{D_i\}_{i=1}^k\subset\wder\mu.$ on a Borel
    set $U$}, if $\chi_U(D_ig_j-\delta_{i,j})=0$ and $\mu(U)>0$. In
  this case, note that the derivations $\{\chi_UD_i\}_{i=1}^k\subset\wder\mu\mrest U.$
  are independent\footnote{we consider the ring $L^\infty(\mu\mrest U)$}.
\end{defn}
\par The following Lemma constructs pseudodual functions given
independent derivations. However, it is a slight improvement of
similar results \cite{gong11-revised,derivdiff} because it controls the
norm of the derivations obtained. This improvement is used in the proof
of Theorem \ref{thm:gl_curr_rep}.
\begin{lem}\label{lem:quantitative_pseudodual}
  Suppose that the derivations $\{D_i\}_{i=1}^k\subset\wder\mu.$ are
  independent. Then there are a Borel $L^\infty(\mu)$-partition of
  unity $V_\alpha$ and there are, for each $\alpha$, derivations
  $\{D_{\alpha,i}\}_{i=1}^k\subset\chi_{V_\alpha}\wder\mu.$ and
  $1$-Lipschitz functions $\{g_{\alpha,j}\}_{j=1}^k\subset\lipalg X.$
  such that:
  \begin{enumerate}
  \item The submodule of $\wder\mu.$ generated by the derivations
    $\{D_{\alpha,i}\}_{i=1}^k$ contains the submodule generated by the derivations
    $\{\chi_{U_\alpha}D_i\}_{i=1}^k$;
  \item The derivations $\{D_{\alpha,i}\}_{i=1}^k$ have norm
 at most $C(k)$, a universal constant depending only on $k$;
  \item The functions $\{g_{\alpha,j}\}_{j=1}^k$ are pseudodual to the derivations
    $\{D_{\alpha,i}\}_{i=1}^k$ on $V_\alpha$.
  \end{enumerate}
\end{lem}
\par To prove Lemma \ref{lem:quantitative_pseudodual} we introduce a
notion of normalization for derivations. We first consider the set
where a given derivation vanishes:
\begin{defn}\label{defn:annih_set}
  Given a derivation $D\in\wder\mu.$, having chosen a Borel
  representative of $\locnorm D,{\wder\mu.}.$, we let
  \begin{equation}
    N_D=\left\{x: \locnorm D,{\wder\mu.}.(x)=0\right\};
  \end{equation} note that $N_D$ is well-defined up to Borel
  $\mu$-null sets and that $\lambda D=0$ iff
  $\lambda\in\chi_{N_D}L^\infty(\mu)$. If $N_D$ is $\mu$-null, we say
  that $D$ is \textbf{nowhere vanishing}.
\end{defn}
\begin{lem}\label{lem:normalization_derivation}
  For a derivation $D\in\wder\mu.$, having chosen a Borel
  representative of $\locnorm D,{\wder\mu.}.$, we let for $n\in\natural$
  \begin{equation}
    V_n=\left\{x: \locnorm D,{\wder\mu.}.\in\left(\|D\|_{\wder\mu.}/(n+1),\|D\|_{\wder\mu.}/n\right]\right\};
  \end{equation} then
  \begin{equation}
    \label{eq:normalization_derivation_s1}
    \tilde D=\sum_{\substack{n=1\\ \mu(V_n)>0}}^\infty \frac{\chi_{V_n}}{\chi_{V_n}\locnorm D,{\wder\mu.}.}D
  \end{equation} defines a derivation, {\normalfont\textbf{the normalization of
    $D$}}, with $\locnorm \tilde D,{\wder\mu.}.=\chi_{(N_D)^c}$. We
  will, with slight abuse of notation, denote the normalization of
  $D$ by $D/\locnorm D,{\wder\mu.}.$.
\end{lem}
\begin{proof}
  The definition of $\tilde D$ by \eqref{eq:normalization_derivation_s1}
  is well-posed because for every $n\in\natural$ such that
  $\mu(V_n)>0$ we have that   $\chi_{V_n}/\bigl(\chi_{V_n}\*\locnorm \tilde D,{\wder\mu.}.\bigr)$ is
  a function in $L^\infty(\mu)$. Moreover, the
  $V_n$ are uniquely determined up to $\mu$-null sets and so $\tilde D$
  does not depend on the choice of a Borel representative for $\locnorm D,{\wder\mu.}.$.
  Note that for $f\in\lipalg X.$ one has
  \begin{equation}
    \chi_{V_n}|Df|\le\chi_{V_n}\locnorm D,{\wder\mu.}.\|f\|_{\lipalg X.},
  \end{equation} and that the sets $\{V_n: \mu(V_n)>0\}$ are an
  $L^\infty(\mu\mrest N_D^c)$-Borel partition of unity. Thus
  \eqref{eq:normalization_derivation_s1} provides a bounded linear map
  $\tilde D:\lipalg X.\to L^\infty(\mu)$ with norm at most $1$. Note
  also that $\tilde D$ satisfies the product rule because $D$ does. 
  \par We show that $\tilde D$ is weak* continuous; by
  the Krein-\v Smulian Theorem, if suffices to check continuity for
  bounded nets. Therefore, suppose that $g\in L^1(\mu)$ and
  $f_{\eta}\xrightarrow{\text{w*}} f$ where the set
  $\{f_{\eta}\}_\eta\cup\{f\}$ is contained in the ball of radius $M$ in
  $\lipalg X.$. For each $\epsi>0$ there is an $N$ such that for all
  $h$ of norm at most $M$ in $\lipalg X.$,
  \begin{equation}\label{eq:normalization_derivation_p1}
    \left|\sum_{\substack{n>N\\ \mu(V_n)>0}}^\infty \int g\frac{\chi_{V_n}}{\chi_{V_n}\locnorm D,{\wder\mu.}.}Dh\,d\mu\right|\le\epsi;
  \end{equation} as
  \begin{equation}
    \tilde D_N=\sum_{\substack{n\le N\\ \mu(V_n)>0}}^\infty \frac{\chi_{V_n}}{\chi_{V_n}\locnorm D,{\wder\mu.}.} D
  \end{equation} is a derivation,
  \begin{equation}
    \label{eq:normalization_derivation_p2}
    \lim_{\eta}\int g\tilde D_Nf_\eta\,d\mu=\int g\tilde D_Nf\,d\mu;
  \end{equation} combining \eqref{eq:normalization_derivation_p1} and
  \eqref{eq:normalization_derivation_p2}, we conclude that
  \begin{equation}
    \lim_{\eta}\int g\tilde D f_\eta\,d\mu=\int g\tilde D f\,d\mu,
  \end{equation} which shows that $\tilde D$ is weak* continuous.
  \par We observe that $\chi_{N_D}$ annihilates $\tilde D$; thus, to
  show that $\locnorm \tilde D,{\wder\mu.}.=\chi_{N_D^c}$, it suffices
  to show that if the subset $U\subset N_D^c$ has positive measure, then
  $\|\chi_{U}\tilde D\|_{\wder\mu.}=1$. This follows because, for some
  $n$, $\mu(U\cap V_n)>0$ and
  \begin{equation}
    \chi_{U\cap V_n}\locnorm \tilde D,{\wder\mu.}.=\locnorm \chi_{U\cap V_n} \tilde D,{\wder\mu.}.=\chi_{U\cap V_n}.
  \end{equation}
\end{proof}
\begin{proof}[Proof of Lemma \ref{lem:quantitative_pseudodual}]
  Without loss of generality, we can assume that $\mu$ is finite.  We
  first prove that for each $\epsi>0$ there is a Borel
  $L^\infty(\mu)$-partition of unity $\{V_\alpha\}$ such that:
  \begin{itemize}
  \item For each  $\alpha$ there are $1$-Lipschitz functions $\{g_{\alpha,j}\}_{j=1}^k$
    and unit norm derivations
    $\{\tilde D_{\alpha,i}\}_{i=1}^k\subset\chi_{V_\alpha}\wder\mu.$;
  \item The submodule generated by the derivations $\{\tilde
    D_{\alpha,i}\}_{i=1}^k\subset\chi_{V_\alpha}\wder\mu.$ contains
    that generated by the derivations
    $\{\chi_{V_\alpha}D_i\}_{i=1}^k$;
  \item   The matrix $(\chi_{V_\alpha}\tilde
    D_{\alpha,i}g_j)_{i,j=1}^k$, with entries in
    $L^\infty(\mu\mrest V_{\alpha})$ (with absolute value
    $\le 1$ because of the first bullet point), is upper
    triangular;
  \item Each entry $\lambda$ on the diagonal of $(\chi_{V_\alpha}\tilde
    D_{\alpha,i}g_j)_{i,j=1}^k$ satisfies $\lambda\ge 1-\epsi$ (in the ring
    $L^\infty(\mu\mrest V_{\alpha})$). 
  \end{itemize}
We will refer to this property as $P(k,\epsi)$ and it will be proved
by induction on $k$.
\par For $k=1$, we first replace $D_1$ by its normalization $\tilde D_1$ (Lemma
\ref{lem:normalization_derivation}) to have $\locnorm
\tilde D_1,{\wder\mu.}.=1$, as $D_1$ is nowhere vanishing. Note that
\eqref{eq:normalization_derivation_s1} implies that $D_1=\locnorm
D_1,{\wder\mu.}.\tilde D_1$. We know that the
class ${\mathcal C}_1$ of Borel subsets $W$ such that there is a
$1$-Lipschitz $g$ with
\begin{equation}\label{eq:quantitative_pseudodual_p1}
  D_1g\ge 1-\epsi\quad\text{$\mu$-a.e. on $W$},
\end{equation} is not empty. We choose
\begin{equation}
  \mu(V_1)\ge\frac{1}{2}\sup_{W\in{\mathcal C}_1}\mu(W)
\end{equation}
and keep going exhausting $X$ in $\mu$-measure (compare the proof of
Theorem 2.43 in \cite{derivdiff}). The functions $g_\alpha$ are
chosen accordingly to the sets $V_\alpha$ so that
\eqref{eq:quantitative_pseudodual_p1} holds. Then one lets $\tilde
D_{\alpha,1}=\chi_{V_\alpha}\tilde D_1$. The derivation $\chi_{V_\alpha}D_1$
belongs to the submodule generated by $\tilde
D_{\alpha,1}$ because $\chi_{V_\alpha}D_1=\locnorm D_1,{\wder\mu.}.\tilde
D_{\alpha,1}$.
\par We now show that $P(k+1,\epsi)$ follows from $P(k,\epsi)$.
Using $P(k,\epsi)$ for the derivations $\{D_i\}_{i=1}^k$ we can assume,
by replacing $\mu$ with a restriction $\mu\mrest V$, that there are $1$-Lipschitz functions
$\{g_j\}_{j=1}^k$ and derivations $\{\tilde D_i\}_{i=1}^k$ such that
$P(k,\varepsilon)$ holds. We let
\begin{align}
  \label{eq:quantitative_pseudodual_p1:1}
  D^{(1)}_{k+1}&=D_{k+1}-\frac{D_{k+1}g_1}{\tilde D_1g_1}\tilde D_1\\
  \label{eq:quantitative_pseudodual_p1:2}
  D^{(l)}_{k+1}&=D^{(l-1)}_{k+1}-\frac{D^{(l-1)}_{k+1}g_l}{\tilde
    D_lg_l}\tilde D_l\quad\text{(for $2\le l\leq k$)},
\end{align}
and consider the normalization $\tilde D_{k+1}$ of $D^{(k)}_{k+1}$, so
that we have:
\begin{equation}
  \tilde D_{k+1}g_j=0\quad(1\le j\le k);
\end{equation} note that $D_{k+1}$ belongs to the submodule generated
by the derivations $\{\tilde D_i\}_{i=1}^{k+1}$. We now apply the argument used in the
case $k=1$ to the derivation $\tilde D_{k+1}$ in order to complete the
proof of $P(k+1,\varepsilon)$.
\par If $M_\alpha$ denotes the matrix $(\tilde
D_{\alpha,i}g_{\alpha,j})_{i,j=1}^k$, its determinant satisfies the
bounds:
\begin{equation}
  \label{eq:quantitative_pseudop1}
  (1-\epsi)^k\le\det M_\alpha\le 1,
\end{equation} 
and its entries lie in $[-1,1]$.
In particular, letting
\begin{equation}\label{eq:quantitative_pseudop2}
  D_{\alpha,i}=\sum_{j=1}^k(M_\alpha^{-1})_{i,j}\tilde D_{\alpha,j},
\end{equation}
we have $\locnorm D_{\alpha,i},\mu\mrest V_\alpha.\le C(k,\epsi)$,
where $C(k,\varepsilon)$ is a universal constant depending only on $k$
and $\varepsilon$, and 
$D_{\alpha,i}g_{\alpha,j}=\delta_{i,j}\chi_{V_\alpha}$. In fact, the
entries of $M_\alpha^{-1}$ can be bounded from above by
$(k-1)!\det(M_\alpha)^{-1}$ using Cramer's formula for the inverse matrix.
Moreover,
solving \eqref{eq:quantitative_pseudop2} for the derivations $\{\tilde
D_{\alpha,i}\}_{i=1}^k$ shows that the derivations
$\{\chi_{V_\alpha}D_i\}_{i=1}^k$ belong to the submodule generated by
the $\{D_{\alpha,i}\}_{i=1}^k$.
\end{proof}
\par Consider a Lipschitz map $F:X\to Y$ and a Radon measure $\mu$ on
$X$; given a derivation $D\in\wder\mu.$ the \textbf{push forward
  $\mpush F.D\in\wder{\mpush F.}\mu.$} is the derivation defined by:
\begin{equation}
  \label{eq:der_push_forw}
  \int_Yg\,(\mpush F.D)f\,d\mpush F.\mu=\int_Xg\circ F\,D(f\circ
  F)\,d\mu\quad(\forall(f,g)\in\metuple 1,Y.).
\end{equation}
\par We now recall the notion of $1$-forms which are dual to
derivations.
\begin{defn}
  \label{defn:module_forms}
  The \textbf{module of $1$-forms} $\wform\mu.$ is the dual module of
  $\wder\mu.$, i.e.~it consists of the bounded module homomorphisms
  $\wder\mu.\to L^\infty(\mu)$. The module $\wform\mu.$ is an
  $L^\infty(\mu)$-normed module and the local norm will be denoted by
  $\locnorm\,\cdot\,,{\wform\mu.}.$.
\par Recall that the norm of $\omega\in{\wform\mu.}$ is the supremum of
$\|\langle D,\omega\rangle\|_{L^\infty(\mu)}$ for $D\in\wder\mu.$ of
norm $1$ (here $\langle\cdot,\cdot\rangle$ denotes the duality
pairing). One can
then give a ``variational'' characterization of $\locnorm
\omega,{\wform\mu.}.$ as the infimum of $\lambda\in L^\infty(\mu)$ with
$\lambda\ge0$ and such that for each $D\in\wder\mu.$ of
norm $1$ one has $|\langle D,\omega\rangle|\le\lambda$.
  \par To each $f\in\lipalg X.$ one can associate the $1$-form
  $df\in\wform\mu.$ by letting:
  \begin{equation}
    \label{eq:module_forms1}
    \langle df,D\rangle=Df\quad(\forall D\in\wder\mu.);
  \end{equation}
  the map $d:\lipalg X.\to\wform\mu.$ is a weak* continuous $1$-Lipschitz
  linear map satisfying the product rule $d(fg)=gdf+fdg$.
\end{defn}
Note that because of Lemma \ref{lem:locality_derivations} one can
extend the domain of $d$ to $\lipfun X.$ so that if $f$ is Lipschitz,
$df$ is a well-defined element of $\wform\mu.$ and
$\|df\|_{\wform\mu.}\le\glip f.$.
\par We also point out that while we follow a notion local norms due
to Weaver, recently Gigli has done a systematic and beautiful work~\cite{gigli-diffstru,gigli-nsdiff-geo}
on derivations and duality for $L^\infty$-modules
which deals also with other notions of norms, e.g.~those arising from
minimal upper gradients. Note also that in this paper we allow for
derivations to have the minimal degree of regularity allowing
differential calculus
and thus many notions, like minimal upper gradients, can become vacuous in
our setting.
\subsection{Correspondence between derivations and Alberti representations}
\label{subsec:corresp}
\par In this Subsection we recall some results in \cite{deralb} about
the correspondence between derivations and Alberti
representations. Throughout this Subsection $F:X\to\real^k$ denotes a
Lipschitz function, $\alpha\in(0,\pi/2)$ an angle, $\delta$ a positive
constant, $w\in{\mathbb S}^{k-1}$ a unit vector and
$\{u_i\}_{i=1}^{k-1}$ an orthonormal basis for the orthogonal
complement of $w$.
\par We first recall an \emph{approximation scheme} (\sync
onedimapprox_multi.\ in \cite{deralb}) which relates Alberti
representations and the weak* topology on $\lipalg X.$:
\begin{thm}
  \label{thm:appxscheme}
  Let $X$ be a compact metric space and $\mu$ a Radon measure on
  $X$. Suppose that $K\subset X$ is compact and
  $\frags\left(X,F,\cone(w,\alpha),\langle
    w,F\rangle,\ge\delta\right)$-null. Denoting by $d_{\delta,\alpha}$
  the distance:
  \begin{equation}
    \label{eq:appxscheme_s1}
    d_{\delta,\alpha}(x,y)=\delta d(x,y) +
    \cot\alpha\sum_{i=1}^{k-1}\left|\langle u_i,F(x)-F(y)\rangle\right|,
  \end{equation}
  there is a sequence of real-valued Lipschitz functions $\{g_n\}$
  and a Borel $S\subset K$ such that:
  \begin{enumerate}
  \item $\mu(K\setminus S)=0$;
  \item $g_n\xrightarrow{\text{w*}}\langle w,F\rangle$;
  \item for each $x\in S$ and each $n$ there is an $r_n>0$ such that
    the restriction $g_n|B(x,r_n)$ is $1$-Lipschitz with respect to the distance
    $d_{\delta,\alpha}$. 
  \end{enumerate}
\end{thm}
Note that here we use the $l^1$-distance in the part
  of $d_{\delta,\alpha}$ multiplied by $\cot\alpha$. In~\cite{deralb}
  we used the $l^2$-distance, but the result is still true because the
  $l^1$-distance is always $\ge$ the $l^2$-distance.
\par We will use the following consequence of Theorem \ref{thm:appxscheme}.
\begin{lem}
  \label{lem:dst_appx}
  Let $X$ be a complete separable metric measure space and $\mu$ a
  Radon measure on $X$. Suppose that the compact set $K\subset X$ is  $\frags\left(X,F,\cone(w,\alpha),\langle
    w,F\rangle,>\delta\right)$-null. Then there are bounded Lipschitz
  functions $\tilde f_n\xrightarrow{\text{w*}}\tilde f$ and a Borel subset
  $S\subset K$ having full $\mu$-measure in $K$ such that:
  \begin{enumerate}
  \item The function $\tilde f$ agrees with $\langle w, F\rangle$ on $K$;
  \item For each $n\in\natural$ there are bounded Lipschitz functions $\tilde
    f_{n,m}\xrightarrow{\text{w*}}\tilde g_{n}$  where $\tilde g_{n}$
    agrees with $\tilde f_n$ on $S$;
  \item For each $(n,m)\in\natural^2$ there are finitely many points
    $\{x_{n,m,a}\}_a\subset S$ and finitely many disjoint
    Borel sets $\{S_{n,m,a}\}_a$ with
    $S=\bigcup_{a}S_{n,m,a}$ and
    \begin{equation}
      \label{eq:dst_appx_s1}
      \tilde f_{n,m} = \tilde f_{n} +
      d_{\delta,\alpha}(\cdot,x_{n,m,a})\quad\text{on $S_{n,m,a}$.}
    \end{equation}
  \end{enumerate}
\end{lem}
\begin{proof}
  We apply Theorem~\ref{thm:appxscheme} using $K$ \emph{both as the
    subset and as the ambient metric space.} We thus find a sequence
  $\{f_n\}_n\subset\lipalg K.$ with $f_n\xrightarrow{\text{w*}}\langle
  w,F\rangle$, and a $\mu$-full measure Borel subset $S\subset K$ such
  that for each $x\in S$ and each $n\in\natural$ there is an $r_n(x)>0$
  such that the restriction $f_n|B(x,r_n(x))\cap K$ (we are still inside
  $K$) is $1$-Lipschitz with respect to the distance
  $d_{\delta,\alpha}$.
\par We now take a Mac Shane's extension $\tilde f_n:X\to\real$ of $f_n$
while keeping
\begin{align*}
  \label{eq:taira_1}
  \glip\tilde f_n. &= \glip f_n.\\
  \|\tilde f_n\|_\infty &= \|f_n\|_\infty.
\end{align*}
As $X$ is separable and as the $\sup_n\glip\tilde f_n.<\infty$ and the
$\tilde f_n$ converge on $K$ to $\langle w,F\rangle$, using
Ascoli-Arzel\'a and up to passing to
a  subsequence, we can assume that $\tilde
f_n\xrightarrow{\text{w*}}\tilde f$ where $\tilde f$ agrees with
$\langle w, F\rangle$ on $K$.
\par For each $(n,m)\in\natural^2$ choose a finite $1/m$-dense subset
$\{x_{n,m,a}\}_a\subset S$ and let:
\begin{equation*}
  \label{eq:taira_3}
  \tilde f_{n,m} = \max_a\left\{
\tilde f_n(x_{n,m,a}) + d_{\delta,\alpha}(\cdot,x_{n,m,a})
\right\}
\end{equation*}
so that conclusion (3) is automatically satisfied. Note that we can
also truncate $\tilde f_{n,m}$ so that:
\begin{equation*}
  \label{eq:taira_4}
  \|\tilde f_{n,m}\|_\infty=\sup_{x\in S}|\tilde f_n(x)|,
\end{equation*}
and without changing its values on points of $S$.
\par Finally, using for each $n$ Ascoli-Arzel\'a and passing to a
subsequence we can assume $\tilde
f_{n,m}\xrightarrow{\text{w*}}\tilde g_n$. To conclude that $\tilde g_n$ agrees with
$\tilde f_n$ on $S$, we pick $x\in S$ and observe that, as the
restriction $f_n|B(x,r_n(x))\cap K$ is $1$-Lipschitz with respect to the distance
  $d_{\delta,\alpha}$, for
$\frac{1}{m}<r_n(x)$ one has:
\begin{equation*}
  \label{eq:taira_5}
\left|\tilde f_{n,m}(x) - \tilde f_n(x) \right|\le C(\alpha,\delta)\frac{1}{m},
\end{equation*}
where $C(\alpha,\delta)$ is independent of $n$ and $m$.
\end{proof}
\par In \sync thm:alb_derivation.\ in \cite{deralb} it was shown that
to a $C$-biLipschitz Alberti representation $\albrep.$ of the measure $\mu$ it is
possible to associate a derivation $D_{\albrep.}\in\wder\mu.$ by using
the formula:
\begin{equation}\label{eq:derivation_alberti}
      \int_X gD_{\albrep.}f\,d\mu=\int_{\frags(X)}dP(\gamma)\int_{\dom\gamma}
      (f\circ
      \gamma)'(t)g\circ\gamma(t)\,d(\mpush\gamma^{-1}.\nu_\gamma)(t)
      \quad(g\in L^1(\mu)\cap\bborel X.)
  \end{equation}
  to define $D_{\albrep.}f$; moreover, one has the norm bound
  $\|D_{\albrep.}\|_{\wder\mu.}\le C$ and if the Alberti representation $\albrep.$ is in the
  $F$-direction of the $k$-dimensional cone field $\cone$, one has
  $D_{\albrep.}F(x)\in\cone(x)$ for $\mu$-a.e.~$x$.
Finally note that using Remark~\ref{rem:new_frags} one can
extend~(\ref{eq:derivation_alberti}) to the case of a $C$-Lipschitz
Alberti representations by doing the replacement:
$$
\int_{\dom\gamma}
      (f\circ
      \gamma)'(t)g\circ\gamma(t)\,d(\mpush\gamma^{-1}.\nu_\gamma)(t)
\mapsto \sum_\alpha\int_{S_\alpha}
      (f\circ
      \gamma)'(t)g\circ\gamma(t)\,d(\mpush\gamma^{-1}.\nu_\gamma)(t).
$$
  \par In order to compare the derivations associated to different
  Alberti representations the following notion of independence for
  cone fields is useful:
  \begin{defn}
    \label{defn:cone_indp}
    We say that the $n$-dimensional cone fields $\{\cone_i\}_{i=1}^k$
    are \textbf{independent} if for each $x\in X$ and each choice of
    $v_{i,x}\in\cone_i(x)$, the vectors $\{v_{i,x}\}_{i=1}^k$ are
    linearly independent.
  \end{defn}
  Note that if the Alberti representations $\{\albrep i.\}_{i=1}^k$
  are in the $F$-directions of independent cone fields, where
  $F:X\to\real^k$ is Lipschitz, the
  corresponding derivations $\{D_{\albrep i.}\}_{i=1}^k$ are
  independent. We will use the following results (\sync thm:taira.\
  and \sync
  cor:mu_arb_cone.\ in \cite{deralb}):
\begin{thm}\label{thm:taira}
  Let $X$ be a complete separable metric space and $\mu$ a Radon measure on $X$. Consider
  a Borel set $V\subset X$, derivations
  $\{D_1,\ldots,D_k\}\subset\wder\mu.$ and a Lipschitz function
  ${g\colon X\to\real^k}$ such that $D_ig_j=\delta_{i,j}\chi_V$. Then for each
  $\epsi>0$, unit vector $w\in{\mathbb S}^{k-1}$, angle $\alpha\in(0,\pi/2)$ and 
 speed parameter $\sigma\in(0,1)$, the measure $\mu\mrest V$ admits a
  ${(1,1+\epsi)}$\nobreakdash-\hspace{0pt}bi-Lipschitz Alberti representation in the
  $g$-direction of $\cone(w,\alpha)$ with 
  \begin{equation}\label{eq:speed_almost_optimal}
    {\text{$\langle w,g\rangle$\nobreakdash-\hspace{0pt}speed}}
    \ge\frac{\sigma}{\locnorm D_w,{\wder\mu\mrest V.}.+(1-\sigma)},
  \end{equation}
where $D_w=\sum_{i=1}^kw_iD_i$.
\end{thm}
  \begin{cor}
    \label{cor:mu_arb_cone}
    Suppose that the measure $\mu$ admits Alberti representations in
    the $F$-direction of $k$ independent cone fields, where $F:X\to\real^k$ is Lipschitz. Then for each
    $\varepsilon>0$ and each $k$-dimensional cone field $\cone$, the
    measure $\mu$ admits a $(1,1+\varepsilon)$-biLipschitz Alberti representation
    in the $F$-direction of $\cone.$
  \end{cor}
  Combining Theorem~\ref{thm:taira} and
  Corollary~\ref{cor:mu_arb_cone} we immediately get:
  \begin{cor}
    \label{cor:mu_arb_cone_special}
    Suppose that the components $\{F_i\}_{i=1}^k$ of $F:X\to\real^k$ are pseudodual to the
    derivations $\{D_i\}_{i=1}^k$; then for any $k$-dimensional cone
    field $\cone$, the measure $\mu$ admits an Alberti representation
    in the $F$-direction of $\cone$.
  \end{cor}
\newcount\orgsec                
\orgsec=0
\ifnum\orgsec>0 {
  \section{Organization}

\subsection*{Outline}
\begin{itemize}
\item choice of metric currents: Ambrosio-Kirchheim;
  metric spaces are complete and to work with Alberti representations
  assume separable;
\item efficient $k$-tuples for the mass Lemma 4.1;
\item dimensional bounds; nontrivial currents bounded by the Assouad
  dimension; if there is a uniform bound on the Hausdorff dimensions
  of the tangents one gets another bound;
\item relationship with Paolini: alberti representations with constant
  direction;
\item state in a clean way theorem 10 in Weaver; clarify how it
  relates to my decomposition result;
\item try to remove dependence on a uniform bound on the infinitesimal
  Assouad dimension of the support of the current using that
  $\wder\mu.$ has a weak* topology; in the negative case try to look
  at examples in hilbert cubes giving a negative answer;
\item Consider getting in Corollary \ref{cor:currents_assouad_bound} a
  bound in terms of the Hausdorff dimension (using ultrafilters).
\item Proof of flatness of $1$-currents in Banach spaces;
\item Notion of push-forward of a derivation;
\item Production of normal derivations.
\item Fragments: do not require positive Lebesgue measure and biLipschitz embedding.
\end{itemize}
\subsection*{Prerequisites}
\begin{itemize}
\item Alberti representations, production and gluing.
\item Metric currents.
\end{itemize}
\subsection*{Questions}
\begin{itemize}
\item Is $\mcurr k,\mu.$ an $L^\infty(\mu)$-normed module?
  \par No, but it is an $L^\infty(\mu)$-module.
\item Can we improve the map \eqref{eq:der_to_curr} to an
  $L^\infty(\mu)$-module isomorphism preserving the local norm?
  \par No.
\item Is the $\Cur\mu.(\wder\mu.)$ closed in $\mcurr 1,\mu.$?
  \par Probably not.
\item How does the Nagata dimension fit with dimensional bounds?
\item Is it possible to give an explicit description of the local norm
  in $\alt k,M,N.$?
\end{itemize}
\subsection*{Notation and Definitions}
\begin{itemize}
\item $\Lambda_{k,N}$ denotes the set of increasing $k$-tuples in the
  $\{1,\ldots,N\}$. 
\item To simplify the notation, $X^*$ will denote the dual Banach
  space and $X'$ the dual module.
\item For a metric space $X$ and a compact interval $I$ we denote by
  $\paths_I(X)$ the set of Lipschitz embeddings of $I$ in $X$.
\end{itemize}
\subsection*{Statements}
}\fi
\section{$1$-dimensional currents and derivations}
\label{sec:onedimcurr}
\par The goal of this Section is to make precise the correspondence
between $1$-dimensional metric currents and derivations via Theorem
\ref{thm:one_currents_derivations}.
\begin{lem}\label{lem:good_ktuples}
  Consider a metric functional $T\in\mfunc k,X.$ with finite mass. If $B\subset X$ is
  Borel and $\cmass T.(B)>0$, then for each $\eta\in(0,1)$ there are
  disjoint Borel sets $B_i\subset B$ and $1$-Lipschitz
  functions\footnote{with respect to the $l^\infty$-norm}
  $\pi^i:X\to\real^k$: 
  \begin{subequations}\label{eq:good_ktuples}
    \begin{align}
      \cmass T.\left(B\setminus\bigsqcup_i B_i\right)&=0;\label{eq:good_ktuples_E}\\
      \left|T(\chi_{B_i},\pi^i_1,\cdots,\pi^i_k)\right|&>\eta\cmass T.(B_i).\label{eq:good_ktuples_T}
    \end{align}
  \end{subequations}
\end{lem}
\begin{proof}
  The proof uses \cite[Prop.~2.7]{ambrosio-kirch} (characterization
  of mass): for each $\epsi>0$ there are disjoint Borel sets  $B_i\subset B$ and
  $1$-Lipschitz functions $\pi^i:X\to\real^k$:
  \begin{align}
    B&=\bigcup_i B_i;\\
    \sum_i\Bigl(\cmass T.(B_i)&-\left|T(\chi_{B_i},\pi^i_1,\cdots,\pi^i_k)\right|\Bigr)<\epsi;
  \end{align}
let
$J_\eta=\{i:\left|T(\chi_{B_i},\pi^i_1,\cdots,\pi^i_k)\right|\le\eta
\cmass T.(B_i)\}$; then one has:
\begin{equation}
  (1-\eta)\sum_{i\in J_\eta} \cmass T.(B_i)<\epsi;
\end{equation}so 
\begin{equation}
  \cmass T.\left(\bigsqcup_{i\in J_\eta}B_i\right)<\frac{\epsi}{1-\eta};
\end{equation} therefore the conclusion of the Lemma is true for those
$i\not\in J_\eta$ which cover all but $\frac{\epsi}{1-\eta}$ of
the $\cmass T.$-measure of $B$. The result follows by an exhaustion argument.
\end{proof}
\begin{thm}\label{thm:one_currents_derivations}
  Let $\mu$ be a finite Radon measure on $X$. There is a map
  \begin{equation}\label{eq:curr_to_der}
    \begin{aligned}
      \Der\mu.:\mcurr 1,\mu.&\to\wder\mu.\\
      T&\mapsto D_T
    \end{aligned}
  \end{equation}
  where  $D_T\in\wder{\cmass T.}.$ is the unique derivation satisfying
  \begin{subequations}
    \begin{align}
      T(f,\pi)&=\int fD_T\pi\,d\cmass T.\quad\text{\normalfont($\forall (f,\pi)\in
        L^1(\cmass T.)\times\lipfun X.$)}\label{eq:curr_def_der1}\\
      \locnorm D_T,{\wder{\cmass T.}.}.&=1.\label{eq:curr_def_der2}
    \end{align}
  \end{subequations}
Moreover, one also has:
\begin{equation}
  \label{eq:xtra+1}
   \locnorm D_T,{\wder{\mu}.}.(x) =
   \begin{cases}
     1&\text{if $\frac{d\|T\|}{d\mu}(x)\ne0$}\\
     0&\text{otherwise.}
   \end{cases}
\end{equation}
  \par Conversely, there is an $L^\infty(\mu)$-module homomorphism
  map
  \begin{equation}
    \begin{aligned}
      \label{eq:der_to_curr}
      \Cur\mu.:\wder\mu.&\to\mcurr 1,\mu.\\
      D&\mapsto T_D
    \end{aligned}
  \end{equation}
  where $T_D$ is the unique current satisfying
  \begin{subequations}
    \begin{align}
    T_D(f,\pi)&=\int f D\pi\,d\mu\quad\text{\normalfont($\forall (f,\pi)\in
      L^1(\cmass T.)\times\lipfun X.$)}\label{eq:der_def_curr1}\\
    \cmass T_D.&=\locnorm D,{\wder\mu.}.\,\mu.\label{eq:der_def_curr2}
  \end{align}
\end{subequations}
\end{thm}
\begin{proof}
  Given $T\in\mcurr 1,\mu.$, for a fixed $f\in\lipalg X.$  one defines
  a linear functional on $L^1(\cmass T.)$ by:
\begin{equation}
  g\mapsto T(g,f)\quad(g\in L^1(\cmass T.));
\end{equation}
the Riesz Representation Theorem gives a unique $D_Tf\in
L^\infty(\cmass T.)$:
  \begin{equation}\label{1dim_rep}
    \int_X gD_Tf\,d\cmass T.=T(g,f);
  \end{equation}
  the map $D_T:\lipalg X.\to L^\infty(\cmass T.)$ is a derivation
  because:
  \begin{itemize}
  \item It is linear by linearity of currents;
  \item It is bounded with norm $1$ because:
    \begin{equation}
      \left|\int_X gD_Tf\,d\cmass T.\right|\le\glip f.\int_X|g|\,d\cmass T.;
    \end{equation}
  \item The product rule follows from \cite[Eq.~3.1 in Thm.~3.5]{ambrosio-kirch}; 
  \item The weak* continuity follows from the continuity axiom for
    currents ((4) in Defn.~\ref{defn:metric_currents}).
  \end{itemize}
  Note that the module $\wder{\cmass T.}.$ can be canonically
  identified with the submodule
  $\chi_{U_T}\wder\mu.$ where
  \begin{equation}
    U_T=\left\{x\in X: \frac{d\cmass T.}{d\mu}(x)>0\right\},
  \end{equation}
  so $\Der\mu.$ is well-defined and then~(\ref{eq:xtra+1}) will follow from~(\ref{eq:curr_def_der2}).
  \par By Lemma \ref{lem:good_ktuples}, for each $\eta\in(0,1)$ we can
  find disjoint Borel sets $B_i$ and $1$-Lipschitz functions
  $\pi^i\in\lipfun X.$
  with $\cmass T.(X\setminus\bigcup_iB_i)=0$ and
  \begin{equation}
    T(\chi_{B_i},\pi^i)>\eta\cmass T.(B_i);
  \end{equation}
  in particular, for each $n\in\natural$ one has
  $\chi_{S_i}D_T\pi^i\ge\frac{n}{n+1}\eta\chi_{S_i}$, where $S_i$ is a subset of
  $B_i$ of measure at least $\frac{\eta}{n+1}\cmass T.(B_i)$; using an
  exhaustion argument and then letting $\eta\to1$ and $n\nearrow\infty$, we
  conclude that \eqref{eq:curr_def_der2} holds. Note that we have used
  the fact that each derivation $D\in\wder\mu.$ can be canonically extended to a map
  $D:\lipfun X.\to L^\infty(\mu)$ (see \sync rem:derivation_extension.\ in \cite{deralb}).
  \par We now prove the second part of this Theorem; note that for $D\in\wder\mu.$
  \eqref{eq:der_def_curr1} uniquely determines a current $T_D\in\mcurr 1,\mu.$
  because the axioms of metric currents follow from the corresponding properties
  of derivations. Note also that $T_{D_1+D_2}=T_{D_1}+T_{D_2}$ and
  $T_{\lambda D}=T_D\crest\lambda$, showing that $\Cur\mu.$ is an
  $L^\infty(\mu)$-module homomorphism.
  \par As $|D\pi|\le\glip\pi.\,\locnorm D,{\wder\mu.}.$, $\cmass
  T_D.\le\locnorm D,{\wder\mu.}.\,\mu$. On the other hand, for each
  $\eta\in(0,1)$ and each Borel set $A$, we can find disjoint
  Borel sets $B_i\subset A$ and $1$-Lipschitz functions $\pi^i$ with $\cmass
  T.(A\setminus\bigcup_iB_i)=0$ and
  \begin{equation}
    \chi_{B_i}D\pi^i\ge\eta\chi_{B_i}\locnorm D,{\wder\mu.}.;
  \end{equation}
  in particular,
  \begin{equation}
    \cmass T_D.(A)\ge\eta\int_A\locnorm D,{\wder\mu.}.\,d\mu
  \end{equation}
  which implies \eqref{eq:der_def_curr2}.
\end{proof}
\begin{rem}
  From Theorem \ref{thm:one_currents_derivations} one obtains the
  following identities:
  \begin{align}
    \Cur\mu.\left(\Der\mu.(T)\right)\crest\frac{d\cmass T.}{d\mu}&=T\\
    \Der\mu.\left(\Cur\mu.(D)\right)&=\frac{D}{\locnorm D,{\wder\mu.}.}.
  \end{align}
\end{rem}
\section{Currents and Alberti representations}
\label{sec:currents_and_alberti}
\par The goal of this Section is to prove Theorem
\ref{thm:currents_arbitrary_cones}.  Throughout this Section we will
denote by $\{e_i\}_{i=1}^k$ the standard basis of $\real^k$. In the
proof of Theorem \ref{thm:currents_arbitrary_cones} we will use the
following consequence of Rainwater's Lemma \cite{rainwater-note}
(compare Corollary 5.8 in \cite{bate-diff} and \sync biLip_dis.\ in
\cite{deralb}):
\begin{lem}
  \label{lem:rainw_part}
  Let $X$ be a complete separable metric space and $\mu$ a Radon
  measure on $X$. Let $f:X\to\real^k$ be a Lipschitz map, $w\in{\mathbb
    S}^{k-1}$, $\alpha\in(0,\pi/2)$ and $\delta>0$. For any Borel
  subset $B\subset X$ there are disjoint Borel sets $A,S$ such that:
  \begin{enumerate}
  \item $A\cup S = B$;
  \item The measure $\mu\mrest A$ admits an Alberti representation in the
    $f$-direction of $\cone(w,\alpha)$ with $\langle w,f\rangle$-speed $\ge\delta$;
  \item The set $S$ is $\frags(X,f,\cone(w,\alpha),\langle w,f\rangle,\ge\delta)$-null.
  \end{enumerate}
\end{lem}
\par The proof of Theorem \ref{thm:currents_arbitrary_cones} relies on
the following Lemma:
\begin{lem}\label{lem:curr_cone_estimate}
  Let $X$ be as above and let $T$ be a $k$-dimensional metric current in $X$.
  Suppose that $T(\chi_B,\pi_1,\cdots,\pi_k)\ge\eta\cmass T.(B)$, where
  $B$ is Borel and $\pi:X\to\real$ is $1$-Lipschitz and $\eta>0$; then for all
  pairs $(\delta,\alpha)\in(0,\eta)\times(0,\pi/2)$ there is a Borel
  partition $B=A_{e_i}\cup S_{e_i}$ with $\cmass T.\mrest A_{e_i}$
  admitting an Alberti representation in the $\pi$-direction of
  $\cone(e_i,\alpha)$ with $\pi_i$-speed $\ge\delta$ and 
  $\cmass T.(A_{e_i})\ge(\eta-\delta)\cmass T.(B)$.
\end{lem}
\begin{proof}
  Without loss of generality, we assume $i=1$.
  Because of Lemma \ref{lem:rainw_part} we will obtain an upper
  bound on $\cmass T.(K)$, where $K\subset B$ is compact and
  $\frags(X,\pi,\cone(e_1,\alpha),\pi_1,\ge\delta)$-null. We apply
  Lemma \ref{lem:dst_appx} and we will use the notation from its
  statement in the remainder of the proof. In particular, we take
  $w=e_1$, $u_i=e_{1+i}$ and $F=(\pi_i)_{i=1}^k$. The following
  estimate is obtained by using the locality axiom ((3) in Definition
  \ref{defn:metric_currents}) and \eqref{eq:dst_appx_s1}:
  \begin{multline}\label{eq:curr_cone_estimate_p1}
    \left|T(\chi_{S_{n,m,a}},\tilde f_{n,m,a},\pi_2,\ldots,\pi_k)\right|\le\delta
    \left|T(\chi_{S_{n,m,a}},\dist\cdot,x_{n,m,a}.,\pi_2,\ldots,\pi_k)\right|\\+\cot\alpha\sum_{\beta>1}
    \left|T(\chi_{S_{n,m,a}}, \left|\pi_\beta-\pi_\beta(x_{n,m,a})\right|,\pi_2,\ldots,\pi_k)\right|;
  \end{multline}
  we now let
  \begin{equation}
    \begin{aligned}
      S_{n,m,a,\beta+}&=\left\{x\in S_{n,m,a}:
        \pi_\beta(x)\ge\pi_\beta(x_{n,m,a})\right\}\\
      S_{n,m,a,\beta-}&=\left\{x\in S_{n,m,a}:
        \pi_\beta(x)<\pi_\beta(x_{n,m,a})\right\},
    \end{aligned}
  \end{equation}
  and conclude that, for $\beta>1$,
  \begin{equation}\label{eq:curr_cone_estimate_p2}
    \begin{split}
      T(\chi_{S_{n,m,a}},
      |&\pi_\beta-\pi_\beta(x_{n,m,a})|,\pi_2,\ldots,\pi_k)
      \\ &= T(\chi_{S_{n,m,a,\beta+}},
      \pi_\beta-\pi_\beta(x_{n,m,a}), \pi_2,\ldots,\pi_k)  \\&\;-
      T(\chi_{S_{n,m,a,\beta-}}, \pi_\beta-\pi_\beta(x_{n,m,a}),
      \pi_2,\ldots,\pi_k)\\
      &= T(\chi_{S_{n,m,a,\beta+}},\pi_\beta,\pi_2,\ldots,\pi_k) \\&\;-
      T(\chi_{S_{n,m,a,\beta-}},\pi_\beta,\pi_2,\ldots,\pi_k)\\&=0
    \end{split}
  \end{equation}
  where in the last inequality we used that currents
  are alternating. Combining \eqref{eq:curr_cone_estimate_p1} and
  \eqref{eq:curr_cone_estimate_p2} we obtain:
  \begin{equation}
    \label{eq:curr_cone_estimate_p3}
    \left|T(\chi_{S_{n,m,a}},\tilde
      f_{n,m,a},\pi_2,\cdots,\pi_k)\right|\le\delta\cmass T.(S_{n,m,a}).
  \end{equation}
 Summing in $a$ and letting $m\nearrow\infty$ we obtain:
  \begin{equation}
    \label{eq:curr_cone_estimate_p4}
    \left|T(\chi_{S_{}},\tilde
      f_{n},\pi_2,\cdots,\pi_k)\right|\le\delta\cmass T.(S_{});
  \end{equation}
  but as $\cmass T.(K\setminus S)=0$:
  \begin{equation}
    \label{eq:curr_cone_estimate_p5}
 \left|T(\chi_K,\tilde f_n,\pi_2,\cdots,\pi_k)\right|\le\delta\cmass T.(K);
  \end{equation}
  letting $n\nearrow\infty$ and using that $\tilde f=\pi_1$ on $K$ we conclude that
  \begin{equation}
    \label{eq:curr_cone_estimate_p6}
\left|T(\chi_K,\pi_1,\pi_2,\cdots,\pi_k)\right|\le\delta\cmass T.(K).
  \end{equation}
  The proof is completed by applying Lemma \ref{lem:rainw_part}.
\end{proof}
\begin{proof}[Proof of Theorem \ref{thm:currents_arbitrary_cones}]
  For $\eta\in(0,1)$ let the sets $B_j$ and the functions $\pi^j$ satisfy the conclusion of
  Lemma \ref{lem:good_ktuples} for $B=X$. Let $\alpha\in(0,\pi/2)$
  be such that the cone fields $\{\cone(e_i,\alpha)\}_{i=1}^k$ are
  independent. For $\delta>0$, Lemma \ref{lem:curr_cone_estimate}
  gives a partition $B_j=A_{j,e_1}\cup S_{j,e_1}$ with $\cmass T.\crest
  A_{j,e_1}$ admitting an Alberti representation in the
  $\pi^j$-direction of $\cone(e_1,\alpha)$ with $\pi^j_1$-speed
  $\ge\delta$ and
  \begin{equation}
    \cmass T.(A_{j,e_1})\ge(\eta-\delta)\cmass T.(B);
  \end{equation}
  proceeding by induction and applying Lemma \ref{lem:good_ktuples}, we
  obtain a partition
  \begin{equation}
    B_j=A_{j,e_1,\ldots,e_k}\cup
    S_{j,e_1,\ldots,e_k}
  \end{equation}
  with $\cmass T.\crest A_{j,e_1,\ldots,e_k}$
  admitting Alberti representations in the $\pi^j$-directions of the
  cone fields
  $\{\cone(e_i,\alpha)\}_{i=1}^k$ and
  \begin{equation}
    \cmass T.(A_{j,e_1,\ldots,e_k})\ge\underbrace{\prod_{i=1}^k(\eta-i\delta)}_{c}\,\cmass T.(B).
  \end{equation}
  If $\delta\in(0,\eta/k)$, $c>0$; as $\cmass T.\crest
  A_{j,e_1,\ldots,e_k}$ admits Alberti representations in the
  $\pi^j$-directions of $k$ independent cone fields, the proof is
  completed by applying Corollary \ref{cor:mu_arb_cone} and an
  exhaustion argument.
\end{proof}
\begin{cor}\label{cor:currents_assouad_bound}
  If $X$ is a metric space with Assouad dimension $\le Q$, then
  \begin{equation}\mcurr k,X.=\{0\}
  \end{equation}
  for $k>Q$; moreover, if $T\in\mcurr k,X.$, the module
  $\wder{\cmass T.}.$ if finitely generated with at most $Q$ generators.
\end{cor}
\begin{proof}
  It follows by Theorem \ref{thm:currents_arbitrary_cones} and by
  \sync cor:assouad_bound.\ in \cite{deralb}.
\end{proof}
\par Note that a more general result, which fully exploits the
alternating property of metric currents, was obtained by Z\"ust
\cite[Prop.~2.5]{zust_phd} who showed that $\mcurr k,X.=\{0\}$ for
$k$ strictly larger than the
\emph{Nagata dimension} of the space $X$. The class of spaces with finite Nagata
dimension is larger than the class of spaces with finite Assouad
dimension and the Assouad dimension bounds the Nagata dimension from
above \cite[Thm.~1.1]{ledonne_rajala_nagata}.
\section{A representation formula}
\label{sec:rep_formula}
The goal of this Section is to prove Theorem \ref{thm:gl_curr_rep} and
the representation formula
\eqref{eq:gl_curr_rep_s1} which expresses metric currents in terms of
derivations. We will use some terminology and results from Subsection
\ref{subsec:ext_prod} where, roughly speaking, we construct the
exterior powers of the modules $\wder\mu.$ and $\wform\mu.$. The
dispirited reader may just want to think of expressions like
$D_1\wedge\cdots\wedge D_k$ and $df_1\wedge\cdots\wedge df_k$ as
analogues of measurable $k$-vectors and $k$-covectors fields and keep
in mind that as $\wder\mu.$ and $\wform\mu.$ are
$L^\infty(\mu)$-normed modules, their exterior products can be
constructed in three different categories: Banach spaces,
$L^\infty(\mu)$-modules and $L^\infty(\mu)$-normed modules.
\begin{rem}\label{rem:ext_pairings}
  We construct a bilinear pairing between the $L^\infty(\mu)$-normed
  modules $\lnmodext k,\mu,{\wder\mu.}.$ and $\lnmodext
  k,\mu,{\wform\mu.}.$; for notational simplicity, we will let $\hwder
  k,\mu.=\lnmodext k,\mu,{\wder\mu.}.$ and $\hwform k,\mu.=\lnmodext
  k,\mu,{\wform\mu.}.$. Consider the map:
  \begin{equation}
    \begin{aligned}
      \Phi:(\wder\mu.)^k\times(\wform\mu.)^k&\to L^\infty(\mu)\\
      \left((D_1,\cdots,D_k),(\omega_1,\cdots,\omega_k)\right)&\mapsto\det(\langle
      D_i,\omega_j\rangle)_{i,j=1}^k.
    \end{aligned}
  \end{equation}
  For a fixed $k$-tuple $\Omega=(\omega_1,\cdots,\omega_k)$, the map
  \begin{equation}
    \begin{aligned}
      \Phi_\Omega:(\wder\mu.)^k&\to L^\infty(\mu)\\
(D_1,\cdots,D_k)&\mapsto \Phi((D_1,\cdots,D_k),\Omega)
\end{aligned}
\end{equation}
is alternating $L^\infty(\mu)$-multilinear and satisfies the bound
\begin{equation}
  \label{eq:ext_pairings1}
  \begin{split}
    |\Phi_\Omega(D_1,\cdots,D_k)|&\le\sum_{\sigma\in\text{Perm}(k)}\prod_{i=1}^k
    |\langle D_{\sigma(i)},\omega_i\rangle|\\
    &\le k!\prod_{i=1}^k\locnorm D_i,{\wder\mu.}.\prod_{j=1}^k\locnorm \omega_j,{\wform\mu.}..
  \end{split}
\end{equation}
By the universal property of $\hwder k,\mu.$ we obtain an
$L^\infty(\mu)$-homomorphism $\hat\Phi_\Omega:\hwder k,\mu.\to
L^\infty(\mu)$. Note that the map
\begin{equation}
  \begin{aligned}
    \Psi:(\wform\mu.)^k&\to(\hwder k,\mu.)'\\
    \Omega&\mapsto \hat\Phi_\Omega
  \end{aligned}
\end{equation}
is an alternating $L^\infty(\mu)$-multilinear map with norm at most
$k!$ (by \eqref{eq:ext_pairings1}). By the universal property of
$\hwform k,\mu.$ we obtain a homomorphism $\hat\Psi:\hwform k,\mu.\to (\hwder
k,\mu.)'$ and thus an $L^\infty(\mu)$-bilinear pairing
\begin{equation}
  \label{eq:ext_pairings2}
  \begin{aligned}
    \langle\cdot,\cdot\rangle:\hwder k,\mu.\times\hwform k,\mu.&\to
    L^\infty(\mu)\\
    (\xi,\omega)&\mapsto \hat\Psi(\omega)(\xi),
  \end{aligned}
\end{equation}
satisfying
\begin{equation}
  \label{eq:ext_pairings3}
  |\langle\xi,\omega\rangle|\le k!\locnorm\xi,{\hwder k,\mu.}.\,\locnorm\omega,{\hwform k,\mu.}..
\end{equation}
\par By a similar argument, we can produce a pairing working in the
category of $L^\infty(\mu)$-modules:
\begin{equation}
  \label{eq:pairings4}
  \langle\cdot,\cdot\rangle:\lmodext k,\mu,{\wder\mu.}.\times
  \lmodext k,\mu,{\wform\mu.}.\to L^\infty(\mu)
\end{equation}
which is $L^\infty(\mu)$-bilinear and satisfies:
\begin{equation}
  \label{eq:pairings5}
  \|\langle\xi,\omega\rangle\|_{L^\infty(\mu)}\le k!\|\xi\|_{\lmodext k,\mu,{\wder\mu.}.}\|\omega\|_{\lmodext k,\mu,{\wform\mu.}.}.
\end{equation}
Working in the category of Banach spaces we can produce a pairing
\begin{equation}
  \label{eq:pairings6}
  \langle\cdot,\cdot\rangle:\banext k,{\wder\mu.}.\times
  \banext k,{\wform\mu.}.\to L^\infty(\mu)
\end{equation}
which is $\real$-bilinear an satisfies
\begin{equation}
  \label{eq:pairings7}
  \|\langle\xi,\omega\rangle\|_{L^\infty(\mu)}\le k!\|\xi\|_{\banext k,{\wder\mu.}.}\|\omega\|_{\banext k,{\wform\mu.}.}.
\end{equation}
\par Note that given $(D_1,\cdots, D_k)\in(\wder\mu.)^k$, we can
regard $D_1\wedge\cdots\wedge D_k$ as either an element of $\hwder
k,\mu.$, or of $\lmodext k,\mu,{\wder\mu.}.$ or of $\banext
k,{\wder\mu.}.$. In the sequel, unless specified all three
possibilities are admitted. A similar observation can be applied to an
expression $df_1\wedge\cdots\wedge df_k$ where
$(f_1,\cdots,f_k)\in(\lipfun X.)^k$ and to a pairing $\langle
D_1\wedge\cdots\wedge D_k, df_1\wedge\cdots\wedge df_k\rangle$.
\end{rem}
\par We now prove the local version of Theorem \ref{thm:gl_curr_rep}:
\begin{lem}\label{lem:loc_curr_rep}
  For $T\in\mcurr k,X.$, suppose that the module $\wder{\cmass T.}.$
  is free on the derivations
  $\{D_i\}_{i=1}^N$ which have pseudodual functions
  $\{g_i\}_{i=1}^N\subset\lipalg X.$. Then there are
  $\{\lambda_a\}_{a\in\Lambda_{k,N}}\subset L^\infty(\cmass T.)$ such
  that:
  \begin{equation}\label{eq:loc_curr_rep}
    T(f,\pi_1,\cdots,\pi_k)=\sum_{a\in\Lambda_{k,N}}\int_Xf\lambda_a\left\langle
      D_{a_1}\wedge\cdots\wedge
      D_{a_k},d\pi_1\wedge\cdots\wedge d\pi_k\right\rangle\,d\cmass T.,
  \end{equation}
  where $\Lambda_{k,N}$ denotes the set of ordered $k$-tuples
  consisting of distinct elements of $\{1,\cdots,N\}$.
\end{lem}
\begin{proof}
  Recall from the discussion soon after~\cite[Eq.~2.5]{ambrosio-kirch}
  that for a metric current $T$ the first argument $f$ can be taken to
  be a bounded Borel function or an element of $L^\infty(\cmass
  T.)$. Therefore we will assume that $f\in L^\infty(\cmass T.)$ with
  $|f|\le 1$
    and that each $\pi_i$ is $1$-Lipschitz. Let $\omega=(f,\pi_1,\cdots,\pi_{k-1})$
  so that the current $T\crest\omega\in\mcurr 1,X.$ satisfies $\cmass
  T\crest\omega.\ll\cmass T.$ by \cite[Eq.~2.5]{ambrosio-kirch}. In
  particular, we can also regard $f$ as an element of $L^\infty(\cmass
  {T\crest\omega}.)$ as there is a natural homomorphism
  $L^\infty(\cmass T.)\to L^\infty(\cmass
  {T\crest\omega}.)$ obtained by restricting each $h\in L^\infty(\cmass
  T.)$ to the set where $d\cmass {T\crest\omega}./d\cmass T.\ne0$ (if
  such a set is empty then the measure $\cmass {T\crest\omega}.$ is
  trivial so $L^\infty(\cmass
  {T\crest\omega}.)$ is also trivial).
  \par By
  Theorem \ref{thm:one_currents_derivations} we have:
  \begin{equation}
    T(f,\pi_1,\cdots,\pi_k)=T\crest\omega(\pi_k)=\int_X
    D_{T\crest\omega}\pi_k\,d\cmass T.
  \end{equation}
where $D_{T\crest\omega}=\Der{\cmass T.}.({T\crest\omega})$ is the derivation associated to the
$1$-dimensional current $T\crest\omega$. 
\par By assumption there are bounded Borel functions
$\{\lambda_i\}_{i=1}^N\subset\bborel X.$: 
\begin{equation}
 D_{T\crest\omega}=\sum_{i=1}^N\lambda_iD_i.
\end{equation}
Note also that as the $\{g_i\}_{i=1}^N$ are pseudodual to the
$\{D_i\}_{i=1}^N$ we have $\lambda_i = D_{T\crest\omega}g_i$.
\par We now get:
\begin{equation}
  \label{eq:mina1}
  \begin{split}
    T\crest \omega(f,\pi_k) &= \int_Xf D_{T\crest\omega}\pi_k\,d\cmass
    T\crest\omega. = \sum_{j=1}^N\int_X f D_{T\crest
      \omega}g_j\,D_j\pi_k\,d\cmass T\crest\omega.\\
    &=\sum_{j=1}^NT\crest\omega(fD_j\pi_k,g_j),
  \end{split}
\end{equation}
which establishes:
\begin{equation}\label{eq:loc_curr_rep_p1}
  T(f,\pi_1,\cdots,\pi_k)=\sum_{j=1}^NT(fD_j\pi_k,\pi_1,\cdots,\pi_{k-1},g_j).
\end{equation}
If $\Lambda'_{k,N}$ denotes the set of $k$-tuples on $\{1,\cdots,N\}$,
by using induction in \eqref{eq:loc_curr_rep_p1},
\begin{equation}
  T(f,\pi_1,\cdots,\pi_k)=\sum_{a\in \Lambda'_{k,N}}T(fD_{a_1}\pi_1\cdots D_{a_k}\pi_k,g_{a_1},\cdots,g_{a_k});
\end{equation}
as currents are alternating
\begin{equation}
  T(f,\pi_1,\cdots,\pi_k)=\sum_{a\in \Lambda_{k,N}}T(f\langle
    D_{a_1}\wedge\cdots\wedge
    D_{a_k},d\pi_1\wedge\cdots\wedge d\pi_k\rangle,g_{a_1},\cdots,g_{a_k});
\end{equation}
the map $\psi\in L^1(\cmass T.)\mapsto T(\psi,g_{a_1},\cdots,g_{a_k})$ defines
a linear functional on $L^1(\cmass T.)$ which is represented by some $\lambda_a\in
L^\infty(\cmass T.)$ by the Riesz representation Theorem. We conclude that:
\begin{equation}
    T(f,\pi_1,\cdots,\pi_k)=\sum_{a\in \Lambda_{k,N}}\int_Xf\lambda_a\left\langle
    D_{a_1}\wedge\cdots\wedge
    D_{a_k},d\pi_1\wedge\cdots\wedge d\pi_k\right\rangle\,d\cmass T..
\end{equation}
\end{proof}
We now prove Theorem \ref{thm:gl_curr_rep}:
\begin{proof}[Proof of Theorem \ref{thm:gl_curr_rep}]
  Suppose that $\wder{\cmass T.}.$ has $N$ generators; then by Theorem
  \ref{thm:free_dec} there is an $L^\infty({\cmass T.})$-Borel
  partition of unity $\{U_\beta\}_{\beta\in J}$ such that $J$ is
  finite with at most $N$ elements and $\wder{\cmass T.\mrest
    U_\beta}.$ is free of rank $N_\beta\le N$. Having selected a local
  basis of derivations for each $U_\beta$, we can apply Lemma
  \ref{lem:quantitative_pseudodual} to obtain an $L^\infty({\cmass
    T.})$-Borel partition of unity $\{V_\alpha\}$ such that:
\begin{itemize}
\item The module  $\wder{\cmass T.\mrest
    V_\alpha}.$ has a basis $\{D_{\alpha,i}\}_{i=1}^{N_\alpha}$.
  \item The norms of the derivations $\{D_{\alpha,i}\}_{i=1}^{N_\alpha}$ are
    bounded by a universal constant $C(N)$.
\item There are $1$-Lipschitz functions
  $\{g_{\alpha,j}\}_{j=1}^{N_\alpha}$ pseudodual to the derivations
  $\{D_{\alpha,i}\}_{i=1}^{N_\alpha}$ on $V_\alpha$.
\end{itemize}
The hypotheses of
  Lemma~\ref{lem:loc_curr_rep} are met by the currents $\{T\crest
  V_\alpha\}$ and we have local representations:
  \begin{equation}\label{eq:gl_curr_rep_p1}
    T\crest V_\alpha(f,\pi_1,\cdots,\pi_k)=\sum_{a\in \Lambda_{k,N_\alpha}}\int_{V_\alpha}
    f\lambda_{\alpha, a}\left\langle D_{\alpha,a_1}\wedge\cdots\wedge
    D_{\alpha,a_k}, d\pi_1\wedge\cdots\wedge d\pi_k\right\rangle
  \,d\cmass T.;
\end{equation}
 for any subset $W\subset V_\alpha$ and any index $a\in \Lambda_{k,N_\alpha}$, letting $\pi_i=g_{\alpha,
    a_i}$, we obtain from \eqref{eq:gl_curr_rep_p1} the lower bound
  \begin{equation}
    \cmass T.(W)\ge T\crest
    V_\alpha(\chi_W,g_{\alpha,a_1},\cdots,g_{\alpha,a_k})=
    \int_W\lambda_{\alpha,a}\,d\cmass T.,
  \end{equation} which implies the upper bound
  $\|\lambda_{\alpha,a}\|_{L^\infty(\cmass T.\crest V_\alpha)}\le 1$.
\par Note that we can regard $\Lambda_{k,N_\alpha}$ as a subset of
$\Lambda_{k,N}$ and for
$a_i\in\{1,\cdots,N\}\setminus\{1,\cdots,N_\alpha\}$ we will
improperly use the notation $D_{\alpha,a_i}$ to denote the trivial
derivations. Similarly, if some entry $a_i$ of $a$ is $>N_\alpha$ we
will let $\lambda_{\alpha,a}=0$. Note that
  \begin{align}
    D_{a_i}&=\sum_{\alpha}\chi_{V_\alpha}D_{\alpha,a_i},\quad{\text{($1\le
        i\le
        k$)}}
  \end{align}
  define elements of $\wder{\cmass T.}.$ with norm bounded by
  $C(N)$, and that
  \begin{equation}
    \label{eq:mina2}
    \lambda_{a}=\sum_{\alpha}\chi_{V_\alpha}\lambda_{\alpha,a}
  \end{equation}
  define elements of $L^\infty(\cmass T.)$ of norm at most $1$.
  Therefore
\begin{equation}
  \omega_T=\sum_{a\in \Lambda_{k,N}}\lambda_aD_{a_1}\wedge\cdots\wedge D_{a_k}
\end{equation} defines an element of $\hwder k,{\cmass T.}.$
 with norm at most $(C(N))^k\binom{N}{k}$. By Remark
 \ref{rem:ext_pairings} one can also regard $\omega_T$ as an element
 of either $\lmodext k,{\cmass T.},{\wder{\cmass T.}.}.$ or $\banext
 k,{\wder{\cmass T.}.}.$. 
\par We now observe that:
\begin{equation}
  \begin{split}
    T(f,\pi_1,\cdots,\pi_k)&=\sum_\alpha\left(T\crest
    V_\alpha\right)(f,\pi_1,\cdots,\pi_k)\\
&=\sum_\alpha\sum_{a\in \Lambda_{k,N_\alpha}}\int_{V_\alpha}
    f\lambda_{\alpha, a}\left\langle D_{\alpha,a_1}\wedge\cdots\wedge
    D_{\alpha,a_k}, d\pi_1\wedge\cdots\wedge d\pi_k\right\rangle
\,d\cmass T.\\
&=\sum_\alpha\sum_{a\in \Lambda_{k,N}}\int_{V_\alpha}f\left\langle
D_{a_1}\wedge\cdots\wedge D_{a_k},
d\pi_1\wedge\cdots\wedge d\pi_k\right\rangle\, d\cmass T.\\
&=\sum_\alpha\int_Xf\chi_{V_\alpha}\langle\omega_T,d\pi_1\wedge\cdots\wedge
d\pi_k\rangle\,d\cmass
T.\\
&=\int_Xf\,\langle\omega_T,d\pi_1\wedge\cdots\wedge d\pi_k\rangle\,d\cmass
T.,
  \end{split}
\end{equation}
which proves \eqref{eq:gl_curr_rep_s1}.
\end{proof}
\begin{rem}\label{rem:curr_extension}
  A consequence of Theorem \ref{thm:gl_curr_rep} is that one can
  regard a $k$-dimensional metric current $T$ as a map defined on
  $L^1({\cmass T.})\times\hwform k,{\cmass T.}.$. Moreover, noting
  that if $T\in\mcurr k,\mu.$ one can regard $L^\infty({\cmass T.})$
  (respectively $\hwder k,{\cmass T.}.$, $\hwform k,{\cmass T.}.$) as
  submodules of $L^\infty(\mu)$ (respectively $\hwder k,\mu.$,
  $\hwform k,\mu.$), the current $T$ can be viewed as a map defined on
  $L^\infty(\mu)\times \hwform k,\mu.$ and one can take
  $\omega_T\in\hwder k,\mu.$.
\end{rem}
\begin{rem}\label{rem:williams}
  Note that Theorem \ref{thm:gl_curr_rep} implies
  \cite[Thm.~1.3]{williams-currents}. In fact, if $(X,\mu)$ is a
  differentiability space, by \sync partialderivatives.\  in
  \cite{deralb} the module $\wder\mu.$
  can be identified with the set of bounded measurable sections of the
  Cheeger's measurable tangent bundle $T_\mu X$ (defined in
\cite[pg.~463]{cheeger99}). Then the module $\hwder
  k,\mu.$ coincides with the set of bounded measurable sections of the
  $k$-th exterior power of $T_\mu X$; in this way, we recover
\cite[Thm.~1.3]{williams-currents}.
\end{rem}
For $k\ge 2$, it is not clear how to identify the elements of $\hwder
k,\mu.$ which give rise to currents. However, we have a partial result
concerning normal currents. We start by generalizing the notion of
\emph{precurrents} which was introduced by Williams in the context of
differentiability spaces.
\begin{defn}
  \label{defn:precurrents}
  Suppose that $\mu$ is a finite Radon measure on $X$. Then each $\xi\in\hwder k,\mu.$ defines a $k$-metric functional $T_\xi$
  by:
  \begin{equation}
    \label{eq:precurrents1}
    T_\xi(f,\pi_1,\cdots,\pi_k)=\int_Xf\langle
    \xi,d\pi_1\wedge\cdots\wedge d\pi_k\rangle\,d\mu;
  \end{equation}
  moreover, $T_\xi$ is multilinear in the arguments
  $(f,\pi_1,\cdots,\pi_k)$ and alternating in the arguments
  $(\pi_1,\cdots,\pi_k)$. Note also that \eqref{eq:ext_pairings3} implies
  that $T_\xi$ has finite mass:
  \begin{equation}
    \label{eq:precurrents2}
    \cmass T_\xi.\le k!\,\locnorm \xi,{\hwder k,\mu.}.\,\mu.
  \end{equation}
  We also have that $T_\xi$ is local in the sense that if
  \begin{equation}
    \label{eq:precurrents3}
    \left\{x:\locnorm \xi,{\hwder k,\mu.}.(x)\ne0\right\}\subset\bigcup_{\alpha=1}^kV_\alpha,
  \end{equation}
  where the $V_\alpha$ are Borel sets with $\pi_\alpha$ constant on
  $V_\alpha$, then
  \begin{equation}
    \label{eq:precurrents4}
    T_\xi(f,\pi_1,\cdots,\pi_k)=0.
  \end{equation}
  In fact, by Theorem \ref{thm:ext_pow_normed_module}, for each
  $\varepsilon>0$ we can find $\xi'\in\hwder k,\mu.$ of the form
  \begin{equation}
    \label{eq:precurrents5}
    \xi'=\sum_{i\in I_\xi}D_{i_1}\wedge\cdots\wedge D_{i_k}
  \end{equation}
  with $\|\xi-\xi'\|_{\hwder k,\mu.}\le\varepsilon$. Then
  \eqref{eq:precurrents4} follows because for each $D\in\wder\mu.$,
  $\chi_{V_\alpha}D\pi_\alpha=0$.
  \par We will call $T_\xi$ the \textbf{$k$-precurrent associated to
    $\xi$} and we will denote by $\pcurr k,\mu.$ the set of $k$-precurrents.
\end{defn}
\begin{thm}
  \label{thm:normal_currs}
  Given $\xi\in\hwder k,\mu.$, if the metric functional $\partial
  T_\xi$ has finite mass, then $T_\xi$ is a normal current. If
  $\wder\mu.$ is finitely generated, the set $\ncurr k,\mu.$, which
  consists of the normal currents whose mass is absolutely continuous
  with respect to $\mu$, coincides with the set of those
  $T_\xi\in\pcurr k,\mu.$ whose boundary $\partial T_\xi$ has finite mass.
\end{thm}
\begin{proof}
  [Proof of Theorem \ref{thm:normal_currs}]
  Assume that the metric functional $\partial T_\xi$ has finite mass. In order to show that
  $T_\xi$ is a metric current, it suffices to check the continuity
  axiom (4) in Definition \ref{defn:metric_currents}. Suppose that $f_h\xrightarrow{\text{w*}}f$ and
  $\pi_{i,h}\xrightarrow{\text{w*}}\pi_i$ for all $1\le i\le k$. Note
  that:
  \begin{equation}
    \label{eq:normal_currs_p1}
    \left|T_\xi(f_h,\pi_{1,h},\cdots,\pi_{k,h})-T_\xi(f,\pi_{1,h},\cdots,\pi_{k,h})\right|\le\prod_{i=1}^k\glip
    \pi_{i,h}.\int_X |f_h-f|\,d\cmass T_\xi.
  \end{equation}
  so that
  \begin{equation}
    \label{eq:normal_currs_p2}
    \lim_{h\to\infty}\left|T_\xi(f_h,\pi_{1,h},\cdots,\pi_{k,h})-T_\xi(f,\pi_{1,h},\cdots,\pi_{k,h})\right|=0.
  \end{equation}
  Moreover, we have:
  \begin{equation}
    \label{eq:normal_currs_p3}    
\begin{split}
  T_\xi(f,\pi_{1,h},\pi_{2,h},\cdots,\pi_{k,h})&-T_\xi(f,\pi_1,\pi_{2,h},\cdots,\pi_{k,h})\\&=\partial
  T_\xi(f(\pi_{1,h}-\pi_1),\pi_{2,h},\cdots,\pi_{k,h})\\&-T_\xi(\pi_{1,h}-\pi_1,f,\pi_{2,h},\cdots,\pi_{k,h});
\end{split}
\end{equation}
as
\begin{align}\label{eq:normal_currs_p4}
  \left|\partial
  T_\xi(f(\pi_{1,h}-\pi_1),\pi_{2,h},\cdots,\pi_{k,h})\right|&\le \prod_{i=2}^k\glip
  \pi_{i,h}.\int_X |f(\pi_{1,h}-\pi_1)|\,d\cmass\partial
  T_\xi.,\\\label{eq:normal_currs_p5}
\left|T_\xi(\pi_{1,h}-\pi_1,f,\pi_{2,h},\cdots,\pi_{k,h})\right|&\le
\glip f.\,\prod_{i=2}^k\glip
  \pi_{i,h}.\int_X|\pi_{1,h}-\pi_1|\,d\cmass T_\xi.,
\end{align}
from \eqref{eq:normal_currs_p3} we have:
\begin{equation}
  \label{eq:normal_currs_p6}
  \lim_{h\to\infty}\left|T_\xi(f,\pi_{1,h},\pi_{2,h},\cdots,\pi_{k,h})-T_\xi(f,\pi_1,\pi_{2,h},\cdots,\pi_{k,h})\right|=0.
\end{equation}
Using that $T_\xi$ is alternating in the last $k$ arguments and
induction in $i$, the previous argument gives:
\begin{equation}
  \label{eq:normal_currs_p7}
  \lim_{h\to\infty}\left|T_\xi(f_h,\pi_{1,h},\pi_{2,h},\cdots,\pi_{k,h})-T_\xi(f,\pi_1,\pi_{2},\cdots,\pi_{k})\right|=0,
\end{equation}
which shows that $T_\xi$ is a metric current. As $\partial T_\xi$ has
finite mass, the current $T_\xi$ is normal. The second part of this Theorem follows
from Theorem \ref{thm:gl_curr_rep} because, if $\wder\mu.$ is finitely
generated, any metric current is a precurrent.
\end{proof}
\section{Applications}
\label{sec:applications}
\subsection{Approximation of $1$-currents by Normal currents}
\label{subsec:normal_currents}
\par The goal of this Subsection is to prove Theorem
\ref{thm:amb_kirch_one}. We make the set theoretic assumption that the
cardinality of any set is an Ulam number so that by
\cite[Lem~2.9]{ambrosio-kirch} the masses of metric currents
are concentrated on countable unions of compact sets. This assumption
is not needed if we consider currents in separable Banach spaces.
\par Let $\curves(X)$ denote the
set of Lipschitz maps from $[0,1]$ to $X$ topologized as a subspace of
$K([0,1]\times X)$. To each $\gamma\in\curves(X)$, one can then associate a
normal current $[\gamma]$ by letting:
\begin{equation}
  \label{eq:constdir_curve_curr}
  [\gamma](fd\pi)=\int_0^1(f\circ\gamma)(t)(\pi\circ\gamma)'(t)\,dt\quad\left((f,\pi)\in\bborel
    X.\times\lipfun X.\right).
\end{equation}
Note that the mass measure of $[\gamma]$ can be bounded by:
\begin{equation}
  \label{eq:constdir_mass}
  \cmass[\gamma].\le\mpush\gamma.\left(\metdiff\gamma\,\cdot\lebmeas\mrest[0,1]\right).
\end{equation}
\begin{lem}\label{lem:pseudodual_banach}
  Let $Z$ be a Banach space and $\mu$ a $\sigma$-finite Radon measure on $Z$. Suppose
  that the derivations $\{D_i\}_{i=1}^k\subset\wder\mu.$ are
  independent. Then there is a Borel $L^\infty(\mu)$-partition of
  unity $V_\alpha$ and there are, for each $\alpha$, derivations
  $\{D_{\alpha,i}\}_{i=1}^k\subset\chi_{V_\alpha}\wder\mu.$ and unit
  norm functionals $\{z^*_{\alpha,j}\}_{j=1}^k\subset Z^*$ such that:
  \begin{enumerate}
  \item The submodule of $\wder\mu.$ generated by the derivations
    $\{D_{\alpha,i}\}_{i=1}^k$ is the same as the submodule generated
    by the derivations
    $\{\chi_{V_\alpha}D_i\}_{i=1}^k$;
      \item The functionals $\{z^*_{\alpha,j}\}_{j=1}^k$ are
        pseudodual to the derivations
    $\{D_{\alpha,i}\}_{i=1}^k$ on $V_\alpha$.
  \end{enumerate}
\end{lem}
\begin{proof}
  Note that $\mu$ is concentrated on a $K_\sigma$-set, i.e.~a
  countable union of compact sets; in particular, $\spt\mu$ is
  separable and we can assume that $Z$ is separable by taking the
  closure of the linear span of $\spt\mu$. Up to passing to a Borel
  $L^\infty(\mu)$-partition of unity we can assume that $Z$ is also bounded. Let
  $\{z_i\}\subset Z$ be a countable dense set and for $i\ne j$ choose a
  unit norm linear functional $z^*_{i,j}$ with
  $\langle z^*_{i,j},z_i-z_j\rangle=\|z_i-z_j\|_Z$. By the
  Stone-Weierstrass Theorem for Lipschitz Algebras
  \cite[Cor.~4.1.9]{weaver_book99}, the family of functionals $\{z^*_{i,j}\}_{i,j}$ is a
  countable generating set\footnote{i.e.~for each $f\in\lipalg Z.$
    there is a sequence of polynomials $\{P_n\}\subset\lipalg Z.$ in
    the $z^*_{i,j}$ 
    with $P_n\xrightarrow{\text{w*}}f$.} for $\lipalg Z.$. By
 \cite[Prop.~2.35]{derivdiff} we can find a Borel
 $L^\infty(\mu)$-partition of unity $\{V_\alpha\}$ and for each
 $\alpha$ unit functionals $\{z^*_{\alpha,j}\}_{j=1}^k$ such that, letting
 $M_\alpha=(D_iz^*_{\alpha,j})_{i,j=1}^k$, we have $\det M_\alpha\ne0$ on
 $V_\alpha$. Up to passing to a further Borel partition we can assume
 that for each $\alpha$ there is a $\delta_\alpha>0$ such that:
 \begin{equation}
   \label{eq:pseudodual_banach_p1}
   \left | \det M_\alpha(x) \right
   |\in(\delta_\alpha,2\delta_\alpha)\quad(\forall x\in V_\alpha);
 \end{equation}
 we then let $D_{\alpha,i}=\sum_{j=1}^k(M_\alpha^{-1})_{i,j}D_j$.
\end{proof}
\begin{proof}[Proof of Theorem \ref{thm:amb_kirch_one}] We make the
  following preliminary Observation ({\bf Obs1}): suppose that
  $\sum_kT_k$ is either a finite sum of $1$-currents or a series with
  \begin{equation}
    \sum_k\|T_k\|_{\mcurr 1,Z.}<\infty,
  \end{equation} and suppose also that for each $n$ there is a sequence of
  normal currents $\{N_{k,n}\}\subset\ncurr 1,Z.$ such that
    \begin{equation}
      \lim_{n\to\infty}\|T_k-N_{k,n}\|_{\mcurr 1,Z.}=0;
    \end{equation} then, if we let $T=\sum_kT_k$, there is a sequence of
  normal currents $\{N_n\}\subset\ncurr 1,Z.$  such that
  \eqref{eq:amb_kirch_one_s1} holds.
  \par As $\wder{\cmass T.}.$ is finitely generated, by Theorem
  \ref{thm:free_dec} and ({\bf Obs1}) we can reduce to the case in
  which $\wder{\cmass T.}.$ is free of rank $N$. Applying Lemma
  \ref{lem:pseudodual_banach} and ({\bf Obs1}), we can assume that
  $\wder{\cmass T.}.$ has a basis consisting of derivations
  $\{D_i\}_{i=1}^N$ such that there are unit norm linear functionals
  $\{z^*_j\}_{j=1}^N$ which are pseudodual to the
  $\{D_i\}_{i=1}^N$. Let $z^*=(z^*_j)_{j=1}^N$ and $\{e_i\}_{i=1}^N$
  the standard basis of $\real^N$; by Corollary
  \ref{cor:mu_arb_cone_special} for any $\alpha\in(0,\pi/2)$ the
  measure $\cmass T.$ admits $C$-Lipschitz Alberti representations
  $\{\albrep i.\}_{i=1}^N$ with $\albrep i.$ in the $z^*$-direction of
  $\cone(e_i,\alpha)$ (and with positive $z_i^*$-speed); note that, up
  to taking an $L^\infty(\cmass T.)$-partition of unity and choosing
  $\alpha$ sufficiently small, we can assume that the derivations
  $\{D_{\albrep i.}\}_{i=1}^N$ form a basis of $\wder{\cmass
    T.}.$. Applying Theorem \ref{thm:alberti_banach}, we can assume
  that $\albrep i.=(P_i,\nu_i)$ with $\spt P_i\subset\curves(Z)$ and
  $(\nu_i)_\gamma=h_i\Psi_\gamma$, where $h_i$ is a Borel function on
  $Z$ and $\Psi_\gamma=\mpush\gamma.\lebmeas\mrest[0,1]$. Denoting the
  derivation $\Der{\cmass T.}.(T)$ by $D_T$, there are bounded Borel
  functions $\{\lambda_i\}_{i=1}^N$ such that
  $D_T=\sum_{i=1}^N\lambda_iD_i$; but this implies that
 \begin{equation}
   T=\sum_{i=1}^N\Cur{\cmass T.}.(\lambda_iD_i),
 \end{equation} and by ({\bf Obs1}) we reduce to the case
 in which $T=\lambda D_{\albrep.}$ where $\lambda$ is a bounded Borel
 function and $\albrep.=(P,\nu)$ is a
 $C$-Lipschitz Alberti representation with $\spt
 P\subset \curves(Z)$ and $\nu_\gamma=h\Psi_\gamma$. Let $\mu$
 denote the measure
 \begin{equation}
   \mu=\int_{\curves(Z)}\Psi_\gamma;
 \end{equation} note that $\cmass T.\ll\mu$ and $h\lambda\in
 L^1(\mu)$; as $\lipalg Z.$ is dense in $L^1(\mu)$, we can find, for
 each $\epsi>0$, a function
 $g\in\lipalg Z.$ such that:
 \begin{equation}\label{eq:amb_kirch_one_p1}
   \|g-h\lambda\|_{L^1(\mu)}\le\epsi.
 \end{equation} 
 Note that the metric current $N$ defined by
 \begin{multline}
   N(fd\pi)=\int_{\curves(Z)}dP(\gamma)\int_\gamma
   f\partial_\gamma\pi\,d\Psi_\gamma\\=
   \int_{\curves(Z)}dP(\gamma)\int_{[0,1]}f\circ\gamma(t)\,(\pi\circ\gamma)'(t)\,dt
 \end{multline} is normal and so $N\crest g$ is normal. However,
 \eqref{eq:amb_kirch_one_p1} implies that
 \begin{equation}
   \|N\crest g-T\|_{\mcurr 1,Z.}\le C\|g-h\lambda\|_{L^1(\mu)}\le C\epsi.
 \end{equation}
\end{proof}
\subsection{Alberti representations with constant directions}
\label{subsec:const_dir}
In this Subsection we illustrate a different method to produce Alberti
representations. This method allows to refine the way in which the
direction is specified. In fact, the cone field is replaced by a
vector field and one can also use countably many Lipschitz functions.
This method relies on results of \cite{paolini_acyclic,paolini_one_normal} on
the structure of $1$-dimensional normal currents.
\par We state the Paolini-Stepanov decomposition of normal currents
using parametrized curves: note, however, that in \cite{paolini_one_normal}
the result is stated using non-parametrized curves. Recall also that
the metric space $X$ is assumed Polish. 
\begin{thm}[Corollary 3.3 in \cite{paolini_one_normal}]
  \label{thm:ps-dec}
  Let $N$ be a $1$-dimensional normal current defined on $X$; then there is a finite
  Radon measure $\eta$ on the space $K([0,1]\times X)$ of compact subsets of $[0,1]\times
  X$ which is concentrated on
  $\curves(X)$, and such that:
  \begin{align}
    \label{eq:ps-dec_s1}
    N&=\int_{\curves(X)}[\gamma]\,d\eta(\gamma);\\
    \label{eq:ps-dec_s2}
    \cmass N.&=\int_{\curves(X)}\cmass[\gamma].\,d\eta(\gamma);\\
    \label{eq:ps-dec_s3}
    \cmass N.(X)&=\int_{\curves(X)}l(\gamma)\,d\eta(\gamma),
  \end{align}
  where $l(\gamma)$ denotes the length of $\gamma$ which is given
  by:
  \begin{equation}
    \label{eq:ps-dec_s4}
    l(\gamma)=\int_0^1\metdiff\gamma(t)\,dt.
  \end{equation}
\end{thm}
Note that the integrals in \eqref{eq:ps-dec_s1} and
\eqref{eq:ps-dec_s2} make sense because the maps
$\gamma\mapsto[\gamma]$ and $\gamma\mapsto\cmass[\gamma].$ are Borel
in the following sense: for each $(f,\pi)\in\bborel
    X.\times\lipfun X.$ and each Borel $E\subset X$, the maps $\gamma\mapsto[\gamma](fd\pi)$ and
    $\gamma\mapsto\cmass\gamma.(E)$ are Borel.
We need to introduce more terminology:
\begin{defn}
  \label{defn:pieces}
  The set of maps $\gamma\in\curves(X)$ with Lipschitz constant at
  most $n$ is a Polish space and is denoted by $\curves_n(X)$. The set
  of Lipschitz maps $\gamma:K\to X$, where $K$ is a nonempty compact
  subset of $[0,1]$, is denoted by $\pieces(X)$ and topologized as a
  subset of $K([0,1]\times X)$. Note that $\pieces (X)$ is a subset of
  $\frags(X)$ and a Borel subset of $K([0,1]\times X)$. The subset of
  maps $\gamma\in\pieces (X)$ with Lipschitz constant at most $n$ is a
  Polish space and is denoted by $\pieces_n(X)$. If
  $(\gamma,\tilde\gamma)\in\curves(X)\times\pieces(X)$ and
  $\gamma|\dom\tilde\gamma=\tilde\gamma$, we say that $\tilde\gamma$
  is \textbf{a piece} of $\gamma$.
\end{defn}
To each $\gamma\in\pieces(X)$, one can
associate a metric current $[\gamma]$ by letting:
\begin{equation}
  \label{eq:constdir_piece_curr}
  [\gamma](fd\pi)=\int_{\dom\gamma}(f\circ\gamma)(t)(\pi\circ\gamma)'(t)\,dt\quad\left((f,\pi)\in\bborel
    X.\times\lipfun X.\right);
\end{equation}
a  modification of the argument in \sync lem:meas_fix.\ 
  in \cite{deralb}
shows that, for each $(f,\pi)\in\bborel
    X.\times\lipfun X.$, the map 
\begin{equation}
  \label{eq:borel_assig}
  \begin{aligned}
    \pieces(X)&\to\real\\
    \gamma&\mapsto[\gamma](fd\pi)
  \end{aligned}
\end{equation}
is Borel. Having fixed an open set $U\subset X$, there is a countable
collection $\mathcal{F}_U$ of $1$-forms $\omega=\sum_if_id\pi_i$ such that, for each
$\gamma\in\pieces(X)$,
\begin{equation}
  \label{eq:mass_borel}
  \cmass\gamma.(U)=\sup_{\omega\in\mathcal{F}_U}[\gamma](\omega);
\end{equation} this implies that, for each Borel $E\subset X$, the map:
\begin{equation}
  \label{eq:mass_borel2}
  \begin{aligned}
    \pieces(X)&\to[0,\infty)\\
    \gamma&\mapsto\cmass[\gamma].(E)
  \end{aligned}
\end{equation} is Borel.
Note also that the mass of the current associated to
$\gamma\in\pieces(X)$ can be bounded from above similarly as in 
\eqref{eq:constdir_mass}:
\begin{equation}
  \label{eq:constdir_mass_piece}
  \cmass[\gamma].\le\mpush\gamma.\left(\metdiff\gamma\,\cdot\lebmeas\mrest\dom\gamma\right).
\end{equation}
\par We now discuss the notion of Alberti representations in the
direction of a vector field $v$. In greater generality, we consider
$l^2$-valued Lipschitz maps, where $l^2$ is the Hilbert space of
$l^2$-summable sequences. In the following, we let $\real^\infty$
denote the product of countably many copies of $\real$ with the
product topology. Note that any map $F:X\to l^2$ is determined by its
components $F_i$; in particular, if $F$ is Lipschitz and
$D\in\wder\mu.$, we can choose a Borel representative of each $DF_i$
and denote by $DF$ the Borel map $DF:X\to\real^\infty$ whose $i$-th
component is $DF_i$. Moreover, we can stipulate that the maps
$DF_i:X\to\real$ are uniformly bounded, with the bound indepedent of
$i$. In the following, this will always be assumed when we apply a
derivation $D\in\wder\mu.$ to a Lipschitz function $F:X\to l^2$. We
finally call a Borel map $v:X\to\real^\infty$, such that the
components $v_i$ are uniformly bounded by some $C>0$, a \textbf{vector
  field}.
\par In connection with the idea of using countably many Lipschitz
maps to control derivations, we point out that the idea has been used
independently by Ambrosio and Trevisan~\cite{ambrosio-lflow} in the study of ODEs
associated to derivations. Note however, that here we consider
derivations with a lesser degree of regularity.
\begin{defn}
  \label{defn:const_alberti}
  Let $F:X\to l^2$ be Lipschitz and $v:X\to\real^\infty$ a vector
  field. Denote by $N_v$ the set where $v$ vanishes:
  \begin{equation}
    \label{eq:const_alberti1}
    N_v=\left\{x\in X: v(x)=0\right\}.
  \end{equation} We say that the Alberti representation $\albrep.=(P,\nu)$ of $\mu\mrest(X\setminus N_v)$ is in the $F$-direction of
  $v$ if for $P$-a.e.~$\gamma$ and $\lebmeas$-a.e.~$t\in\dom\gamma$
  there is a $\lambda=\lambda(\gamma,t)>0$ such that:
  \begin{equation}
    \label{eq:defnconst_alberti2}
    (F\circ\gamma)'(t)=\lambda v\left(\gamma(t)\right).
  \end{equation}
\end{defn}
Given a Lipschitz map $F:X\to l^2$, to produce vector fields $v$ with
$\mu\mrest(X\setminus N_v)$ admitting an Alberti representation in the
$F$-direction of $v$, we will use a special class of derivations.
\begin{defn}
  \label{defn:normal_der}
  A derivation $D\in\wder\mu.$ is called \textbf{normal} if there is
  a Borel $L^\infty(\mu\mrest(X\setminus N_D))$-partition of unity
  $\{U_\alpha\}$ such that for each $\alpha$ there are:
  \begin{enumerate}
  \item An isometric embedding $\iota_\alpha:U_\alpha\to Z_\alpha$ where
    $Z_\alpha$ is a Polish space.
  \item A normal current $N_\alpha$ in $Z_\alpha$ with $\mpush
    \iota_\alpha.(\mu\mrest U_\alpha)\ll\cmass N_\alpha.$.
  \item Denoting by $D_N\in\wder{\cmass N_\alpha.}.$ the derivation
    associated to $N_\alpha$ given by Theorem
    \ref{thm:one_currents_derivations}, there is $\lambda_\alpha\in
    L^\infty(\cmass N_\alpha.)$ with $\lambda_\alpha\ge0$ and
    \begin{equation}
      \label{eq:normal_der1}
      \mpush \iota_\alpha.\chi_{U_\alpha}D=\lambda_\alpha D_{N_\alpha}.
    \end{equation}
  \end{enumerate}
Note that in \eqref{eq:normal_der1} we have used that (2) allows to
identify $\mpush \iota_\alpha.D$ with a derivation in $\wder{\cmass
  N_\alpha.}.$.
\end{defn}
\begin{rem}
  \label{rem:normal_derivations}
  We want to remark that there are many normal derivations. Suppose
  that $\mu$ admits an Alberti representation in the $f$-direction of
  an $n$-dimensional cone field $\cone$. The proof of \sync
  alberti_rep_prod.\ in \cite{deralb}
  allows us to assume that there is an $L^\infty(\mu)$-partition of
  unity $\{K_\alpha\}$ such that, for each $\alpha$:
  \begin{enumerate}
  \item The set $K_\alpha$ is compact and embedds isometrically in
    $S_\alpha$, which is a convex compact subset of some Banach space;
  \item Regarding $\mu\mrest K_\alpha$ as a measure on $S_\alpha$, it
    admits a $1$-Lipschitz Alberti representation $\albrep\alpha.$ in
    the $f$-direction of $\cone$;
  \item The Alberti representation $\albrep\alpha.$ is of the form
    \begin{equation}
      \label{eq:rem_normal_derivations1}
      \mu\mrest K_\alpha = \int_{\frags(S_\alpha)}g_\alpha\,\Psi_\gamma\,dP_\alpha;
    \end{equation}
  \item $g_\alpha$ is a bounded Borel function vanishing on
    $S_\alpha\setminus K_\alpha$;
  \item The probability measure $P_\alpha$ is concentrated on the set
    $\text{\normalfont Lip}_1([0,\tau_\alpha], S_\alpha)$ of
    $1$-Lipschitz maps $[0,\tau_\alpha]\to S_\alpha$, where $\tau_\alpha\in(0,1]$;
  \item $\Psi_\gamma=\mpush\gamma.\lebmeas\mrest[0,\tau_\alpha]$.
  \end{enumerate}
  We can then define a normal current $N_\alpha\in\ncurr 1,S_\alpha.$
  by:
  \begin{equation}
    \label{eq:rem_normal_derivations2}
    N_\alpha=\int_{\frags(S_\alpha)}[\gamma]\,dP_\alpha,
  \end{equation} so that $\mu\mrest K_\alpha\ll\cmass N_\alpha.$ and
  $D_{\albrep\alpha.}=\chi_{\left\{g_\alpha\ne0\right\}} D_{N_\alpha}$ for some nonnegative
  $\lambda_\alpha\in\bborel S_\alpha.$ which vanishes on
  $S_\alpha\setminus K_\alpha$. Thus, the derivation $D\in\wder\mu.$
  defined by $D=\sum_\alpha
  \chi_{K_\alpha}D_{\albrep\alpha.}$ is a normal derivation. Moreover,
  if $\wder\mu.$ is finitely generated, by choosing Alberti
  representations in the directions of independent cone fields, we get
  a generating set for $\wder\mu.$ consisting of normal
  derivations. If $\wder\mu.$ is not finitely generated,
  \catcode`*=11\sync thm:weak*density.\catcode`*=12\ in \cite{deralb} implies that the
  $\lipalg X.$-span of the set of normal derivations is weak* dense in
  $\wder\mu.$. Note that in this case it is necessary to use the
  $\lipalg X.$-span instead of the $L^\infty(\mu)$-span. In fact, if
  $D_1,D_2$ are normal derivations and if $\lambda_1,\lambda_2\in
  L^\infty(\mu)$, then $\lambda_1D_1+\lambda_2D_2$ might not be a
  normal derivation. However, if $\lambda_1$ and $\lambda_2$ are 
  Lipschitz\footnote{more precisely, $\lambda_1$ and $\lambda_2$ have Lipschitz representatives},
  then $\lambda_1D_1+\lambda_2D_2$ is a normal derivation because if
  $N$ is a normal current and $f$ is Lipschitz, then $N\crest f$ is
  still a normal current.
\end{rem}
The goal of this Subsection is the proof of the following Theorem:
\begin{thm}
  \label{thm:const_dir}
  Let $F:X\to l^2$ a Lipschitz map and $D\in\wder\mu.$ a normal
  derivation. Then $\mu\mrest(X\setminus N_{DF})$ admits a $1$-Lipschitz Alberti
  representation in the $F$-direction of $DF$.
\end{thm}
The proof of Theorem \ref{thm:const_dir} requires some preparation and
part of it has been split into some intermediate Lemmas.
\begin{lem}
  \label{lem:reduction}
  In proving Theorem \ref{thm:const_dir} we can assume that:
  \begin{enumerate}
  \item The metric space $X$ is a compact subset of a Polish space $Z$.
  \item The map $F:X\to l^2$ is $1$-Lipschitz and extends to a
    $1$-Lipschitz map $F:Z\to l^2$.
  \item There is a normal current $N\in\ncurr 1,Z.$ with $\mu\ll\cmass
    N.$ and $D=\lambda D_N$, where $D_N$ is the derivation
    associated to $N$ given by Theorem
    \ref{thm:one_currents_derivations}, and $\lambda\in
    L^\infty(\cmass N.)$ is nonnegative.
  \item There are constants $0<C_1\le C_2$ such that:
    \begin{equation}
      \label{eq:reduction_s1}
      C_1\le\frac{d\mu}{d\cmass N.}(x)\le C_2\quad(\forall x\in X).
    \end{equation}
  \end{enumerate}
\end{lem}
\begin{proof}
  The proof makes repeated use of the gluing principle for Alberti
  representations, Theorem \ref{thm:alb_glue}. Let
  $\{U_\alpha,Z_\alpha,N_\alpha,\iota_\alpha\}$ be as in the
  definition of a normal derivation \ref{defn:normal_der}. By taking
  an $L^\infty(\mu\mrest U_\alpha)$-partition of unity of each
  $U_\alpha$, we can assume that the $U_\alpha$ are compact. By the
  gluing principle for Alberti representations (Theorem
  \ref{thm:alb_glue}), it suffices to show that the desired
  representation esists for each $\mu\mrest(U_\alpha\setminus
  N_{DF})$. In the following we can thus write $X$ for $U_\alpha$ and
  drop the index $\alpha$ from the notation. Note also that the vector
  field $DF\circ \iota^{-1}$ can be extended to a vector field $v:Z\to
  \real^\infty$. By \sync compact_reduction.\ in \cite{deralb} one
  can also show that the desired representation exists for $\mpush
  \iota.(\mu\mrest(X\setminus N_{DF}))$; note that in this case the
  direction is determined by the function $F\circ
  \iota^{-1}:\iota(X)\to l^2$. In the following we will then identify
  $\iota(X)$ with $X$, $\mpush \iota.\mu$ with $\mu$, and $\mpush
  \iota. D$ with $D$. We now take a MacShane extension
  \begin{equation}
    \label{eq:reduction_p1}
    \tilde F_i:Z\to \real
  \end{equation} of $F_i$ with the same Lipschitz constant $\glip
  F_i.$ and then choose $c_i\in(0,1)$ such that
  \begin{equation}
    \label{eq:reduction_p2}
    \sum_ic_i^2\glip F_i.^2\le1.
  \end{equation}
  In particular, the map $G:Z\to l^2$ with components $G_i=c_i\tilde F_i$ is
  $1$-Lipschitz. Recalling the discussion before Definition
  \ref{defn:const_alberti}, we also have, after choosing appropriate
  Borel representatives, that the components of the vector field $DG$
  satisfy:
  \begin{equation}\label{eq:reduction_p3}
    DG_i=c_iDF_i.
  \end{equation}
  Consider a fragment $\gamma:K\to X$. As $l^2$ has the Radon-Nikodym
  property, $F\circ\gamma$ and $G\circ\gamma$ are differentiable for
  $t\in Q\subset K$, where the Borel set $Q$ satisfies $\lebmeas(K\setminus Q)=0$. Moreover, at each
  point $t\in Q$ we have that $(F\circ\gamma)'(t)$ and
  $(G\circ\gamma)'(t)$ are determined by the derivatives
  $(F_i\circ\gamma)'(t)$ and $(G_i\circ\gamma)'(t)$ which are related
by
\begin{equation}
  \label{eq:reduction_p4}
  (F_i\circ\gamma)'(t)=c_i(G_i\circ\gamma)'(t).
\end{equation}
In particular, for $\lambda>0$ the following equations are equivalent:
\begin{align}
  \label{eq:reduction_p5}
  (F\circ\gamma)'(t)&=\lambda DF\left(\gamma(t)\right)\\
    \label{eq:reduction_p6}
  (G\circ\gamma)'(t)&=\lambda DG\left(\gamma(t)\right),
\end{align} and so we can replace $F$ with $G$. Finally, we take
another $L^\infty(\mu)$-partition of unity to ensure that (4) holds.
\end{proof}
\par The second ingredient in the proof of Theorem \ref{thm:const_dir}
is the following notion of strict convexity for the local norm in $\wder\mu.$.
\begin{defn}\label{defn:str_conv}
  The local norm $\locnorm\,\cdot\,,{\wder\mu.}.$ on $\wder\mu.$ is called \textbf{strictly
  convex} if the following holds: whenever one has that for derivations
  $D_1,D_2\in\wder\mu.$ and for a Borel set $U$:
  \begin{equation}
    \label{eq:str_conv1}
    \locnorm D_1+D_2,{\wder\mu.}.(x)=\locnorm
    D_1,{\wder\mu.}.(x)+\locnorm D_2,{\wder\mu.}.(x)\quad
    \text{(for $\mu$-a.e.~$x\in U$)},
  \end{equation}
  then there are Borel sets $V_1,V_2$ and nonnegative $\lambda_1\in
  L^\infty(\mu\mrest V_1)$, $\lambda_2\in L^\infty(\mu\mrest V_2)$
  such that one has:
  \begin{align}
    \label{eq:str_conv2}
    \mu\left(U\setminus (V_1\cup V_2)\right) &= 0;\\
    \label{eq:str_conv3}
    \chi_{V_1}D_1&=\lambda_1 D_2;\\
    \label{eq:str_conv4}
    \chi_{V_2}D_2&=\lambda_2 D_1.
  \end{align}
\end{defn}
In Subsection \ref{subsec:renorming} we show (Theorem
\ref{thm:renorm}) that it is always possible to perturb the metric on
$X$ in a biLipschitz way and obtain a strictly convex local norm on
$\wder\mu.$. Therefore, for $\varepsilon>0$, we can assume that the
metric $d$ on $Z$ has been replaced by a metric $\edst , .$ such that:
\begin{equation}
  \label{eq:bilip_renorm}
  d\le\edst , .\le (1+\varepsilon)d,
\end{equation} and $\elocnorm\,\cdot\,,{\wder{\cmass N.}.}.$ is
strictly convex. We now apply Theorem \ref{thm:ps-dec} to obtain
decompositions of $N$ as in \eqref{eq:ps-dec_s1} and
\eqref{eq:ps-dec_s2}. We also construct countably many vector fields
$w_j:Z\to\real^\infty$ such that:
\begin{enumerate}
\item For each $j$, there is $M_j\in\natural$ such that $i>M_j$ implies
  $(w_j)_i=0$, where $(w_j)_i$ is the $i$-th component of $w_j$.
  \item If $DF(z)\ne0$ and $\xi\in\real^\infty\setminus\{0\}$ is not a
    positive multiple of $DF(z)$, then $\langle w_j(z),\xi\rangle>0$ for
    some $j$.
\item For each $z\in Z$, one has $\langle w_j(z),DF(z)\rangle\le0$.
\end{enumerate}
We will denote by $w_0:Z\to\real^\infty$ the null vector field.
\par We now introduce the set $\Omega_{\text{\normalfont fail}}$ of
those curves which, roughly speaking, meet $X$ in a set of positive 
 measure  where the direction of $F\circ\gamma$ fails to be a
positive multiple of $DF$. Specifically, we say that a curve
$\gamma\in\curves(Z)$ belongs to $\Omega_{\text{\normalfont fail}}$ if and only if there is a
piece $\tilde\gamma$ of $\gamma$ such that:
\begin{enumerate}
\item $F\circ\gamma$ is differentiable at each point
  $t\in\dom\tilde\gamma$.
  \item At each point $t\in\dom\tilde\gamma$, the vector $(F\circ\gamma)'(t)$ is
    either $0$ or, if it is nonzero, it is
    not a positive multiple of $DF\circ\gamma(t)$.
    \item The piece $\tilde\gamma$ meets $X\setminus N_{DF}$ in
      positive mass measure: $\cmass[\tilde\gamma].(X\setminus N_{DF})>0$.
\end{enumerate}
In general, the set $\Omega_{\text{\normalfont fail}}$ is not Borel, but,
after completing $\eta$, we will show that it becomes
$\eta$-measurable. The goal is then to show that
$\eta(\Omega_{\text{\normalfont fail}})=0$. Note that the set
$\Omega_{\text{\normalfont fail}}$ is a countable union of the sets
\begin{equation}
  \label{eq:countable_parts}
  \Omega_n(w_j)\subset\curves_n(Z)
\end{equation} defined as follows: $\gamma\in\curves_n(Z)$ belongs to
$\Omega_n(w_j)$ if and only if there is a piece $\tilde\gamma$ of
$\gamma$ such that:
\begin{description}
\item[F1] $F\circ\gamma$ is differentiable at each point
  $t\in\dom\tilde\gamma$.
\item[F2] At each point $t\in\dom\tilde\gamma$, if $j\ne0$ one has
  $\left\langle(F\circ\gamma)'(t),w_j\left(\gamma(t)\right)\right\rangle\ge\frac{1}{n}$,
  and if $j=0$ one has $(F\circ\gamma)'(t)=0$.
\item[F3] The piece $\tilde\gamma$ meets $X\setminus N_{DF}$ in mass measure at least
  $1/n$: $\cmass[\tilde\gamma].(X\setminus N_{DF})\ge\frac{1}{n}$.
\end{description}
We will thus study the measurability properties of each set
$\Omega_n(w_j)$, which is the projection of
\begin{multline}
  \label{eq:omega1}
  \Omega_n^{(1)}(w_j)= \Biggl\{(\gamma,\tilde\gamma)\in\curves_n(Z)\times\pieces_n(Z):\quad\text{$\tilde\gamma$
      is a piece of $\gamma$}\\ \text{and \textbf{(F1)}, \textbf{(F2)} and
      \textbf{(F3)} hold}\Biggr\}
  \end{multline} on $\curves_n(Z)$.
  \begin{lem}
    \label{lem:exo_mea}
    The set $\Omega_n^{(1)}(w_j)$ is of class $\Pi_1^1$,
    i.e.~coanalytic. Thus $\Omega_n(w_j)$ is of class $\Sigma_2^1$
    and, moreover, there is a uniformizing function
    $\sigma_{j,n}:\Omega_n(w_j)\to\Omega_n^{(1)}(w_j)$ which is
    universally measurable and whose graph is of class $\Pi_1^1$.
  \end{lem}
  \begin{proof} We prove the Lemma for $j\ne0$ as the case $j=0$
    requires a minor modification of the argument.
    Consider the set
    $\Omega_n^{(2)}(w_j)\subset\curves_n(Z)\times\pieces_n(Z)\times[0,1]$
    consisting of the triples $(\gamma,\tilde\gamma,t)$ such that:
    \begin{description}
    \item[G1] $\tilde\gamma$ is a piece of $\gamma$.
    \item[G2] $\cmass[\tilde\gamma].(X\setminus N_{DF})\ge\frac{1}{n}$.
    \item[G3] either $t\not\in\dom\tilde\gamma$ or
      $t\in\dom\tilde\gamma$ and $F\circ\gamma$ is differentiable at
      $t$ with $\left\langle(F\circ\gamma)'(t),w_j\left(\gamma(t)\right)\right\rangle\ge\frac{1}{n}$.
    \end{description}
We show that $\Omega_n^{(2)}(w_j)$ is Borel. First note that the set of couples
$(\gamma,\tilde\gamma)$ such that $\tilde\gamma$ is a piece of
$\gamma$ is closed in $\curves_n(Z)\times\pieces_n(Z)$. Second, as the
map $\tilde\gamma\mapsto\cmass[\tilde\gamma].(X\setminus N_{DF})$ is Borel
\eqref{eq:mass_borel2}, the set of pieces with
$\cmass[\tilde\gamma].(X\setminus N_{DF})\ge\frac{1}{n}$ is Borel. Third, the set of
pairs $(\tilde\gamma,t)$ with $t\in\dom\tilde\gamma$ is
closed. Therefore, we have only to show that the set
\begin{equation}
  \label{eq:lem:exo_mea_p1}
  \tilde\Omega=\left\{(\gamma,t)\in\curves_n(Z)\times[0,1]:\text{
      $(F\circ\gamma)'(t)$ exists and $\left\langle(F\circ\gamma)'(t),w_j\left(\gamma(t)\right)\right\rangle\ge\frac{1}{n}$}\right\}
\end{equation} is Borel. Let $\mathscr{S}$ denote a countable dense
set of $l^2$. We then have:
\begin{multline}
  \label{eq:lem:exo_mea_p2}
  \tilde\Omega=\bigcap_{\varepsilon\in\rational\cap(0,1)}\bigcup_{\delta\in\rational\cap(0,1)}\bigcap_{s_1,s_2\in\rational\cap(0,1)}\bigcup_{\xi\in\mathscr{S}}\biggl(
  \curves_n(Z)\times\bigl\{t\in(0,1):\text{$|t-s_1|\ge\delta$}\\\text{or
    $|t-s_2|\ge\delta$}\bigr\}\cup S(\varepsilon,\delta,s_1,s_2,\xi)\biggr),
\end{multline}
where $(\gamma,t)\in S(\varepsilon,\delta,s_1,s_2,\xi)$ if and
only if the following inequalities hold:
\begin{align}
  |t-s_i|&<\delta\quad(i=1,2)\\
  \left\|F\circ\gamma(t)(s_1-s_2)-(t-s_2)F\circ\gamma(s_1)+(t-s_1)F\circ\gamma(s_2)\right\|_{l^2}&\le\varepsilon|t-s_1|\,|t-s_2|\\
  \left\|F\circ\gamma(t)-F\circ\gamma(s_1)-\xi(t-s_1)\right\|_{l^2}&\le\varepsilon|t-s_1|\\
  \left\langle\xi,w_j\left(\gamma(t)\right)\right\rangle&\ge\frac{1}{n}-\varepsilon.
\end{align}
We conclude that $S(\varepsilon,\delta,s_1,s_2,\xi)$ is Borel
and so $\tilde\Omega$ is Borel, which completes the proof that
$\Omega_n^{(2)}(w_j)$ is Borel. Note that $\Omega_n^{(1)}(w_j)$ is the
coprojection of $\Omega_n^{(2)}(w_j)$ on
$\curves_n(Z)\times\pieces_n(Z)$, which implies that
$\Omega_n^{(1)}(w_j)$ is coanalytic. By the definition of the class
$\Sigma_2^1$, as $\Omega_n(w_j)$ is the projection of a conalytic set,
it is of class $\Sigma_2^1$. By the $\Sigma_1^1$-determinacy
\cite[Cor.~36.21]{kechris_desc}, $\Omega_n(w_j)$ is
universally measurable and there is a uniformizing function
$\sigma_{j,n}$ as in the statement of this Lemma.
\end{proof}
We now define maps
\begin{equation}
  \label{eq:xi}
  \begin{aligned}
    \Xi_{j,n}:\curves_n(Z)&\to M_1(Z)\\
    \gamma&\mapsto
    \begin{cases}
      \left[\sigma_{j,n}(\gamma)\right]&\text{if $\gamma\in\Omega_n(w_j)$}\\
      0&\text{otherwise,}
    \end{cases}
  \end{aligned}
\end{equation}
and 
\begin{equation}
  \label{eq:xic}
  \begin{aligned}
    \Xi^c_{j,n}:\curves_n(Z)&\to M_1(Z)\\
    \gamma&\mapsto
    \begin{cases}
      [\gamma]-\left[\sigma_{j,n}(\gamma)\right]&\text{if $\gamma\in\Omega_n(w_j)$}\\
      [\gamma]&\text{otherwise.}
    \end{cases}
  \end{aligned}
\end{equation}
Note that for each $(f,\pi)\in\bborel Z.\times\lipfun Z.$ and each Borel set $E\subset Z$, the maps:
\begin{align}
  \gamma&\mapsto\Xi_{j,n}(\gamma)(fd\pi)\\
  \gamma&\mapsto\Xi^c_{j,n}(\gamma)(fd\pi)\\
  \gamma&\mapsto\cmass\Xi_{j,n}(\gamma).(E)\\
  \gamma&\mapsto\cmass\Xi^c_{j,n}(\gamma).(E)
\end{align}
are universally measurable. In particular, they are $\eta$-measurable,
as we assume that $\eta$ is complete. Moreover, by definition of the
maps $\Xi_{j,n}$ and $\Xi^c_{j,n}$, we have the relation:
\begin{equation}
  \label{eq:sumxi}
  [\gamma]=\Xi_{j,n}(\gamma)+\Xi^c_{j,n}(\gamma);
\end{equation}
this implies that
\begin{equation}
  \label{eq:sumxi2}
  \cmass[\gamma].\le\cmass\Xi_{j,n}(\gamma).+\cmass\Xi^c_{j,n}(\gamma).;
\end{equation}
however, for $\eta$-a.e.~$\gamma$, if $\gamma\in\Omega_n(w_j)$, \eqref{eq:ps-dec_s3} implies that:
\begin{multline}
  \label{eq:sumbal}
  \cmass[\gamma].(Z)=l(\gamma)=\int_0^1\metdiff\gamma(t)\,dt\\=\int_{\dom\sigma_{j,n}(\gamma)}\metdiff\gamma(t)\,dt+\int_{[0,1]\setminus\dom\sigma_{j,n}(\gamma)}\metdiff\gamma(t)\,dt
\\ \ge\cmass\Xi_{j,n}(\gamma).(Z)+\cmass\Xi^c_{j,n}(\gamma).(Z),
\end{multline} and thus, for $\eta$-a.e.~$\gamma$, we have:
\begin{equation}
  \label{eq:sumxieq}
  \cmass[\gamma].=\cmass\Xi_{j,n}(\gamma).+\cmass\Xi^c_{j,n}(\gamma)..
\end{equation}
\begin{lem}
  \label{lem:fail_vanishing}
  For each $n$ and $j$ we have that $\eta(\Omega_n(w_j))=0$.
\end{lem}
\begin{proof}[Proof of Lemma \ref{lem:fail_vanishing}]
  We argue by contradiction assuming that
  $\eta(\Omega_n(w_j))>0$. Note that:
  \begin{equation}
    \label{eq:fail_vanishing_p1}
    N=\underbrace{\int_{\curves(Z)}\Xi_{j,n}(\gamma)\,d\eta(\gamma)}_{T_{j,n}}+
    \underbrace{\int_{\curves(Z)}\Xi^c_{j,n}(\gamma)\,d\eta(\gamma)}_{T^c_{j,n}},
  \end{equation}
  and, using \eqref{eq:sumxieq},
  \begin{multline}
    \label{eq:fail_vanishing_p2}
    \cmass
    N.(Z)=\int_{\curves(Z)}\cmass[\gamma].(Z)\,d\eta(\gamma)=\int_{\curves(Z)}\cmass\Xi_{j,n}(\gamma).(Z)\,d\eta(\gamma)
    \\+\int_{\curves(Z)}\cmass\Xi^c_{j,n}(\gamma).(Z)\,d\eta(\gamma)\\
    \ge\cmass T_{j,n}.(Z)+\cmass T^c_{j,n}.(Z),
  \end{multline}
  where we used:
  \begin{align}
    \label{eq:fail_vanishing_p3}
    \int_{\curves(Z)}\cmass\Xi_{j,n}(\gamma).(Z)\,d\eta(\gamma)&\ge\cmass
    T_{j,n}.(Z),\\
    \label{eq:fail_vanishing_p4}
    \int_{\curves(Z)}\cmass\Xi^c_{j,n}(\gamma).(Z)\,d\eta(\gamma)&\ge\cmass
    T^c_{j,n}.(Z).
  \end{align}
  In particular, $T_{j,n}$ and $T^c_{j,n}$ are \emph{complementary
  subcurrents} of $N$ because \eqref{eq:fail_vanishing_p2} implies that
  \begin{equation}
    \label{eq:fail_vanishing_p5}
    \cmass N.=\cmass T_{j,n}.+\cmass T^c_{j,n}..
  \end{equation}
  Moreover, we also have that:
  \begin{align}
    \label{eq:fail_vanishing_p6}
    \cmass
    T_{j,n}.&=\int_{\curves(Z)}\cmass\Xi_{j,n}(\gamma).\,d\eta(\gamma),\\
    \label{eq:fail_vanishing_p7}
    \cmass
    T^c_{j,n}.&=\int_{\curves(Z)}\cmass\Xi^c_{j,n}(\gamma).\,d\eta(\gamma).
  \end{align}
  By Theorem \ref{thm:one_currents_derivations} we find derivations
  $D_{j,n},D^c_{j,n}\in\wder{\cmass N.}.$ such that
  \begin{align}
    \label{eq:fail_vanishing_p8}
    T_{j,n}(fd\pi)&=\int_ZfD_{j,n}\pi\,d\cmass N.\\
    \label{eq:fail_vanishing_p9}
    T^c_{j,n}(fd\pi)&=\int_ZfD^c_{j,n}\pi\,d\cmass N.\\
\label{eq:fail_vanishing_p10}
    \cmass T_{j,n}.&=\elocnorm D_{j,n},{\wder{\cmass N.}.}.\,\cmass
N.\\
\label{eq:fail_vanishing_p11}
\cmass T^c_{j,n}.&=\elocnorm D^c_{j,n},{\wder{\cmass N.}.}.\,\cmass
N..
\end{align}
Note that \eqref{eq:reduction_s1} implies that the measures $\cmass
N.\mrest X$ and $\mu$ are in the same measure class and we can thus
identify the rings $L^\infty(\cmass
N.\mrest X)$ and $L^\infty(\mu)$ and the modules $\wder{\cmass
N.\mrest X}.$ and $\wder\mu.$. Having picked a Borel representative
of $\elocnorm D_{j,n},{\wder{\cmass N.}.}.$ and letting
\begin{equation}
  \label{eq:fail_vanishing_p12}
  X_{j,n}=\left\{x\in X\setminus N_{DF}:\elocnorm D_{j,n},{\wder{\cmass N.}.}.(x)>0\right\},
\end{equation}
we show that $\mu(X_{j,n})>0$ by showing that $\cmass T_{j,n}.(X\setminus N_{DF})>0$:
\begin{equation}
  \label{eq:fail_vanishing_p13}
  \cmass T_{j,n}.(X\setminus N_{DF})=\int_{\curves(Z)}\cmass \Xi_{j,n}(\gamma).(X\setminus N_{DF})\,d\eta(\gamma)\ge\frac{1}{n}\eta(\Omega_n(w_j))>0.
\end{equation}
We now combine \eqref{eq:fail_vanishing_p5},
\eqref{eq:fail_vanishing_p10} and \eqref{eq:fail_vanishing_p11} with
 the strict convexity of $\elocnorm\,\cdot\,,{\wder{\cmass N.}.}.$ and
 the fact that $\elocnorm D_{j,n},{\wder{\cmass N.}.}.>0$ on $X_{j,n}$,
 to conclude that there are nonnegative $\lambda^c_{j,n},\lambda_{j,n}\in\bborel Z.$,
 which vanish on $Z\setminus X_{j,n}$ and are such that:
 \begin{align}
   \label{eq:fail_vanishing_p14}
   \lambda^c_{j,n}D^c_{j,n}&=\lambda_{j,n}D_{j,n};\\
   \label{eq:fail_vanishing_p14bis}
   \lambda^c_{j,n}(z)&>0\quad(\forall z\in Z).
 \end{align}
 We then conclude that
 \begin{equation}
   \label{eq:fail_vanishing_p15}
   \lambda^c_{j,n}D_N=(\lambda^c_{j,n}+\lambda_{j,n})D_{j,n}.
 \end{equation}
 If $j=0$ we have $\lambda^c_{0,n}D_NF=0$ which contradicts the fact
 that $\lambda^c_{0,n}DF\ne0$. For $j\ne0$ we argue as follows:
 let $M_j$ be the maximal index such that $(w_j)_{M_j}\ne0$; we
 consider the $1$-form $\omega=\sum_{k=1}^{M_j}(w_j)_k\,dF_k$ and let
 $g$ denote a nonnegative continuous function; we have:
 \begin{equation}
   \label{eq:fail_vanishing_p16}
   \int_Zg\langle w_j,D_{j,n} F\rangle\,d\cmass N.=T_{j,n}(g\omega)=\int_{\curves(Z)}\Xi_{j,n}(\gamma)(g\omega)\,d\eta(\gamma);
 \end{equation}
now, if $\gamma\in\Omega_n(w_j)$,
$\sum_{k=1}^{M_j}(w_j)_k(\gamma(t))\,(F_k\circ\gamma)'(t)\ge1/n$ for
$t\in\dom\sigma_{j,n}$, which implies:
\begin{equation}
  \label{eq:fail_vanishing_p17}
  \int_Zg\langle w_j,D_{j,n} F\rangle\,d\cmass N.\ge\frac{1}{n}\int_{\Omega_n(w_j)}d\eta(\gamma)\int_{\dom\sigma_{j,n}}g\circ\gamma(t)\,dt;
\end{equation}
as the curves in $\Omega_n(w_j)$ are $n$-Lipschitz and because of
\eqref{eq:constdir_mass_piece}, we obtain
\begin{equation}
  \label{eq:fail_vanishing_p18}
  \begin{aligned}
    \int_Zg\langle w_j,D_{j,n} F\rangle\,d\cmass
    N.&\ge\frac{1}{n^2}\int_{\Omega_n(w_j)}d\eta(\gamma)\int_{\dom\sigma_{j,n}}g\circ\gamma(t)\,\metdiff
    \gamma(t)\,dt\\
    &\ge\frac{1}{n^2}\int_{\Omega_n(w_j)}d\eta(\gamma)\int_Zg\,d\cmass\Xi_{j,n}(\gamma).\\
    &=\frac{1}{n^2}\int_Zg\,d\cmass T_{j,n}.\\
    &=\frac{1}{n^2}\int_Zg\elocnorm D_{j,n},{\wder\mu.}.\,d\cmass N..
  \end{aligned}
\end{equation} From \eqref{eq:fail_vanishing_p18} we conclude that
$\langle w_j,D_{j,n} F\rangle>0$ on $X_{j,n}$; moreover, from
\eqref{eq:fail_vanishing_p15} we obtain $\langle w_j,DF\rangle>0$ on
$X_{j,n}$, but this contradicts the choice of $w_j$. Thus, $\eta(\Omega_n(w_j))=0$.
\end{proof}
\begin{proof}[Proof of Theorem \ref{thm:const_dir}]
  By Lemma \ref{lem:fail_vanishing} we have $\eta(\Omega_n(w_j))=0$
  which implies $\eta(\Omega_{\text{\normalfont fail}})=0$. Therefore,
  for $\eta$-a.e.~$\gamma$ and $\lebmeas\mrest\dom\gamma$-a.e.~$t$,
  $(F\circ\gamma)'(t)$ is a positive multiple of $DF(\gamma(t))$. The
  desired Alberti representation is then obtained using the measure
  $\eta$. Specifically, let
  \begin{equation}
    \label{eq:sfail_vanishing_p1}
      \rprm:\curves(Z)\to\frags(Z)
    \end{equation} be a Borel map
    which reparametrizes each $\gamma\in\curves(Z)$ to a $1$-Lipschitz
    map $\rprm:[0,\lceil\glip\gamma.\rceil]\to Z$.  Note that up to
    passing to a Borel $L^\infty(\mu)$-partition of unity we can
    assume that the set $X\setminus N_{DF}$ is compact; we now consider 
    the measure:
    \begin{equation}
      \label{eq:sfail_vanishing_p2}
      \nu_1=\int_{\curves(Z)}\cmass[\rprm(\gamma)].\,d\eta(\gamma)=\int_{\frags(Z)}\cmass[\gamma].\,d(\mpush\rprm.\eta)(\gamma)
    \end{equation}
    and observe that $\cmass N.\ll\nu_1$ and that $\mpush\rprm.\eta$
    is concentrated on the set of $1$-Lipschitz fragments. We now let
    \begin{equation}
      \label{eq:sfail_vanishing_p3}
      \frags(Z,X\setminus N_{DF})=\left\{\gamma\in\frags(Z): \gamma^{-1}(X\setminus N_{DF})\ne\emptyset\right\}
    \end{equation}
    and note that $\frags(Z,X\setminus N_{DF})$ is a closed subset of $\frags(Z)$. An
    argument similar to that of \sync borel_rest.\ in \cite{deralb} shows that the map:
    \begin{equation}
      \label{eq:sfail_vanishing_p4}
      \begin{aligned}
        \frest X\setminus N_{DF}.:\frags(Z,X\setminus N_{DF})&\to\frags(X)\\
        \gamma&\mapsto\gamma|\gamma^{-1}(X\setminus N_{DF})
      \end{aligned}
    \end{equation} is Borel. We now consider the measure
    \begin{equation}
      \label{eq:sfail_vanishing_p5}
      \nu_2=\int_{\frags(Z,X\setminus N_{DF})}\cmass[{\frest
        X\setminus N_{DF}.}(\gamma)].\,d(\mpush\rprm.\eta)(\gamma)=\int_{\frags(X)}\cmass[\gamma].\,\underbrace{d(\mpush{\frest
        X\setminus N_{DF}.}.\mpush\rprm.\eta)(\gamma)}_{\eta_2}
  \end{equation}
  and note that $\mu\ll\nu_2$; an Alberti representation as in the
  statement of this Theorem is then:
  \begin{equation}
    \label{eq:sfail_vanishing_p6}
    \mu=\int_{\frags(X)}(\mpush\rprm.\eta)(\frags(Z,X\setminus N_{DF}))\,\cmass[\gamma\crest\frac{d\mu}{d\nu_2}].\,\frac{d\eta_2(\gamma)}{(\mpush\rprm.\eta)(\frags(Z,X\setminus N_{DF}))}.
  \end{equation}
\end{proof}
\section{Technical tools}
\label{sec:technical}
\subsection{Exterior Products}
\label{subsec:ext_prod}
In this Subsection we define the exterior powers in different
categories:
\begin{itemize}
\item In the category $\bancat$, whose objects are Banach
  spaces and whose morphisms are bounded linear maps;
\item In the category $\lmodcat$, whose objects are
  $L^\infty(\mu)$-modules and whose morphisms are bounded 
  module homomorphisms;
\item In the category $\lnmodcat$, whose objects are
  $L^\infty(\mu)$-normed modules and whose morphisms are bounded
  module homomorphisms.
\end{itemize}
\par In the following, if $Z$ is a Banach space, we will denote by
$Z^*$ its dual. If $Z$ is also an $L^\infty(\mu)$-module, we will
denote by $Z'$ the dual module; note that $Z^*$ and $Z'$ are, in
general, different (Example \ref{exa:lp_algdual_nullity}).
\begin{defn}\label{def:alternating_bana}
  For Banach spaces $Z$ and $W$, let $\alt k,Z,W.$ denote the set of
  alternating multilinear maps $\varphi:Z^k\to W$ which are bounded
  with respect to the norm:
  \begin{equation}
    \|\varphi\|_{\alt
      k,Z,W.}=\sup\left\{\|\varphi(m_1,\cdots,m_k)\|_W:
      \max_{i=1,\cdots,k}\|m_i\|_Z\le1\right\}.
  \end{equation}
\end{defn}
\begin{defn}\label{def:alternating_modules}
  For $L^\infty(\mu)$-modules $M$ and $N$, let $\alt k,M,N.$ denote the
  set of alternating $L^\infty(\mu)$-multilinear maps $\varphi:M^k\to N$ which
  are bounded with respect to the norm:
  \begin{equation}
    \|\varphi\|_{\alt
      k,M,N.}=\sup\left\{\|\varphi(m_1,\cdots,m_k)\|_N:
      \max_{i=1,\cdots,k}\|m_i\|_M\le1\right\}.
  \end{equation}
\end{defn}
\begin{defn}\label{def:ext_pow_bana}
  Let $Z$ be an Banach space. The \textbf{projective $k$-th power
  of $Z$ in the category $\bancat$} is a pair
$(\banext k,Z.,\pi)$, where $\banext k,Z.$ is an
Banach space and $\pi\in\alt k,Z, {\banext k,Z.}.$,
  which satisfies the following universal property: for each
  $\varphi\in\alt k,Z,W.$, where $W$ is an Banach space,
 there is a unique $\hat\varphi\in\hom(\banext k,Z.,W)$
 which makes the following diagram commutative:
  \begin{equation}\label{eq:ext_power_bana_kpower}
    \xy
    (0,20)*+{Z^k}="vk"; (30,20)*+{\banext k,Z.}="ext"; (0,0)*+{W}="w";
    {\ar "vk";"w"}?*!/^2mm/{\varphi};
    {\ar "vk";"ext"}?*!/_2mm/{\pi};
    {\ar@{-->}"ext";"w"}?*!/_2mm/{\hat\varphi}
    \endxy
  \end{equation}
and such that $\|\hat\varphi\|_{\hom(\banext k,Z.,W)}=\|\varphi\|_{\alt k,Z,W.}$.
\end{defn}
\begin{defn}\label{def:ext_pow_module}
  Let $M$ be an $L^\infty(\mu)$-module. The \textbf{projective $k$-th power
  of $M$ in the category $\lmodcat$} is a pair
$(\lmodext k,\mu,M.,\pi)$, where $\lmodext k,\mu,M.$ is an
$L^\infty(\mu)$-module and $\pi\in\alt k,M,{\lmodext k,\mu,M.}.$,
  which satisfies the following universal property: for each
  $\varphi\in\alt k,M,N.$, where $N$ is an $L^\infty(\mu)$-module,
 there is a unique $\hat\varphi\in\hom(\lmodext k,\mu,M.,N)$
 which makes the following diagram commutative:
  \begin{equation}\label{eq:ext_pow_module_kpower}
    \xy
    (0,20)*+{M^k}="vk"; (30,20)*+{\lmodext k,\mu,M.}="ext"; (0,0)*+{N}="w";
    {\ar "vk";"w"}?*!/^2mm/{\varphi};
    {\ar "vk";"ext"}?*!/_2mm/{\pi};
    {\ar@{-->}"ext";"w"}?*!/_2mm/{\hat\varphi}
    \endxy
  \end{equation}
and such that $\|\hat\varphi\|_{\hom(\lmodext k,\mu,M.,N)}=\|\varphi\|_{\alt k,M,N.}$.
\end{defn}
\begin{defn}\label{def:ext_pow_normed_module}
  Let $M$ be an $L^\infty(\mu)$-normed module. The \textbf{projective $k$-th power
  of $M$ in the category $\lnmodcat$} is a pair
$(\lnmodext k,\mu,M.,\pi)$, where $\lnmodext k,\mu,M.$ is an
$L^\infty(\mu)$-normed module and $\pi\in\alt k,M,{\lnmodext k,\mu,M.}.$,
  which satisfies the following universal property: for each
  $\varphi\in\alt k,M,N.$, where $N$ is an $L^\infty(\mu)$-normed module,
 there is a unique $\hat\varphi\in\hom(\lnmodext k,\mu,M.,N)$
 which makes the following diagram commutative:
  \begin{equation}\label{eq:ext_pow_normed_module_kpower}
    \xy
    (0,20)*+{M^k}="vk"; (30,20)*+{\lnmodext k,\mu,M.}="ext"; (0,0)*+{N}="w";
    {\ar "vk";"w"}?*!/^2mm/{\varphi};
    {\ar "vk";"ext"}?*!/_2mm/{\pi};
    {\ar@{-->}"ext";"w"}?*!/_2mm/{\hat\varphi}
    \endxy
  \end{equation}
and such that $\|\hat\varphi\|_{\hom(\lnmodext k,\mu,M.,N)}=\|\varphi\|_{\alt k,M,N.}$.
\end{defn}
\par We now present some illustrative examples. Recall that an \emph{atom}
for a measure $\mu$ is a positive measure set $A$ such that for each
proper subset $B$, $\mu(B)=0$; note that if $A$ is an atom for a Radon
measure $\mu$, $A$ is a singleton. A measure without atoms is called
\emph{non-atomic}; in particular, a Radon measure $\mu$ is non-atomic
if and only if $\mu(\{x\})=0$ for each singleton $\{x\}$. We now
recall the Sierpi\'nski's Theorem \cite[pg.~39]{fryszkowski_dec_book}:
\begin{thm}\label{thm:sierpinski}
  If $\mu$ is a non-atomic measure on a space $X$ with
  $\mu(X)=c<\infty$ and $\Sigma$ is the $\sigma$-algebra of
  $\mu$-measurable sets, then there is a function $S:[0,c]\to\Sigma$
  which is monotone with respect to inclusion and is a right inverse
  of $\mu:\Sigma\to[0,c]$.
\end{thm}
\par In the following we will assume $p\in[1,\infty)$.
\begin{exa}
  If $\mu$ is a finite sum of Dirac masses, $L^p(\mu)$ can be
  identified with $L^\infty(\mu)$ and so is free of rank 1.
  \par Suppose that $\mu$ is non-atomic; in particular, by Theorem
  \ref{thm:sierpinski}, given any positive measure set $U$, it is
  possible to find $f\in L^p(\mu\mrest U)$ with $\|f\|_{L^p(\mu)}\le1$
  and $\forall n$ $\mu(x\in U: |f(x)|>n)>0$. Suppose that $L^p(\mu)$
  was generated by $f_1,\cdots,f_N$; then there would be a set of
  positive measure $U$ on which the $f_i$, and hence all the element
  in $L^p(\mu)$ would be uniformly bounded, leading to a
  contradiction.
  \par However, any two elements of $L^p(\mu)$ are linearly dependent
  over $L^\infty(\mu)$. If $f\in L^p(\mu)$ vanishes on a set of
  positive measure $U$, it suffices to note that $f$ is annihilated by
  $\chi_U$. If $f$ and  $g$ are nowhere vanishing, there is a positive measure
  set $U$
  on which $0<c_0<|f|,|g|<c_1<\infty$; then it is
  possible to find $\lambda\in L^\infty(\mu)$ with $\chi_Uf+\lambda
  g=0$. In particular, if $f\in L^p(\mu)$ is nowhere vanishing, the
  algebraic submodule generated by $f$ is dense.
\end{exa}
\begin{exa}\label{exa:lp_algdual_nullity}
  Given an $L^\infty(\mu)$-module $M$, there are two notions of
  dual. The dual module of $M$, $\hom(M,L^\infty(\mu))=M'$ is an
  $L^\infty(\mu)$-normed module. However, the dual Banach space of
  $M$, $M^*$, is also an $L^\infty(\mu)$-module if we let
  \begin{equation}
    \lambda.\varphi(m)=\varphi(\lambda m).
  \end{equation}
  For example, if $M=L^p(\mu)$, then $M^*=L^q(\mu)$.
  \par We show that if $\mu$ is
  non-atomic, then the algebraic dual of $M$ (and hence $M'$) is
  trivial. By replacing $\mu$ by $\mu\mrest U$, where $U$ is a set of
  positive measure, we can assume that $\mu$ is finite, so that
  $L^\infty(\mu)\subset L^p(\mu)$; let $\Phi:L^p(\mu)\to
  L^\infty(\mu)$ be a module homomorphism; supposing that $\Phi(1)\ne0$, we can use
  Theorem \ref{thm:sierpinski} to find $f\in L^p(\mu)$ and
  $\mu$-measurable sets $U_n$ such that:
  \begin{itemize}
  \item for each $n$, $\Phi(1)\chi_{U_n}f\in L^\infty(\mu)$;
  \item for each $n$:
    \begin{equation}
      \mu\left(\left\{x\in U_n: |\Phi(1)\chi_{U_n}f|(x)>n\right\}\right)>0.
    \end{equation}
  \end{itemize}
  Note that
  \begin{equation}
    \chi_{U_n}\Phi(f)=\Phi(\chi_{U_n}f)=\Phi(1)\chi_{U_n}f
  \end{equation} shows that $\Phi(f)\notin L^\infty(\mu)$, a
  contradiction. Thus $\Phi(1)=0$ implying $\Phi=0$. In this case, the
  dual module of $L^p(\mu)$ is trivial.
  \par Suppose now that $\mu$ is a countable sum of Dirac masses:
  $\mu=\sum_nc_n\delta_{p_n}$, so that a function $f$ is in the unit
  ball of $L^p(\mu)$ if and only if
  \begin{equation}\sum_n|f_n|^pc_n\le1\quad(f_n=f(p_n));
  \end{equation}
  let $\varphi\in M'$ and note that for $m\ne n$ one has:
  \begin{equation}
    \label{eq:lp_algdual_nullity_1}
    \chi_{\{p_m\}}\varphi\left(\chi_{\{p_n\}}\right)=\varphi\left(\chi_{\{p_m\}}\cdot\chi_{\{p_n\}}\right)=\varphi(0)=0;
  \end{equation}
  therefore there is a sequence $\{\lambda_n\}\subset\real$ satisfying:
  \begin{equation}
    \label{eq:lp_algdual_nullity_2}
    \varphi\left(\chi_{\{p_n\}}\right)=\lambda_n\chi_{\{p_n\}}.
  \end{equation}
  The sequence $\{\lambda_n\}$ satisfies also the bound:
  \begin{equation}
    \label{eq:lp_algdual_nullity_3}
    |\lambda_n|\le\|\varphi\|\,(c_n)^{\frac{1}{p}}
  \end{equation}
  and, for each $f\in L^p(\mu)$, one has:
  \begin{equation}
    \label{eq:lp_algdual_nullity_4}
    \varphi\left(f\chi_{\{p_n\}}\right)=f_n\lambda_n\,\chi_{\{p_n\}};
  \end{equation}
  we thus conclude that
  \begin{equation}
    \label{eq:lp_algdual_nullity_5}
    \varphi(f)=\sum_{n=1}^\infty f_n\lambda_n\,\chi_{\{p_n\}}.
  \end{equation}
  Conversely, any sequence $\{\lambda_n\}\subset\real$ satisfying
  $\sup_n(c_n)^{-1/p}|\lambda_n|<\infty$ gives rise to a $\varphi\in
  M'$ via (\ref{eq:lp_algdual_nullity_5}). We finally remark that the
  norm of $\varphi$ is determined by the corresponding $\{\lambda_n\}$:
  \begin{equation}
    \label{eq:lp_algdual_nullity_6}
    \|\varphi\|=\sup_n|\lambda_n|(c_n)^{-\frac{1}{p}}.
  \end{equation}
\end{exa}
\begin{exa}\label{exa:alternating_nullity}
  For an $L^\infty(\mu)$-module $N$, $\alt k,L^p(\mu),N.$ is
  trivial for $k\ge2$, while the case $k=1$ has been treated in Example
  \ref{exa:lp_algdual_nullity}. Let $\Omega$ denote the set of those
  $f\in L^p(\mu)\cap L^\infty(\mu)$ such that the set:
  \begin{equation}
    \label{eq:alternating_nullity_1}
    E_f=\left\{x:f(x)\ne0\right\}
  \end{equation}
  has finite measure. Then $\Omega$ is a dense 
  algebraic submodule of $L^p(\mu)$; in particular, $T\in \alt
  k,L^p(\mu),N.$ is determined by its values on $\Omega^k$; now let
  $\{f_1,\cdot,f_k\}\subset \Omega$ and $E=\bigcup_{i=1}^kE_{f_i}$;
  then $\chi_E\in\Omega$ and
  \begin{equation}
    \label{eq:alternating_nullity_2}
    T(f_1,\cdots,f_k)=f_1f_2\cdots f_k\,\cdot T(\chi_E,\chi_E,\cdots,\chi_E)=0
  \end{equation}
  by the alternating property. We thus conclude that $T=0$.
  \par Note that the nullity of $\alt k,L^p(\mu),N.$ for each 
  $L^\infty(\mu)$-normed module $N$ implies that $\lnmodext
  k,\mu,L^p(\mu).=0$. 
\end{exa}
\begin{exa}\label{exa:exterior_non_null_pow}
  Let $\|\cdot\|$ a norm on $\real^n$; on $\bigwedge^k\real^n$ we
  consider the norm:
  \begin{equation}\label{exa:std_norm}
    \|\omega\|=\inf\left\{\sum_{i\in I}\|v_{i_1}\|\cdots\|v_{i_k}\|:
      \omega=\sum_{i\in I}v_{i_1}\wedge\cdots\wedge v_{i_k}\right\};
  \end{equation}
  the fact that $\|\,\cdot\,\|$ is non-degenerate follows either from
  Lemma \ref{lem:ext_norm_char_ban_lnmod} or by modifying the proof of
  Theorem \ref{thm:ext_pow_bana}.
  We will denote by $\mu$ a non-atomic Radon measure.
  \par We claim that
  $\lnmodext k,\mu,L^p(\mu;\real^n).$ is trivial. By the Hahn-Banach
  Theorem, it suffices to show
  that $\alt k,L^p(\mu;\real^n),L^\infty(\mu).$ is trivial; suppose
  that for $U$ a Borel set of finite measure and
  $\{v_i\}_{i=1}^k\subset\real^n$ independent vectors we had
  \begin{equation}
    T(\chi_Uv_1,\cdots,\chi_Uv_k)\ne0
  \end{equation} where $T\in\alt k,L^p(\mu;\real^n),L^\infty(\mu).$;
  arguing as in Example \ref{exa:lp_algdual_nullity}, we would reach a contradiction.
  \par However we show
  that $\lmodext k,\mu,L^k(\mu;\real^n).$ can be identified with
  $L^1(\mu;\bigwedge\nolimits^k\real^n)$.
   By H\"older's inequality, the multilinear alternating map
  \begin{equation}
    \begin{aligned}
      E:(L^k(\mu;\real^n))^k&\to L^1(\mu;\bigwedge\nolimits^k\real^n)\\
      (f_1,\cdots,f_k)&\mapsto f_1\wedge\cdots\wedge f_k
    \end{aligned}
  \end{equation}
  has norm at most $1$. For $\psi\in L^1(\mu)$ define:
  \begin{equation}
    \begin{aligned}
      T_\psi:(\real^n)^k&\to N\\
      (v_1,\ldots,v_k)&\mapsto T(\sgn\psi\,|\psi|^{1/k}v_1,|\psi|^{1/k}v_2,\ldots,|\psi|^{1/k}v_k);
    \end{aligned}
  \end{equation}
  the map $T_\psi$ is multilinear and alternating (as a map of vector
  spaces); let $\hat
  T_\psi:\bigwedge^k\real^n\to N$ denote the corresponding linear map
  given by the universal property of $\bigwedge^k\real^n$.
  Consider $\omega\in\bigwedge^k\real^n$ and, having fixed
  $\varepsilon>0$, write
  \begin{equation}
    \label{eq:exa_exterior_non_null_pow_1}
    \omega=\sum_{i\in I}v_{i_1}\wedge\cdots\wedge v_{i_k}
  \end{equation}
  in a way that satisfies:
  \begin{equation}
    \label{eq:exa_exterior_non_null_pow_2}
    \sum_{i\in I}\|v_{i_1}\|\cdots\|v_{i_k}\|\le\|\omega\|+\varepsilon;
  \end{equation}
  then
  \begin{equation}
    \label{eq:exa_exterior_non_null_pow_3}
    \left\|\hat
      T_\psi(\omega)\right\|\le\|T\|\,\|\psi\|_{L^1(\mu)}\sum_{i\in I}\|v_{i_1}\|\cdots\|v_{i_k}\|;
  \end{equation}
  letting $\varepsilon\searrow0$ we conclude that:
  \begin{equation}
    \label{eq:exa_exterior_non_null_pow_4}
    \left\|\hat T_\psi(\omega)\right\|\le\|T\|\,\|\psi\|_{L^1(\mu)}\,\|\omega\|.
  \end{equation}
  Consider now $\psi_1,\psi_2\in L^1(\mu)$ and let:
  \begin{equation}
    \label{eq:exa_exterior_non_null_pow_5}
    \psi_i^{(n)}=\psi_i\cdot\chi_{|\psi_i|\le n}\cdot\chi_{B(0,n)}\quad(i=1,2);
  \end{equation}
  since $\{\sgn\psi_i^{(n)}\cdot|\psi_i^{(n)}|^{1/k}\}$ converges to
  $\sgn\psi_i\cdot|\psi_i|^{1/k}$ in $L^k(\mu)$, the
  continuity of $T$ implies:
  \begin{equation}
    \label{eq:exa_exterior_non_null_pow_6}
    \hat T_{\psi_1+\psi_2}(\omega)=\lim_{n\to\infty}\hat T_{\psi_1^{(n)}+\psi_2^{(n)}}(\omega);
  \end{equation}
  since $\psi_i^{(n)}\in L^\infty(\mu)$, the multilinearity of $T$
  implies:
  \begin{equation}
    \label{eq:exa_exterior_non_null_pow_7}
    T_{\psi_1^{(n)}+\psi_2^{(n)}}=\hat T_{\psi_1^{(n)}}+\hat T_{\psi_2^{(n)}};
  \end{equation}
  we thus conclude that
  \begin{equation}
    \label{eq:exa_exterior_non_null_pow_8}
    \hat T_{\psi_1+\psi_2}=\hat T_{\psi_1}+\hat T_{\psi_2}.
  \end{equation}
  A similar argument can be also used to show that for $\lambda\in
  L^\infty(\mu)$ one has $\hat
  T_{\lambda\psi}=\lambda \hat T_\psi$.
  \par We now fix a basis $\{\omega_\alpha\}$ of $\bigwedge^k\real^n$
  consisting of simple vectors. The norm:
  \begin{equation}
    \label{eq:exa_exterior_non_null_pow_9}
    \|\omega\|'=\left\{\max|\sigma_\alpha|:\omega=\sum_\alpha\sigma_\alpha\omega_\alpha\right\}
  \end{equation}
  is equivalent to the norm introduced in (\ref{exa:std_norm}) and so any
  $\psi\in
  L^1(\mu;\bigwedge\nolimits^k\real^n)$ can be written as
  \begin{equation}\label{eq:exa_lpwedge_basis_dev1}
    \psi=\sum_\sigma\psi_\alpha\omega_\alpha,
  \end{equation} where $\psi_\alpha\in L^1(\mu)$;
  in particular, we can define $\hat
  T:L^1(\mu;\bigwedge\nolimits^k\real^n)\to N$ by
  \begin{equation}\label{eq:exa_lpwedge_basis_dev1bis}
    \hat T(\psi)=\sum_\alpha\hat T_{\psi_\alpha}(\omega_\alpha),
  \end{equation} and obtain the bound:
  \begin{equation}
    \label{eq:exa_exterior_non_null_pow_10}
    \|\hat T\|\le C\|T\|
  \end{equation}
  where $C$ depends only on $n$, $k$, $\|\,\cdot\,\|'$ and
  $\|\,\cdot\,\|$. Using the density of simple functions in
  $L^k(\mu;\real^n)$ one can show that $\hat T\circ \pi= T$. We now
  prove that
  \begin{equation}
    \label{eq:exa_exterior_non_null_pow_11}
    \|\hat T\|\le \|T\|
  \end{equation}
  by showing that
  \begin{equation}
    \label{eq:exa_exterior_non_null_pow_12}
    \left\|\hat T(\psi)\right\|\le\|T\|\,\|\psi\|_{L^1(\mu;\bigwedge^k\real^n)}
  \end{equation} when $\psi$ is simple. We write
  $\psi=\sum_j\tilde\omega_j\chi_{U_j}$ where
  $\tilde\omega_j=\sum_\alpha\sigma_{j,\alpha}\omega_\alpha$.
  Choosing vectors $\{v_i^{(\alpha)}\}$ such that
  $\omega_\alpha=v_1^{(\alpha)}\wedge\cdots\wedge v_k^{(\alpha)}$, we get:
  \begin{equation}
    \label{eq:exa_exterior_non_null_pow_13}
    \begin{split}
      \hat T(\psi)&=\hat
      T\left(\sum_\alpha\left(\sum_j\sigma_{j,\alpha}\chi_{U_j}\right)\omega_\alpha\right)
      =\sum_\alpha \hat
      T_{\sum_j\sigma_{j,\alpha}\chi_{U_j}}(\omega_\alpha)\\
      &=\sum_\alpha
      T\left(\sgn\left(\sum_j\sigma_{j,\alpha}\chi_{U_j}\right)\left|\sum_j\sigma_{j,\alpha}\chi_{U_j}
          \right|^{1/k}v_1^{(\alpha)},\cdots,\left|\sum_j\sigma_{j,\alpha}\chi_{U_j}
          \right|^{1/k}v_k^{(\alpha)}\right)\\
        &=\sum_\alpha\sum_jT\left(\sgn\sigma_{j,\alpha}\cdot|\sigma_{j,\alpha}|^{1/k}\chi_{U_j}v_1^{(\alpha)},\cdots,
          |\sigma_{j,\alpha}|^{1/k}\chi_{U_j}v_k^{(\alpha)}\right)\\
        &=\sum_jT_{\chi_{U_j}}(\tilde\omega_j);
    \end{split}
  \end{equation}
  so using (\ref{eq:exa_exterior_non_null_pow_4}) we conclude that
  (\ref{eq:exa_exterior_non_null_pow_12}) holds and the proof that
  $L^1(\mu;\bigwedge\nolimits^k\real^n)$ is the exterior $k$-power
  of $L^k(\mu;\real^n)$ is complete.
\end{exa}
\par In the remainder of this section we assume that $\mu$ is a Radon
measure. The following Lemma summarizes some properties of the Banach
space $\alt k,M,N.$.
\begin{lem}\label{lem:alt_module}
  Let $M$, $N$ be $L^\infty(\mu)$-modules; then $\alt k,M,N.$ is an
  $L^\infty(\mu)$-module and it is an $L^\infty(\mu)$-normed module if
  $N$ is an $L^\infty(\mu)$-normed module. Moreover if $M$ and $N$ are
  $L^\infty(\mu)$-normed modules, for $\varphi\in\alt k,M,N.$
  and $\{m_i\}_{i=1}\subset M$
  \begin{equation}\label{eq:lem_alt_mod_s1}
    \locnorm {\varphi(m_1,\cdots,m_k)},N.\le\locnorm\varphi,{\alt
      k,M,N.}.\locnorm m_1,M.\cdots\locnorm m_k,M..
  \end{equation}
\end{lem}
\begin{proof}[Proof of Lemma \ref{lem:alt_module}]
  The fact that $\alt k,M,N.$ is a Banach space with the norm
  $\|\cdot\|_{\alt k,M,N.}$ follows from a standard argument. For
  $(\varphi,\lambda)\in\alt k,M,N.\times L^\infty(\mu)$ the product
  $\lambda\varphi$ can be defined by:
  \begin{equation}
    \lambda\varphi(m_1,\cdots,m_k)=\varphi(m_1,\cdots,\lambda
    m_i,\cdots,m_k)\quad\text{(any choice of $i$)}
  \end{equation}
  which makes $\alt k,M,N.$ an $L^\infty(\mu)$-module.
  \par If $N$ is an $L^\infty(\mu)$-normed module, for a
  $\mu$-measurable subset $U\subset X$, we have
  \begin{multline}
    \|\varphi\|_{\alt
      k,M,N.}=\sup_{\|m_i\|_M\le1}\|\varphi(m_1,\cdots,m_k)\|_N\\
    =
    \sup_{\|m_i\|_M\le1}\max\left(\|\chi_U\varphi(m_1,\cdots,m_k)\|_N,
      \| \chi_{X\setminus U}\varphi(m_1,\cdots,m_k)\|_N\right)\\
    =\max\left(\sup_{\|m_i\|_M\le1}\|(\chi_U\varphi)(m_1,\cdots,m_k)\|,
      \sup_{\|m_i\|_M\le1}\|(\chi_{X\setminus
        U}\varphi)(m_1,\cdots,m_k)\| \right)\\
    =\max\left(\|\chi_U\varphi\|_{\alt k,M,N.},\|\chi_{X\setminus U}\varphi\|_{\alt k,M,N.}\right);
  \end{multline}
  by \cite[Thm.~2]{weaver00} $\alt k,M,N.$ is an
  $L^\infty(\mu)$-normed module.
  \par We now show \eqref{eq:lem_alt_mod_s1} under the assumption that
  $M$ and $N$ are $L^\infty(\mu)$-normed modules. By
  \cite[Cor.~6]{weaver00} we can find $\Phi_{m_1,\cdots,m_k}\in N'$
  with $\|\Phi_{m_1,\cdots,m_k}\|_{N'}\le1$ and
  \begin{equation}
    \locnorm{\varphi(m_1,\cdots,m_k)},N.=\left\langle\Phi_{m_1,\cdots,m_k},\varphi(m_1,\cdots,m_k)\right\rangle;
  \end{equation}
  let $\xi\in\alt k,M,L^\infty(\mu).$ be defined by
  \begin{equation}
    \xi(\tilde m_1,\cdots,\tilde
    m_k)=\left\langle\Phi_{m_1,\cdots,m_k},\varphi(\tilde
      m_1,\cdots,\tilde m_k)\right\rangle;
  \end{equation}
  for $\varepsilon>0$ we can find an $L^\infty(\mu)$-partition of
  unity $\{U_\alpha\}$ such that for $x\in U_\alpha$ and $1\le i\le k$,
  \begin{align}
    \label{eq:lem_alt_mod_p1_sub1}
    \locnorm\xi,{\alt
      k,M,L^\infty(\mu).}.(x)&\in\left(\|\chi_{U_\alpha}\xi\|_{\alt k,M,L^\infty(\mu).}-\epsi, \|\chi_{U_\alpha}\xi\|_{\alt k,M,L^\infty(\mu).}\right];\\ \label{eq:lem_alt_mod_p1_sub2}
    \locnorm
    m_i,M.(x)&\in\left(\|\chi_{U_\alpha}m_i\|_{M}-\epsi, \|\chi_{U_\alpha}m_i\|_{M}\right].
  \end{align}
Using the definition of norm in $\alt k,M,L^\infty(\mu\mrest
U_\alpha).$ and \eqref{eq:lem_alt_mod_p1_sub1} and
\eqref{eq:lem_alt_mod_p1_sub2},
\begin{equation}\label{eq:lem_alt_mod_p2}
  \begin{split}
    \xi(m_1,\cdots,m_k)&=\sum_{\alpha}\chi_{U_\alpha}\xi(m_1,\cdots,m_k)\\
    &=\sum_{\alpha}(\chi_{U_\alpha}\xi)(\chi_{U_\alpha}m_1,\cdots,\chi_{U_\alpha}m_k)\\
    &\le\sum_{\alpha}\chi_{U_\alpha}\|\chi_{U_\alpha}\xi\|_{\alt
      k,M,L^\infty(\mu).}\|\chi_{U_\alpha}m_1\|_{M}\cdots\|\chi_{U_\alpha}m_k\|_M\\
    &\le\sum_{\alpha}\chi_{U_\alpha}\left(\locnorm\xi,{\alt
        k,M,L^\infty(\mu).}.+\epsi\right)\prod_{i=1}^k
    \left(\locnorm m_i,M.+\epsi\right)\\
    &=\left(\locnorm\xi,{\alt
        k,M,L^\infty(\mu).}.+\epsi\right)\prod_{i=1}^k
    \left(\locnorm m_i,M.+\epsi\right).
  \end{split}
\end{equation} Note that \eqref{eq:lem_alt_mod_s1} follows from
\eqref{eq:lem_alt_mod_p2} letting $\epsi\searrow0$ provided we show
\begin{equation}
  \label{eq:lem_alt_mod_p3}
  \locnorm\xi,{\alt k,M,L^\infty(\mu).}.\le\locnorm\varphi,{\alt k,M,N.}..
\end{equation} As for each $\mu$-measurable $U$ we have
\begin{equation}
  \|\chi_U\xi\|_{\alt k,M,L^\infty(\mu).}\le\|\chi_U\varphi\|_{\alt k,M,N.},
\end{equation} \eqref{eq:lem_alt_mod_p3} holds.        
\end{proof}
We now prove the existence of the exterior powers in the category $\bancat$.
\begin{thm}\label{thm:ext_pow_bana}
For $Z$ a Banach space, the $k$-th exterior
power in the category $\bancat$ exists and can be realized as a
closed subspace of the dual space $\alt k,M,\real.^*$;
moreover,
 the algebraic $k$-th exterior power $\bigwedge^kZ$
 is dense in $\banext k,Z.$. 
\end{thm}
\begin{proof}[proof of Theorem \ref{thm:ext_pow_bana}]
  For $\varphi\in\alt k,Z,\real.$ let
  $\tilde\varphi:\bigwedge^kZ\to\real$ denote the unique linear map
  corresponding to $\varphi$ given by the universal property of
  $\bigwedge^kZ$. In particular, we obtain a map $E$ from
  $\bigwedge^kZ$ to the algebraic dual of $\alt k,Z,\real.$ by letting
  $\langle E(w),\varphi\rangle=\tilde\varphi(w)$. We now show that
  $E(w)$ is a bounded functional. Let
  \begin{equation}
    w=\sum_{i\in I}z_{i_1}\wedge\cdots\wedge z_{i_k}
  \end{equation} and note that
  \begin{equation}\label{eq:ext_pow_bana1}
    \begin{split}
      \left\|\sum_{i\in I}z_{i_1}\wedge\cdots\wedge
        z_{i_k}\right\|_{(\alt k,Z,\real.)^*}&=\sup_{\|\varphi\|_{\alt k,Z,\real.}\le 1}\left|    \left\langle\sum_{i\in I}z_{i_1}\wedge\cdots\wedge
          z_{i_k}
          ,\varphi\right\rangle\right|\\
      &\le\sup_{\|\varphi\|_{\alt
          k,Z,\real.}\le
        1}\sum_{i\in I}\left|\varphi(z_{i_1},\cdots,z_{i_k})\right|\\
      &\le\sum_{i\in I}\|z_{i_1}\|_X\cdots\|z_{i_k}\|_X.
    \end{split}
  \end{equation}
  \par We now show that $E$ is injective; suppose $w\ne0$; let $Z_0$
  denote the linear span of $\Omega=\{z_{i_j}:{j=1,\ldots,k; i\in
    I}\}$ so that $Z_0$ is a finite dimensional vector space of
  dimension $L\ge k$. Having chosen a basis
  $\{v_\alpha\}_{\alpha=1}^L$ of $Z_0$, without loss of generality we
  can assume that
  \begin{equation}
    w=\sum_{j\in\Lambda_{k,L}}c_j v_{j_1}\wedge\cdots\wedge v_{j_k}
  \end{equation} with $c_{(1,\ldots,k)}\ne0$. If 
  $\{v^*_{\alpha}\}_{\alpha=1}^L$ is the dual basis of $\{v_\alpha\}_{\alpha=1}^L$, by
  the Hahn-Banach Theorem the functionals $v^*_{\alpha}$ can be extended to elements of $Z^*$; in
  particular,
  \begin{equation}
    \begin{aligned}
      \Xi:Z^k&\to\real\\
      (z_1,\cdots,z_k)&\mapsto\det(\left(\langle v^*_\alpha,z_i\rangle\right)_{\alpha,i=1}^k)
    \end{aligned}
  \end{equation} defines an element of $\alt k,Z,\real.$ and
  \begin{equation}
    \langle E(w),\Xi\rangle=c_{(1,\cdots,k)}\ne0
  \end{equation} showing that $E$ is injective.
  \par We can thus identify $\bigwedge^kZ$  with a linear subspace of $\alt
k,Z,\real.^*$ and we will denote its completion in the
$\|\cdot\|_{(\alt k,Z,\real.)^*}$ norm by $\banext k,Z.$. The map
$\pi$ is defined by
\begin{equation}
  \pi(z_1,\cdots,z_k)=z_1\wedge\cdots\wedge z_k;
\end{equation}
note that $\pi$ is alternating and multilinear and \eqref{eq:ext_pow_bana1} shows that
it is bounded. Let $\varphi\in\alt k,Z,W.$ and define
$\hat\varphi:\bigwedge^kZ\to W$ by
\begin{equation}
\hat\varphi\left(\sum_{i\in I}z_{i_1}\wedge\cdots\wedge
    z_{i_k}\right)=\sum_{i\in I}\varphi(z_{i_1},\cdots,z_{i_k});
\end{equation}
this is well-defined because $\varphi$ is alternating multilinear and
because of the universal property of $\bigwedge^kZ$. In order to
show that $\hat\varphi$ has a unique extension $\hat\varphi:\banext
k,Z.\to W$, it suffices to show that $\hat\varphi$  is
bounded:
\begin{equation}\label{eq:ext_pow_bana2}
  \begin{split}
    \left\|\hat\varphi\left(\sum_{i\in I}z_{i_1}\wedge\cdots\wedge
    z_{i_k}\right)\right\|_W=\sup_{w^*\in
      W^*:\|w^*\|_{W^*}\le1}\left\langle
    w^*,\hat\varphi\left(\sum_{i\in I}z_{i_1}\wedge\cdots\wedge
    z_{i_k}\right)\right\rangle\\
    =\|\varphi\|_{\alt k,Z,W.}\sup_{\substack{w^*\in
      W^*:\\ \|w^*\|_{W^*}\le1}}\sum_{i\in I}\left\langle w^*,\frac{1}{
    \|\varphi\|_{\alt k,Z,W.}}\varphi(z_{i_1},\cdots,z_{i_k})\right\rangle\\
    \le\|\varphi\|_{\alt k,Z,W.}\sup_{\substack{\tau\in\alt k,Z,\real.:\\ \|\tau\|_{\alt k,Z,\real.}\le1}}
    \left\langle\tau,\sum_{i\in I}z_{i_1}\wedge\cdots\wedge
    z_{i_k}\right\rangle\\
    \le\|\varphi\|_{\alt k,Z,W.}         \left\|\sum_{i\in I}z_{i_1}\wedge\cdots\wedge
    z_{i_k}\right\|_{(\alt k,Z,\real.)^*}.
  \end{split}
\end{equation}
\par Note that \eqref{eq:ext_pow_bana2} shows that
\begin{equation}
\|\hat\varphi\|_{\hom({\banext k,Z.},W)}\le \|\varphi\|_{\alt k,Z,W.};
\end{equation}
for the reverse inequality,
observe that for each $\epsi>0$, there are $z_i\in Z$
$(i\in\{1,\cdots,k\})$ such that $\|z_i\|_Z\le1$ and
\begin{equation}
  \|\varphi\|_{\alt k,Z,W.}<\epsi+\|\varphi(z_1,\cdots, z_k) \|_Z;
\end{equation}
but 
\begin{equation}
\varphi(z_1,\cdots, z_k)=\hat\varphi(z_1\wedge\cdots\wedge z_k)
\end{equation} and by \eqref{eq:ext_pow_bana1} 
\begin{equation}
  \left\|z_1\wedge\cdots\wedge z_k\right\|_{\extpow k.Z}\le1;
\end{equation} thus
\begin{equation}
\|\varphi\|_{\alt k,Z,W.}<\epsi+\|\hat\varphi\|_{\hom({\banext k,Z.},W)}.  
\end{equation}
\end{proof}
We now turn to the existence of exterior powers in the category $\lnmodcat$.
\begin{thm}\label{thm:ext_pow_normed_module}
For $M$ an $L^\infty(\mu)$-normed module, the $k$-th exterior
power in the category $\lnmodcat$ exists and can be realized as a
closed submodule of the dual module $\alt k,M,L^\infty(\mu).'$;
moreover,
 the algebraic $k$-th exterior power ${}_{L^\infty(\mu)}\bigwedge^kM$
 is dense in $\lnmodext k,\mu,M.$. 
\end{thm}
\begin{proof}[Proof of Theorem \ref{thm:ext_pow_normed_module}]
  Part of the proof is similar to the Banach space case
  (Theorem \ref{thm:ext_pow_bana}). For $\varphi\in\alt k,M,L^\infty(\mu).$
  let $\tilde\varphi:{}_{L^\infty(\mu)}\bigwedge^kM\to L^\infty(\mu)$ denote
  the unique module homomorphism corresponding to $\varphi$ given by
  the universal property of ${}_{L^\infty(\mu)}\bigwedge^kM$. The same
  estimate \eqref{eq:ext_pow_bana1} used in the Banach space case 
  shows that the map:
  \begin{equation}
    E:{}_{L^\infty(\mu)}\bigwedge^kM\to\alt k,M,L^\infty(\mu).'
  \end{equation} sending $w\in{}_{L^\infty(\mu)}\bigwedge^kM$ to the
  functional $E(w)$ satisfying
  \begin{equation}
    \langle E(w),\varphi\rangle=\tilde\varphi(w),
  \end{equation} is well-defined.
  \par We now show that $E$ is injective. Let
  \begin{equation}
    w=\sum_{i\in I}m_{i_1}\wedge\cdots\wedge m_{i_k}\ne0
  \end{equation} and $M_0$ the $L^\infty(\mu)$-submodule of $M$
  generated by the finite set
  \begin{equation}\Omega=\{m_{i_j}:{j=1,\ldots,k; i\in I}\}. 
  \end{equation}
  By \cite[Lem.~9]{weaver00} there are disjoint measurable sets
  $\{U_i\}_{i=1}^{\#\Omega}$ such that
  \begin{equation}1=\sum_{i=1}^{\#\Omega}\chi_{U_i},
  \end{equation}
  and if $\mu(U_i)>0$, then
  $\chi_{U_i}M_0$, regarded as an $L^\infty(\mu\mrest U_i)$-module, is
  free of rank $i$; as we are assuming $w\ne0$, $\chi_{U_L}w\ne0$ for some
  index $L\ge k$. Let $\{\tilde m_i\}_{i=1}^L$ a basis of
  $\chi_{U_L}M_0$ over $L^\infty(\mu\mrest U_i)$; without loss of
  generality, we can assume that
  \begin{equation}
    \chi_{U_L}w=\sum_{j\in \Lambda_{k,N}}\lambda_j \tilde
    m_{j_1}\wedge\cdots \wedge \tilde m_{j_k}, 
  \end{equation} with $\lambda_{(1,\cdots,k)}\ne0$. Moreover, by
  \cite[Thm.~10]{weaver00} we can choose a
  measurable $V\subset U_L$ with $\chi_{V}\lambda_{(1,\cdots,k)}\ne0$
  and find $C>0$ such that, if we define for $x\in V$
  \begin{equation}
    \begin{aligned}
      p_x:\real^L&\to(0,\infty)\\
      v&\mapsto\locnorm \sum_{i=1}^Lv_i\tilde m_i,M.(x),
    \end{aligned}
  \end{equation} then $p_x$ is a norm satisfying
  \begin{equation}\label{eq:ext_pow_normed_module1}
    Cp_x(v)\ge\|v\|_{\infty}\quad(\forall (x,v)\in V\times\real^L).
  \end{equation} Note that functions in $L^\infty(\mu\mrest V)$ can be
  canonically extended to $L^\infty(\mu)$ because we can indentify
  $L^\infty(\mu\mrest V)$ with $\chi_{V}L^\infty(\mu)$; the maps
  \begin{equation}
    \begin{aligned}
      \xi_i:\chi_{V}M_0&\to L^\infty(\mu)\quad(i=1,\ldots,L)\\
      \sum_{i=1}^L\lambda_i\tilde m_i&\mapsto\lambda_i,
    \end{aligned}
  \end{equation} are bounded linear functionals by
  \eqref{eq:ext_pow_normed_module1}. By the Hanh-Banach
  Theorem \cite[Thm.~5]{weaver00} the $\{\xi_i\}$ can be extended to elements
  of $M'$; in particular,
  \begin{equation}
    \begin{aligned}
      \Xi:M^k&\mapsto L^\infty(\mu)\\
      (m_1,\ldots,m_k)&\mapsto\det(\left(\langle\xi_i,m_j\rangle\right)_{i,j=1}^k)
    \end{aligned}
  \end{equation} defines an element of $\alt k,M,L^\infty(\mu).$ and
  \begin{equation}
    E(w)(\chi_{V}\Xi)=\chi_V\lambda_{(1,\ldots,k)}\ne0
  \end{equation} showing that $E$ is injective. The proof is now
  completed as in Theorem \ref{thm:ext_pow_bana}.
\end{proof}
We now provide a characterization of the norms in the exterior powers.
\begin{lem}\label{lem:ext_norm_char_ban_lnmod}
  For $Z$ a Banach space, if $w\in\bigwedge^kZ\hookrightarrow\banext
  k,Z.$
  \begin{equation}
    \label{eq:ext_norm_char_ban_lnmod_s1}
    \|w\|_{\banext k,Z.}=\inf\left\{\sum_{i\in
        I}\|z_{i_1}\|_Z\cdots\|z_{i_k}\|_Z: w=\sum_{i\in I}z_{i_1}\wedge\cdots\wedge
      z_{i_k}\right\}.
  \end{equation}
  \par If $M$ is an $L^\infty(\mu)$-normed module, for each
  \begin{equation}w\in{}_{L^\infty(\mu)}\bigwedge^kM\hookrightarrow\lnmodext
    k,\mu, M.,
  \end{equation}
  \begin{multline}
    \label{eq:ext_norm_ban_lnmod_s2}
    \|w\|_{\lnmodext k,\mu,M.}=\inf\biggl\{
      \left\|\sum_{i\in I}\locnorm m_{i_1},M.\cdots\locnorm
        m_{i_k},M.\right\|_{L^\infty(\mu)}:\\ w=\sum_{i\in I}m_{i_1}\wedge\cdots\wedge m_{i_k}\biggr\};
  \end{multline}
  moreover, if $w=\sum_{i\in I}m_{i_1}\wedge\cdots\wedge m_{i_k}$,
  \begin{equation}
    \label{eq:ext_norm_ban_lnmod_s3}
    \locnorm w,{\lnmodext k,\mu,M.}.\le\sum_{i\in I}\locnorm
    m_{i_1},M.\cdots\locnorm m_{i_k},M..
  \end{equation}
\end{lem}
\begin{proof}[Proof of Lemma \ref{lem:ext_norm_char_ban_lnmod}]
  For $Z$ a Banach space, define for $w\in\bigwedge^kZ$
  \begin{equation}
    \gamma(w)=\inf\left\{\sum_{i\in
        I}\|z_{i_1}\|_Z\cdots\|z_{i_k}\|_Z: w=\sum_{i\in I}z_{i_1}\wedge\cdots\wedge
      z_{i_k}\right\};
  \end{equation}
  then $\gamma(w)$ is a seminorm and \eqref{eq:ext_pow_bana1} shows that
  \begin{equation}
    \label{eq:ext_norm_char_ban_lnmod_p1}
    \|w\|_{\banext k,Z.}\le\gamma(w);
  \end{equation} in particular \eqref{eq:ext_norm_char_ban_lnmod_p1}
  shows that $\gamma$ is a norm on $\bigwedge^kZ$ and the same
  argument used in the proof of Theorem \ref{thm:ext_pow_bana}
  (compare \eqref{eq:ext_pow_bana2}) shows
  that the completion of $\bigwedge^kZ$ in the $\gamma$-norm satisfies
  the universal property characterizing $\banext k,Z.$; thus
  $\|w\|_{\banext k,Z.}=\gamma(w)$.
  \par Let $M$ an $L^\infty(\mu)$-normed module; we first show
  \eqref{eq:ext_norm_ban_lnmod_s3}. It suffices to show that for
  each $U$ $\mu$-measurable,
  \begin{equation}\label{eq:ext_norm_char_ban_lnmod_p2}
    \|\chi_U w\|_{\lnmodext k,\mu,M.}\le\left\|\chi_U\sum_{i\in I}\locnorm
      m_{i_1},M.\cdots\locnorm m_{i_k},M.\right\|_{L^\infty(\mu)};
  \end{equation} from the definition of $\|\cdot\|_{\lnmodext
    k,\mu,M.}$ (proof of Theorem \ref{thm:ext_pow_normed_module}) we
  can find, for each $\epsi>0$, an alternating map $\varphi\in\alt k,M,L^\infty(\mu).$ with norm
  at most $1$ and satisfying:
  \begin{equation}
    \|\chi_Uw\|_{\lnmodext k,\mu,M.}\le\|\tilde\varphi(\chi_Uw)\|_{L^\infty(\mu)}+\epsi;
  \end{equation} but \eqref{eq:lem_alt_mod_s1} implies
  \begin{equation}
    \left|\tilde\varphi(\chi_Uw)\right|\le\chi_U\sum_{i\in I}\locnorm
    m_{i_1},M.\cdots\locnorm m_{i_k},M.,
  \end{equation}from which we obtain
  \eqref{eq:ext_norm_char_ban_lnmod_p2} taking the essential $\sup$
  and letting $\epsi\searrow0$. To show
  \eqref{eq:ext_norm_ban_lnmod_s2} let
  \begin{equation}
    \gamma(w)=\inf\biggl\{
    \left\|\sum_{i\in I}\locnorm m_{i_1},M.\cdots\locnorm
      m_{i_k},M.\right\|_{L^\infty(\mu)}:\\ w=\sum_{i\in I}m_{i_1}\wedge\cdots\wedge m_{i_k}\biggr\};
  \end{equation} then $\gamma(w)$ is a seminorm on
  ${}_{L^\infty(\mu)}\bigwedge^kM$. Note that
  \eqref{eq:ext_norm_ban_lnmod_s3} implies $\|\cdot\|_{\lnmodext
    k,\mu,M.}\le\gamma$, so that $\gamma$ is a norm; the proof of
  Theorem \ref{thm:ext_pow_normed_module} implies that the completion $Y$
  of ${}_{L^\infty(\mu)}\bigwedge^kM$ in the $\gamma$-norm satisfies
  the universal property defining $\lnmodext k,\mu,M.$ provided
  that $Y$ is an $L^\infty(\mu)$-normed module. To show that $Y$ is an
  $L^\infty(\mu)$-normed module it suffices to
  show that for a $\mu$-measurable set $U$,
  \begin{equation}\label{eq:ext_norm_char_ban_lnmod_p3}
    \gamma(w)=\max(\gamma(\chi_Uw),\gamma(\chi_{U^c}w)).
  \end{equation} Having shown \eqref{eq:ext_norm_char_ban_lnmod_p3},
  uniqueness of $\lnmodext k,\mu,M.$ will
  imply that $\|\cdot\|_{\lnmodext k,\mu,M.}=\gamma$. To show
  \eqref{eq:ext_norm_char_ban_lnmod_p3}, for $\epsi>0$ let
  \begin{align}
    \chi_Uw&=\sum_{i\in
      I_U}\chi_Um^{(1)}_{i_1}\wedge\cdots\wedge\chi_Um^{(1)}_{i_k},\\
    \chi_{U^c}w&=\sum_{i\in I_{U^c}}\chi_{U^c}m^{(2)}_{i_1}\wedge\cdots\wedge\chi_{U^c}m^{(2)}_{i_k},
  \end{align} with
  \begin{align}
    \left\|\sum_{i\in
        I_U}\locnorm \chi_Um^{(1)}_{i_1},M.\cdots\locnorm
      \chi_Um^{(1)}_{i_k},M.\right\|_{L^\infty(\mu)}&<\gamma(\chi_Uw)+\epsi\\
    \left\|\sum_{i\in
        I_{U^c}}\locnorm \chi_{U^c}m^{(2)}_{i_1},M.\cdots\locnorm \chi_{U^c}m^{(2)}_{i_k},M.\right\|_{L^\infty(\mu)}&<\gamma(\chi_{U^c}w)+\epsi;
  \end{align} without loss of generality (introducing null terms) we
  can assume that $I_U=I_{U^c}=I$ so that
  \eqref{eq:ext_norm_char_ban_lnmod_p3} follows
  observing that
  \begin{equation}
    w=\sum_{i\in I}(\chi_Um^{(1)}_{i_1}+\chi_{U^c}m^{(2)}_{i_1})\wedge\cdots\wedge(\chi_Um^{(1)}_{i_k}+\chi_{U^c}m^{(2)}_{i_k})
  \end{equation} and letting $\epsi\searrow0$.
\end{proof}
\par There are also pairings between exterior powers:
\begin{lem}\label{lem:ext_bana_prod_pairings}
  Suppose $Z$ is a Banach space; the bilinear mapping
  \begin{equation}
    \wedge:\bigwedge\nolimits^k Z\times \bigwedge\nolimits^l Z\to\bigwedge\nolimits^{k+l}Z
  \end{equation} which on pairs of simple multivectors is given by:
  \begin{equation}
    \wedge:((z_1,\cdots,z_k),(u_1,\cdots,u_l))\mapsto
    z_1\wedge\cdots\wedge z_k\wedge u_1\wedge\cdots\wedge u_l, 
  \end{equation} extends to a bounded bilinear map
  \begin{equation}
    \wedge:\banext k,Z.\times \banext l,Z.\to\banext k+l,Z.
  \end{equation} satisfying
  \begin{equation}
    \label{eq:ext_bana_prod_pairings_s1}
    \|\omega_1\wedge \omega_2\|_{\banext
      k+l,Z.}\le\|\omega_1\|_{\banext k,Z.}\,\|\omega_2\|_{\banext l,Z.}.
  \end{equation}
  \par Suppose $M$ is an $L^\infty(\mu)$-module; for $1\le i\le k$,
  the bilinear mapping (in the category $\bancat$)
  \begin{equation}
    \begin{aligned}
      \Phi_i:L^\infty(\mu)\times\bigwedge\nolimits^kM&\to\bigwedge\nolimits^kM\\
      \left(\lambda,\sum_{j\in J}m_{j_1}\wedge\cdots\wedge
        m_{j_k}\right)&\mapsto \sum_{j\in J}m_{j_1}\wedge\cdots\wedge
      \lambda m_{j_i}\wedge\cdots \wedge m_{j_k}
    \end{aligned}
  \end{equation}
  extends to a bounded bilinear map
  \begin{equation}
    \Phi_i:L^\infty(\mu)\times\banext k,M.\to\banext k,M.
  \end{equation}
  satisfying
  \begin{equation}
    \label{eq:ext_bana_prod_pairings_s2}
    \|\Phi_i(\lambda,\omega)\|_{\banext
      k,Z.}\le\|\lambda\|_{L^\infty(\mu)}\,\|\omega\|_{\banext k,Z.}.
  \end{equation}
\end{lem}
\begin{proof}[Proof of Lemma \ref{lem:ext_bana_prod_pairings}]
  It follows from the first part of Lemma
  \ref{lem:ext_norm_char_ban_lnmod}; in particular,
  \eqref{eq:ext_bana_prod_pairings_s1} and
  \eqref{eq:ext_bana_prod_pairings_s2}  follow from
  \eqref{eq:ext_norm_char_ban_lnmod_s1}. 
\end{proof}
\par We now turn to the existence of the exterior power in the
category $\lmodcat$.
\begin{thm}\label{thm:ext_pow_module}
  For $M$ an $L^\infty(\mu)$-module, the $k$-th exterior power in the
  category $\lmodcat$ exists and can be realized as a quotient space
  of $\banext k,M.$ (in $\bancat$) by the closure of the linear span
  of the set
  \begin{equation}\label{eq:ext_pow_module_s1}
    \left\{\Phi_i(\lambda,\omega)-\Phi_j(\lambda,\omega): 1\le i,j\le
      k,\; \lambda\in L^\infty(\mu), \omega\in\bigwedge\nolimits^kM\right\}.
  \end{equation}
\end{thm}
\begin{proof}[Proof of Theorem \ref{thm:ext_pow_module}]
  Let $\mathcal{Q}$ denote the linear span of the set
  \eqref{eq:ext_pow_module_s1}. If $\varphi\in\alt k,M,N.$, where $N$
  is an $L^\infty(\mu)$-module, let $\tilde\varphi:\banext k,M.\to N$
  denote the corresponding map given by the universal property of
  $\banext k,M.$; note that $\tilde\varphi$ annihilates
  $\mathcal{Q}$. Moreover, $\banext k,M./\mathcal{\bar Q}$ becomes an
  $L^\infty(\mu)$-module letting
  \begin{equation}
    \lambda.[\omega]=[\Phi_i(\lambda,\omega)]\quad\text{($(\lambda,\omega)\in
      L^\infty(\mu)\times \banext k,M.$ and $1\le i\le k$)}.
  \end{equation} If we let $\pi'$ denote the composition of
  $\pi:M^k\to\banext k,M.$ with the quotient map $\banext
  k,M.\to\banext k,M./\mathcal{\bar Q}$, then $\pi'\in\alt k,M,{\banext
    k,M./\mathcal{\bar Q}}.$; similarly, if we let
  $\hat\varphi:\banext k,M./\mathcal{\bar Q}\to N$ the map induced by
  $\tilde\varphi$, then $\hat\varphi\in\hom(\banext k,M./\mathcal{\bar
    Q},N)$. Note that $\hat\varphi\circ\pi'=\varphi$ and that
  uniqueness of $\hat\varphi$ follows from uniqueness of
  $\tilde\varphi$. Finally, as $\|[\omega]\|_{\banext
    k,M./\mathcal{\bar Q}}\le\|\omega\|_{\banext k,M.}$,
  $\|\hat\varphi\|_{\hom(\banext k,M./\mathcal{\bar
      Q},N)}=\|\tilde\varphi\|_{\hom(\banext k,M.,N)}= \|\varphi\|_{\alt k,M,N.}$.
\end{proof}
\begin{rem}\label{rem:ext_compatibility}
  Note that if $M$ is an $L^\infty(\mu)$-module, we have an
  $\real$-linear surjection
  \begin{equation}
    \banext k,M.\to\lmodext k,\mu,M.
  \end{equation}
with norm at most $1$; similarly, if $M$ is an $L^\infty(\mu)$-normed
module, we have an $L^\infty(\mu)$-linear surjection
\begin{equation}
  \lmodext k,\mu,M.\to\lnmodext k,\mu,M.
\end{equation} with norm at most $1$.
\end{rem}
\subsection{Alberti representations in Banach spaces}
\label{subsec:alberti_banach}
In this Subsection we prove a refinement for the production of Alberti
representations in Banach spaces when the speed and direction are
specified using bounded linear maps.
\begin{thm}\label{thm:alberti_banach}
  Suppose that $Z$ is a separable Banach space, $\mu$ is a Radon measure on $Z$
  and suppose that $f:Z\to\real^q$ and $g:Z\to\real$ are bounded
  linear maps. Let $\cone(w,\alpha)$ be a $q$-dimensional cone field on
  $Z$ and $\delta:Z\to(0,\infty)$ a Borel map; then the following are
  equivalent:
  \begin{enumerate}
  \item The measure $\mu$ admits an Alberti representation in the
    $f$-direction of $\cone(w,\alpha)$ with $g$-speed $>\delta$.
  \item The measure $\mu$ admits a
    $(\delta/\|g\|_{Z^*},1)$-biLipschitz Alberti representation
    $\albrep.=(P,\nu)$ in the $f$-direction of $\cone(w,\alpha)$ with
    $g$-speed $>\delta$ and such that $\spt P\subset\curves(Z)$ and
    $\nu_\gamma=h\Psi_\gamma$ where $h$ is a Borel function on $Z$ and
    \begin{equation}
      \label{eq:alberti_banach_s1}
      \Psi_\gamma=\mpush\gamma.\lebmeas\mrest [0,1].
    \end{equation}
  \end{enumerate}
\end{thm}
\begin{proof}[Proof of Theorem \ref{thm:alberti_banach}]
  It suffices to show that (1) implies (2). For the moment, we assume
  that the functions $w$, $\alpha$ and $\delta$ are constant and that
  the set $\spt\mu$ is compact. By rescaling $g$ and $\delta$, we can
  assume that $\|g\|_{Z^*}=1$. Note that $\spt\mu$ must contain a
  fragment $\gamma$ with $(g\circ\gamma)'(t)>\delta\metdiff\gamma(t)$
  and $(f\circ\gamma)'(t)\in\cone(w,\alpha)$ for
  $\lebmeas\mrest\dom\gamma$-a.e.~$t$. In particular, there is a
  vector $z\in Z$ in the unit sphere of $Z$ satisfying $g(z)\ge\delta+1/n_0$ and
  $f(z)\in\bar\cone(w,\alpha-1/n_0)$ for some $n_0$. Let
  $\convgeo$ denote the closed convex hull of $\spt\mu\cup(\spt\mu+z)$
  in $Z$ and note that $\convgeo$ is compact. For $n\in\natural$
 let $\mathcal{G}_n$ denote
  the compact set of all $(\delta,1)$-biLipschitz maps $\gamma:[0,1]\to
  \convgeo$ satisfying:
  \begin{align}
    \sgn(t-s)\left(f\circ\gamma(t)-f\circ\gamma(s)\right)&\in\bar\cone(w,\alpha-1/n)\\
    \sgn(t-s)\left(g\circ\gamma(t)-g\circ\gamma(s)\right)&\ge(\delta+1/n)|t-s|.
  \end{align} 
  \par Applying \sync biLip_dis.\ in \cite{deralb} repeatedly,
  we obtain a decomposition $\mu=\mu'+\mu\mrest F$ where $\mu'$ has an
  Alberti representation of the desired form and $F\subset\spt\mu$ is an
  $F_{\sigma\delta}$ which is $\mathcal{G}_n$-null for every
  $n$. We elucidate the first two steps of the induction: one first
  writes $\mu=\mu_{\mathcal{G}_1}+\mu\mrest F_1$ where
  $\mu_{\mathcal{G}_1}$ admits an Alberti representation whose
  probability measure $P_1$ is concentrated on $\mathcal{G}_1$ and
  where $F_1$ is an $F_\sigma$-set which is $\mathcal{G}_1$-null. As a consequence, the measures
  $\mu\mrest F_1$ and $\mu_{\mathcal{G}_1}$ are singular. In the second
  step one applies \sync biLip_dis.\ in \cite{deralb} to $\mu\mrest
  F_1$ obtaining $\mu\mrest F_1$ = $\mu_{\mathcal{G}_2}+\mu\mrest F_2$
  where $\mu_{\mathcal{G}_2}$ admits an Alberti representation whose
  probability measure $P_2$ is concentrated on $\mathcal{G}_2$ and
  where $F_2\subset F_1$ is an $F_\sigma$-set which is
  $\mathcal{G}_2$-null (and also $\mathcal{G}_1$-null being a subset
  of $F_1$). One continues in this way and
  at the end one lets $\mu'=\sum_{n=1}^\infty\mu_{\mathcal{G}_n}$ and $F=\bigcap_nF_n$.
  \par We now show that for each fragment $\gamma\in\frags(\spt\mu)$ in the
  $f$-direction of $\cone(w,\alpha)$ and with $g$-speed $>\delta$, the
  set $F$ is $\hmeas 1._\gamma$-null; by (1), this will imply that $\mu(F)=0$.
  \par Let $\gamma$ be such a fragment and assume that it is
  $L$-Lipschitz. Note that, if we find countably many compact sets
  $K_\alpha\subset\dom\gamma$ with $\hmeas 1._{\gamma|K_\alpha}(F)=0$
  and $\lebmeas(\dom\gamma\setminus\bigcup_\alpha K_\alpha)=0$,
    then $\hmeas 1._\gamma(F)=0$. This allows to use Egorov and Lusin's
  Theorems to simplify the discussion.
  \par Concretely, fix $\varepsilon>0$ and use Lusin's Theorem
  \cite[Thm.~7.1.13]{bogachev_measure} to find a compact set $K_0\subset\dom\gamma$ with
  $\lebmeas(\dom\gamma\setminus K_0)<\varepsilon$ and such that
  $(f\circ\gamma)'$ is continuous when restricted to $K_0$. Applying
  Egorov's Theorem \cite[Thm.~7.1.12]{bogachev_measure} we can find another compact set
  $K_1\subset K_0$ with $\lebmeas (K_0\setminus K_1)<\varepsilon$ and
  such that:
  \begin{equation}
    \label{eq:heike1}
    \lim_{n\to\infty}\sup_{\substack{s,t\in K_1\\ 0<|s-t|\le 1/n}}\frac{\|
      f(\gamma(t)) - f(\gamma(s)) - (f\circ\gamma)'(t)(t-s)
      \|}
    {|t-s|} = 0.
  \end{equation}
  Havin fixed $\varepsilon_0>0$ can then choose $n_1$ such that:
  \begin{align}
    \sup_{\substack{s,t\in K_1\\ 0<|s-t|\le 1/n_1}}\frac{\|
      f(\gamma(t)) - f(\gamma(s)) - (f\circ\gamma)'(t)(t-s)
      \|}
    {|t-s|} &< \frac{\varepsilon_0}{2}\\
    \sup_{\substack{s,t\in K_1\\ 0\le |s-t|\le
        1/n_1}}\|(f\circ\gamma)'(t) - (f\circ\gamma)'(s)\|&<\frac{\varepsilon_0}{2},
  \end{align}
  and subdivide $K_1$ into finitely many compact subsets
  $\{K_{1,j}\}_j$ of diameter $<1/(2n_1)$. Having chosen for each $j$
    a $t_j\in K_{1,j}$ and having let $w_j=(f\circ\gamma)'(t_j)$ we
    conclude that:
    \begin{equation}
      \label{eq:heike2}
      \|f(\gamma(t)) - f(\gamma(s))
      -w_j(t-s)\|\le\varepsilon_0|t-s|\quad(\forall t,s\in K_{1,j}).
    \end{equation}
    \par Thus the previous argument shows that after subdividing the
    domain of $\gamma$ we can assume that $\gamma$ is in the
  $f$-direction of $\bar\cone(w,\alpha-1/n_2)$ for some
  $n_2\in\natural$ that can be made arbitrarily large choosing $n_1$ 
  appropriately. Passing to a further subdivision and applying a
  similar argument to the function $g$, we can also assume that
  $\gamma$ has $g$-speed $\ge\delta+1/n_2$. Finally, letting
  $I_\gamma$ denote the minimal interval containing $\dom\gamma$,
  applying the Lebesgue Differentiation Theorem and passing to a
  further subdivision of $\dom\gamma$ (and restricting $\gamma$) 
  we can assume that any point in $I_\gamma$ is
  within distance $\diam(I_\gamma)/(n_2(L+1000+q^8))$ from a point of
  $\dom\gamma$. For further details we refer to the argument of 
  \sync alberti_rep_prod.\ in \cite{deralb}.
    \par We now use the fact that $\convgeo$ is convex and that the
  functions $f$ and $g$ are linear to extend $\gamma$ to an
  $(L+1)$-Lipschitz map $\tilde\gamma:I_\gamma\to Z$. In fact,
  following \cite[2.79]{deralb} on each component $(u,v)$ of
  $I_\gamma\setminus\dom\gamma$ we let:
  \begin{equation}
    \label{eq:heike3}
    \tilde\gamma(t) = \frac{t - u}{v - u}\gamma(v) + \frac{v - t}{v - u}\gamma(u).
  \end{equation}
  Using the linearity of $f$ and $g$ we also conclude that
  $\tilde\gamma$ is in the $f$-direction of
  $\bar\cone(w,\alpha-1/(2n_2))$ with $g$-speed $\ge\delta+1/(2n_2)$.
  \par Now, after
  precomposing $\tilde\gamma$ with an affine map and dividing $I_\gamma$
  into smaller subintervals, we can reduce to the case in which
  $\tilde\gamma$ is $1$-Lipschitz,
  $I_\gamma\subset[0,1]$ and the left extremum of $I_\gamma$ is $0$. 
  Letting $t_0$ denote the right extremum of
  $I_\gamma$, we extend $\tilde\gamma$ to $[t_0,1]$ by letting
  $\tilde\gamma|[t_0,1]$ be the segment joining $\tilde\gamma(t_0)$ to
  $\tilde\gamma(t_0)+(1-t_0)z$. Note that
  $\metdiff\tilde\gamma\le 1$ and, letting $n_3=\max(n_0,n_1,n_2)$, we have
  $(g\circ\tilde\gamma)'\ge\delta+1/(2n_3)$ and
  $(f\circ\tilde\gamma)'\in\bar\cone(w,\alpha-1/(2n_3))$. In particular,
  $\tilde\gamma\in\mathcal{G}_{n_3}$ which implies
 $\hmeas 1._{\tilde\gamma}(F)=0$ and then $\hmeas 1._\gamma(F)=0$.
\par The case in which $\spt\mu$ is not compact and the functions $w$,
$\alpha$ and $\delta$ are not constant, is treated by using Egorov and
Lusin's Theorems like in the last part of the proof of \sync alberti_rep_prod.\ in \cite{deralb}.
\end{proof}
\subsection{Renorming}\label{subsec:renorming}
The goal of this Subsection is the proof of the following result about
renorming the module $\wder\mu.$ by taking a biLipschitz deformation
of the metric on $X$.
\begin{thm}
  \label{thm:renorm} Let $(X,d)$ be a Polish space and $\mu$ a Radon
  measure on $X$. For each $\varepsilon>0$ there is a metric $\edst , .$
  which satisfies
  \begin{equation}
    \label{eq:renorm_s1}
    d\le\edst , .\le (1+\epsi)d
  \end{equation}
  and such that the corresponding local norm
  $\elocnorm\,\cdot\,,{\wder\mu.}.$ is strictly convex.
\end{thm}
We now fix some notation that will be used throught this
Subsection. We let $\{\psi_n\}$ be a countable generating set for the
Lipschitz algebra $\lipalg X.$ where $\psi_1$ is the constant function
equal to
$1$, and where we assume that for $n>1$ each function
$\psi_n$ is $1$-Lipschitz and vanishes at a fixed basepoint $\tilde x$. We then introduce the pseudometrics
\begin{align}
  \label{eq:semimetrics}
  \Psi(x,y)&=\left\|\left(\frac{\psi_n(x)-\psi_n(y)}{n}\right)_n\right\|_{l^2}\\
  \Psi_M(x,y)&=\left(\sum_{n=1}^M\frac{(\psi_n(x)-\psi_n(y))^2}{n^2}\right)^{1/2},
\end{align}
and observe that $\Psi_M\le\Psi\le \frac{\pi}{\sqrt{6}}d$.
We also define functions
\begin{equation}
    \label{eq:first_fun}
\begin{aligned}
    \Phi: X&\to l^2\\
    x&\mapsto\left(\frac{\psi_n(x)}{n}\right)_n
    \end{aligned}
\end{equation} and 
\begin{equation}
  \label{eq:second_fun}
\begin{aligned}
  \Phi_M: X &\to \real^M\\
  x&\mapsto\left(\frac{\psi_n(x)}{n}\right)_{n=1}^M,
\end{aligned}
\end{equation}
and observe that $\Phi$ and $\Phi_M$ are
$\frac{\pi}{\sqrt{6}}$-Lipschitz with respect to the distance $d$.
We finally let
\begin{equation}
  \label{eq:edist}
  \edst , . = d + \varepsilon \Psi
\end{equation} so that
\begin{equation}\label{eq:biLip}
  d\le\edst , . \le \left(1+\varepsilon\frac{\pi}{\sqrt{6}}\right) d.
\end{equation} Note that, given a derivation $D$, after choosing a
Borel representative for each $D\psi_n$, we obtain Borel
maps\footnote{The Borel $\sigma$-algebras for the strong and the weak
  topologies on $l^2$ coincide}
\begin{equation}
  \label{eq:borel_phi}
  \begin{aligned}
    D\Phi: X&\to l^2\\
    x&\mapsto \left(\frac{D\psi_n(x)}{n}\right)_n,
  \end{aligned}
\end{equation}
and 
\begin{equation}
  \label{eq:borel_phiM}
  \begin{aligned}
    D\Phi_M: X&\to \real^M\\
    x&\mapsto\left(\frac{D\psi_n(x)}{n}\right)_{n=1}^M.
  \end{aligned}
\end{equation}
We will now prove that the local norm
$\elocnorm\,\cdot\,,{\wder\mu.}.$ corresponding to the distance $\edst
, .$ is strictly convex. We start with the following Lemma, which is
essentially folklore and whose proof is included for completeness.
\begin{lem}
  \label{lem:change_var}
If $g\in C^1(\real^k)$ and the functions $\{\psi_i\}_{i=1}^k$ are in
$\lipalg X.$, then for any derivation $D\in\wder\mu.$ it follows that
\begin{equation}
  \label{eq:change_var}
  Dg(\psi_1,\cdots,\psi_k)=\sum_{l=1}^k\frac{\partial g}{\partial y^l}(\psi_1,\cdots,\psi_k)D\psi_l.
\end{equation}
\end{lem}
\begin{proof}[Proof of Lemma \ref{lem:change_var}]
  The idea of the proof is essentially based on
  \cite[Thm.~3.5(i)]{ambrosio-kirch}. As the functions
  $\{\psi_i\}_{i=1}^k$ are bounded, letting $\psi:X\to\real^k$ be the
  Lipschitz function whose $i$-th component is $\psi_i$,
  there is a $k$-dimensional simplex $S$ (we take simplices to
    be closed) centred about the origin such
  that $\psi(X)$ lies in the interior of $S$. Using that $g\in C^1(\real^k)$,
  it is possible to construct Lipschitz functions $g_n:S\to\real$ such
  that:
  \begin{enumerate}
  \item there is $M_n\in\natural$ such that, if $S^{M_n}$ denotes the
    $M_n$-th iterated barycentric subdivision of $S$, the function
    $g_n$ is affine linear on each simplex $\Delta\in S^{M_n}$:
    \begin{equation}
      \label{eq:change_var_p1}
      g_n(v)=\langle V_{n,\Delta}, v\rangle + c_{n,\Delta}\quad(v\in\Delta).
    \end{equation}
  \item For each simplex $\Delta\in S^{M_n}$ one has
    \begin{align}
      \sup_{v\in\Delta}\left|g(v)-g_n(v)\right|&\le\frac{1}{n}\\      \label{eq:change_var_p2}
      \sup_{v\in\Delta}\left\|V_{n,\Delta}-\nabla g(v)\right\|_2&\le\frac{1}{n}.
    \end{align}
  \end{enumerate}
We now let
\begin{align}
  f(x)&=g\left(\psi_1(x),\cdots,\psi_k(x)\right)\\
  f_n(x)&=g_n\left(\psi_1(x),\cdots,\psi_k(x)\right),
\end{align} and observe that as $f_n|\psi^{-1}(\Delta)$ agrees with
the function
\begin{equation}
  x\mapsto \langle V_{n,\Delta}, \psi(x)\rangle + c_{n,\Delta},
\end{equation} the locality property of derivations implies that
\begin{equation}
  \label{eq:change_var_p3}
Df_n(x)=\langle V_{n,\Delta},D\psi(x)\rangle
\end{equation} for $\mu\mrest\psi^{-1}(\Delta)$-a.e.~$x$. As
$f_n\xrightarrow{\text{w*}} f$, \eqref{eq:change_var} follows from
\eqref{eq:change_var_p3} and \eqref{eq:change_var_p2}.
\end{proof}
\par The following Lemma is a key step in the proof of Theorem \ref{thm:renorm}.
\begin{lem}
  \label{lem:dir_density}
  Let $F:X\to\real^M$ be Lipschitz, $D\in\wder\mu.$ and $\theta:X\to(0,\pi/2)$ a Borel
  map. Let
  \begin{equation}
    \label{eq:dir_density}
    V_F=\{x:DF(x)\ne0\};
  \end{equation}
  then $\mu\mrest V_F$ admits an Alberti representation in the
  $F$-direction of $\cone\left(\frac{DF}{\|DF\|_2},\theta\right)$.
\end{lem}
\begin{proof}[Proof of Lemma \ref{lem:dir_density}]
The proof is essentially based on the argument used in
\sync lem:vector_alberti.\ in \cite{deralb} and details are included for
completeness. 
  We consider a Borel $L^\infty(\mu\mrest V_F)$-partition of unity
  $\left\{V^{(0)}_l\right\}_{l\in\natural}$ such that, for each $l$, there is a pair
    $(s_l,\theta_l)\subset(0,\infty)\times(0,\pi/2)$ with:
    \begin{align}
      \label{eq:dir_density_p1}
      \locnorm D,{\wder\mu\mrest V_F.}.(x)&\in(s_l,2s_l)\quad(\forall
      x\in V^{(0)}_l)\\
      \label{eq:dir_density_p2}
      \theta(x)&\in(\theta_l,2\theta_l)\quad(\forall
      x\in V^{(0)}_l);
    \end{align}
we further subdivide the $\left\{V^{(0)}_l\right\}_{l\in\natural}$ to
obtain  a Borel $L^\infty(\mu\mrest V_F)$-partition of unity
  $\left\{V^{(1)}_l\right\}_{l\in\natural}$ such that, for each $l$,
  \eqref{eq:dir_density_p1} and \eqref{eq:dir_density_p2} hold and
  there are $c_l>0$ and $\varepsilon_l^{(1)}\in(0,c_l/2)$ such that:
  \begin{equation}
    \label{eq:dir_density_p3}
    \left\|DF(x)\right\|_2\in(c_l,c_l+\varepsilon_l^{(1)})\quad(\forall
    x\in V^{(1)}_l);
  \end{equation}
note that the values of each $\varepsilon_l^{(1)}$ will be chosen later
depending on the corresponding values of $s_l$ and $\theta_l$ which
were obtained in the previous step. We finally subdivide the
$\left\{V^{(1)}_l\right\}_{l\in\natural}$ to obtain a Borel $L^\infty(\mu\mrest V_F)$-partition of unity
  $\left\{V^{(2)}_l\right\}_{l\in\natural}$ such that, for each $l$,
  \eqref{eq:dir_density_p1}, \eqref{eq:dir_density_p2} and
  \eqref{eq:dir_density_p3} hold and there are
  $w_l\in\mathbb{S}^{M-1}$ and
  $\varepsilon_l^{(2)}\in(0,\varepsilon_l^{(1)})$ such that:
  \begin{align}
    \label{eq:dir_density_p4}
    \cone(w_l,\theta_l/2)&\subset
    \cone\left(\frac{DF(x)}{\left\|DF(x)\right\|_2},\theta_l\right)\quad(\forall
    x\in V^{(2)}_l)\\
    \label{eq:dir_density_p5}
    \left\|\frac{DF(x)}{\left\|DF(x)\right\|_2}-w_l\right\|_2&\le\varepsilon_l^{(2)}\quad(\forall
    x\in V^{(2)}_l);
  \end{align}
note that the values of each $\varepsilon_l^{(2)}$ will be chosen later
depending on the corresponding values of $s_l$, $\theta_l$, $c_l$ and
$\varepsilon_l^{(1)}$ which
were obtained in the previous steps. We now estimate the error in
approximating $DF$ by $c_lw_l$ on $V^{(2)}_l$:
\begin{equation}
  \label{eq:dir_density_p6}
  \begin{aligned}
  \left\|DF-c_lw_l\right\|_2&\le\left\|DF-\|DF\|_2w_l\right\|_2+\left\|\,\|DF\|_2w_l-c_lw_l\right\|_2
  \\ &\le
  \|DF\|_2\,\left\|\frac{DF(x)}{\left\|DF(x)\right\|_2}-w_l\right\|_2+\|DF\|_2-c_l\\
  &\le\underbrace{(c_l+\varepsilon_l^{(1)})\varepsilon_l^{(2)}+\varepsilon_l^{(1)}}_{\eta_l}.
\end{aligned}
\end{equation}
In particular, if $u$ is a unit vector orthogonal to $w_l$, 
\begin{equation}
  \label{eq:dir_density_p7}
  \chi_{V^{(2)}_l}\left|D\langle
    u,F\rangle\right|=\chi_{V^{(2)}_l}\left|\langle u,
    DF-w_lc_l\rangle\right|\le\frac{\eta_l}{s_l}\locnorm D,{\wder\mu.}..
\end{equation}
We now suppose that the Borel set $S_l\subset V^{(2)}_l$ is
$\frags(X,F,\tilde\delta_l,w_l,\theta_l/2)$-null; using
\eqref{eq:dir_density_p7} and Lemma \ref{lem:dst_appx} 
(compare also \sync lem:loc_estimate_dist.\ and
  \sync lem:loc_estimate_fnull.\ in \cite{deralb} for details) we obtain
\begin{equation}
  \label{eq:dir_density_p8}
  \chi_{S_l}\left|D\langle
    w_l,F\rangle\right|\le\left(\tilde\delta_l+(M-1)\frac{\eta_l}{s_l}\cot(\theta_l/2)\right)\locnorm D,{\wder\mu.}.;
\end{equation}
on the other hand, we have
\begin{equation}
  \label{eq:dir_density_p9}
  \chi_{V^{(2)}_l}D\langle w_l,F\rangle\ge\chi_{V^{(2)}_l}(c_l-\eta_l).
\end{equation} In particular, if $\mu(S_l)>0$ we have
\begin{equation}
  \label{eq:dir_density_p10}
  \tilde\delta_l\ge\frac{c_l-\eta_l}{2s_l}-(M-1)\frac{\eta_l}{s_l}\cot(\theta_l/2);
\end{equation} this implies that $\mu\mrest V^{(2)}_l$ admits an
Alberti representation $\albrep l.$ in the $F$-direction of
$\cone(w_l,\theta_l/2)$ with $F$-speed
\begin{equation}
  \label{eq:dir_density_p11}
  \ge\delta_l=\frac{c_l-2\eta_l}{2s_l}-(M-1)\frac{\eta_l}{s_l}\cot(\theta_l/2),
\end{equation} provided that $\delta_l$ is positive. Note that
\begin{equation}
  \label{eq:dir_density_p12}
  \delta_l=\frac{1}{2s_l}\left(c_l-2(c_l+\varepsilon_l^{(1)})\varepsilon_l^{(2)}-2\varepsilon_l^{(1)}\right)-(M-1)\frac{(c_l+\varepsilon_l^{(1)})\varepsilon_l^{(2)}+\varepsilon_l^{(1)}}{s_l}\cot(\theta_l/2);
\end{equation} if at each step the
$\varepsilon_l^{(1)}$ and $\varepsilon_l^{(2)}$ are chosen sufficiently
small, one can ensure that $\delta_l>0$. The proof is completed by
gluing together the $\{\albrep l.\}$ (Theorem \ref{thm:alb_glue}) and
using \eqref{eq:dir_density_p4}.
\end{proof}
\begin{lem}
  \label{lem:local_norms}
  The local norms $\locnorm\,\cdot\,,{\wder\mu.}.$ and
  $\elocnorm\,\cdot\,,{\wder\mu.}.$ are related by the following
  equation:
  \begin{equation}
    \label{eq:local_norms_s1}
    \elocnorm D,{\wder\mu.}.=\locnorm
    D,{\wder\mu.}.+\varepsilon\left\|D\Phi\right\|_{l^2}\quad(\forall D\in\wder\mu.).
  \end{equation}
\end{lem}
\begin{proof}[Proof of Lemma \ref{lem:local_norms}]
  We first show that
  \begin{equation}
    \label{eq:local_norms_p1}
    \elocnorm D,{\wder\mu.}.\le\locnorm D,{\wder\mu.}.+\varepsilon\|D\Phi\|_{l^2}
  \end{equation} by showing that, for each $x\in X$, the distance
  function $\edst x,\cdot.$ satisfies
  \begin{equation}
     \label{eq:local_norms_p2}
     \left|D\edst x,\cdot.\right|\le\locnorm D,{\wder\mu.}.+\varepsilon\|D\Phi\|_{l^2}.
 \end{equation}
 Without loss of generality, we can assume that $X$ is bounded.
  Let $\emdst , .=d+\varepsilon\Psi_M$ and observe that the sequence of Lipschitz
  functions $\{\emdst x,\cdot.\}_{M\in\natural}$ converges to
  $\edst x,\cdot.$, in the weak*-topology, as $M\nearrow\infty$. As
  $d(x,\cdot)$ is $1$-Lipschitz with respect to $d$, we have:
  \begin{equation}
    \label{eq:local_norms_p3}
    \left|Dd(x,\cdot)\right|\le\locnorm D,{\wder\mu.}..
  \end{equation}
  On the closed set $C_0=\left\{y:\Psi_M(x,y)=0\right\}$, one has
  $D\Psi_M(x,\cdot)=0$ by locality of derivations. For $\delta>0$ consider the closed set
  $C_\delta=\left\{y:\Psi_M(x,y)\ge\delta\right\}$. We can find a
  function $g:\real^M\to(0,\infty)$ of class $C^1(\real^M)$ such that,
  if for a $v\in\real^M$ one has 
  \begin{equation}
    \label{eq:local_norms_p4}
    \left(\sum_{n=1}^M\frac{|v_n|^2}{n^2}\right)^{1/2}\ge\frac{\delta}{2},
  \end{equation} then
  \begin{equation}
    \label{eq:local_norms_p5}
    g(v)=\left(\sum_{n=1}^M\frac{|v_n|^2}{n^2}\right)^{1/2}.
  \end{equation}
In particular, on $C_\delta$, the function $\Psi_M(x,\cdot)$ coincides with
\begin{equation}
  \label{eq:local_norms_p6}
  g\left(\psi_1(\cdot)-\psi_1(x),\cdots,\psi_M(\cdot)-\psi_M(x)\right),
\end{equation} and Lemma \ref{lem:change_var} gives
\begin{equation}
  \label{eq:local_norms_p7}
  D\Psi_M(x,y)=\frac{1}{\Psi_M(x,y)}\sum_{n=1}^M\frac{\psi_n(y)-\psi_n(x)}{n}\,\frac{D\psi_n(y)}{n}
\end{equation} for $\mu\mrest C_\delta$-a.e.~$y$. Using the Cauchy
inequality and a sequence $\delta_n\searrow0$, we conclude that
\begin{equation}
  \label{eq:local_norms_p8}
  \left|D\Psi_M(x,\cdot)\right|\le\left\|D\Phi_M\right\|_2.
\end{equation} Combining \eqref{eq:local_norms_p3} and
\eqref{eq:local_norms_p8} we obtain \eqref{eq:local_norms_p2} and so
\eqref{eq:local_norms_p1} is proved.
\par We now show that
\begin{equation}
  \label{eq:local_norms_p9}
      \elocnorm D,{\wder\mu.}.\ge\locnorm
      D,{\wder\mu.}.+\varepsilon\|D\Phi\|_{l^2}, 
\end{equation} and we will assume that a Borel representative has been
chosen for each $D\psi_n$. We first consider the Borel set $V_0$ where
$\left\|D\Phi\right\|_{l^2}=0$. Having fixed $\eta>0$, we take a Borel
$L^\infty(\mu\mrest V_0)$-partition of unity
$\left\{U_\alpha\right\}$ such that, for each $\alpha$, there is a
function $f_\alpha$ which is $1$-Lipschitz  with respect to the
distance $d$ and satisfying:
\begin{equation}
  \label{eq:local_norms_p10}
  \chi_{U_\alpha}Df_\alpha\ge(1-\eta)\chi_{U_\alpha}\locnorm D,{\wder\mu.}.;
\end{equation} this implies that
\begin{equation}
  \label{eq:local_norms_p10b}
  \chi_{V_0}\elocnorm D,{\wder\mu.}.\ge(1-\eta)\chi_{V_0}\locnorm D,{\wder\mu.}..
\end{equation} We now consider the Borel set $V_1$ where
$\left\|D\Phi\right\|_{l^2}>0$. For each $\eta>0$, we take an
$L^\infty(\mu\mrest V_1)$-partition of unity
$\left\{U_\alpha\right\}$, where each set $U_\alpha$ is compact and
such that for each $\alpha$ there is a
quadruple $(f_\alpha,M_\alpha,\theta_\alpha,\delta_\alpha)$ satisfying:
\begin{description}
\item[(P1)] The function $f_\alpha$ is $1$-Lipschitz with respect to
  the distance $d$, $M_\alpha$ is a
  natural number, $\theta_\alpha\in(0,\pi/2)$, and $\delta_\alpha>0$.
\item[(P2)] The following inequality holds \begin{equation}
  \label{eq:local_norms_p10c}
  \chi_{U_\alpha}Df_\alpha\ge(1-\eta)\chi_{U_\alpha}\locnorm D,{\wder\mu.}..
\end{equation}
\item[(P3)] The Borel functions $\left\|D\Phi\right\|_{l^2}$ and
  $\left\|D\Phi_{M_\alpha}\right\|_2$ are continuous on
  $U_\alpha$ and satisfy
  \begin{equation}
    \label{eq:local_norms_p11}
    \left\|D\Phi_{M_\alpha}\right\|_2\ge(1-\eta)\left\|D\Phi\right\|_{l^2}\ge\delta_\alpha>0.
  \end{equation}
\item[(P4)] For all $x,y\in U_\alpha$, if
  $u\in\cone\left(\frac{D\Phi_{M_\alpha}(x)}{\left\|D\Phi_{M_\alpha}(x)\right\|_2},2\theta_\alpha\right)\cap
  {\mathbb S}^{M_\alpha-1}$, then
  \begin{equation}
    \label{eq:local_norms_p12}
    \left\langle u,D\Phi_{M_\alpha}(y)\right\rangle\ge(1-\eta)\left\|D\Phi_{M_\alpha}(y)\right\|_2.
  \end{equation}
\end{description}
By Lemma \ref{lem:dir_density} the measure $\mu\mrest U_\alpha$ admits
an Alberti representation in the $\Phi_{M_\alpha}$-direction of the cone
field
$\cone\left(\frac{D\Phi_{M_\alpha}}{\left\|D\Phi_{M_\alpha}\right\|_2},\theta_\alpha\right)$;
in particular, for $\mu\mrest U_\alpha$-a.e.~$x$, there is a fragment
$\gamma_x\in\frags(U_\alpha)$ such that:
\begin{enumerate}
\item $0$ is a Lebesgue density point of $\dom\gamma_x$ and
  $\gamma_x(0)=x$.
\item There is a
  $v_x\in\cone\left(\frac{D\Phi_{M_\alpha}(x)}{\left\|D\Phi_{M_\alpha}(x)\right\|_2},\theta_\alpha\right)$
  with
  \begin{equation}
    \label{eq:local_norms_p13}
    \Phi_{M_\alpha}\left(\gamma(r)\right)=\Phi_{M_\alpha}(x)+v_xr+o(r).
  \end{equation}
\end{enumerate}
In particular, there are $r_x,R_x>0$ such that for each $y\in\ball x,R_x.\cap U_\alpha$\footnote{the
  ball can be taken either with respect to $d$ or $\edst , .$.}, one has
\begin{equation}
  \label{eq:local_norms_p14}
  \frac{\Phi_{M_\alpha}\left(\gamma_x(r_x)\right)-\Phi_{M_\alpha}(y)}{\left\|
    \Phi_{M_\alpha}\left(\gamma_x(r_x)\right)-\Phi_{M_\alpha}(y)\right\|_2}\in\cone\left(\frac{D\Phi_{M_\alpha}(x)}{\left\|D\Phi_{M_\alpha}(x)\right\|_2},2\theta_\alpha\right).
\end{equation}  Let
\begin{equation}
  \label{eq:local_norms_p15}
  \tilde f_\alpha=f_\alpha-\varepsilon
  \Psi_{M_\alpha}\left(\gamma_x(r_x),\cdot\right),
\end{equation} and observe that $\tilde f_\alpha$ is $1$-Lipschitz
with respect to the distance $\edst , .$ and that
\begin{equation}
  \label{eq:local_norms_p16}
  D\tilde f_\alpha= D f_\alpha-\varepsilon
  D\Psi_{M_\alpha}\left(\gamma_x(r),\cdot\right);
\end{equation} an argument similar to that used to prove
\eqref{eq:local_norms_p7} shows that for $\mu\mrest(U_\alpha\cap\ball
x,R_x.)$-a.e.~$y$,
\begin{multline}
   \label{eq:local_norms_p17}
 D\Psi_{M_\alpha}\left(\gamma_x(r_x),y\right)=-\frac{\left\langle\Phi_{M_\alpha}\left(\gamma_x(r_x)\right)
     - \Phi_{M_\alpha}(y),D\Phi_{M_\alpha}(y)\right\rangle}{\left\|\Phi_{M_\alpha}\left(\gamma_x(r_x)\right)
- \Phi_{M_\alpha}(y)\right\|_2}\\
\le-(1-\eta)\left\|D\Phi_{M_\alpha}\right\|_2,
\end{multline} where in the last step we used
\eqref{eq:local_norms_p14} and \textbf{(P4)}. 
Combining \eqref{eq:local_norms_p17} with \textbf{(P2)} we obtain
\begin{equation}
  \label{eq:local_norms_p18}
\chi_{U_\alpha}D\tilde f_\alpha\ge(1-\eta)\chi_{U_\alpha}\locnorm
D,{\wder\mu.}.+\varepsilon(1-\eta)^2\chi_{U_\alpha}\|D\Phi\|_{l^2},
\end{equation} which implies
\begin{equation}
  \label{eq:local_norms_p19}
\chi_{V_1}\elocnorm
D,{\wder\mu.}.\ge(1-\eta)\chi_{V_1}\locnorm
D,{\wder\mu.}.+\varepsilon(1-\eta)^2\chi_{V_1}\|D\Phi\|_{l^2};
\end{equation} letting $\eta\searrow0$ in \eqref{eq:local_norms_p19} and
\eqref{eq:local_norms_p10b}, \eqref{eq:local_norms_p9} follows.
\end{proof}
\begin{proof}[Proof of Theorem \ref{thm:renorm}]
  Because of \eqref{eq:biLip}, we just need to show that the local
  norm $\elocnorm\,\cdot\,,{\wder\mu.}.$ associated to $\edst , .$ is
  strictly convex. Consider derivations $D_1,D_2\in\wder\mu.$ and
  suppose that for $\mu\mrest U$~a.e.~$x\in U$ one has:
  \begin{equation}
    \label{eq:renorm_p1}
    \elocnorm D_1+D_2,{\wder\mu.}.(x)=\elocnorm
    D_1,{\wder\mu.}.(x)+\elocnorm D_2,{\wder\mu.}.(x);
  \end{equation}
by Lemma \ref{lem:local_norms} we have
\begin{multline}
  \label{eq:renorm_p2}
  \locnorm
  D_1+D_2,{\wder\mu.}.(x)+\varepsilon\left\|D_1\Phi(x)+D_2\Phi(x)\right\|_{l^2}
  = \locnorm
  D_1,{\wder\mu.}.(x)+\varepsilon\left\|D_1\Phi(x)\right\|_{l^2} \\+
  \locnorm D_2,{\wder\mu.}.(x)+\varepsilon\left\|D_2\Phi(x)\right\|_{l^2};
\end{multline}
because
\begin{align}
  \locnorm
  D_1+D_2,{\wder\mu.}.&\le \locnorm
  D_1,{\wder\mu.}.+ \locnorm D_2,{\wder\mu.}.\\
\left\|D_1\Phi+D_2\Phi\right\|_{l^2} & \le
\left\|D_1\Phi \right\|_{l^2} + \left\|D_2\Phi \right\|_{l^2},
\end{align} after choosing Borel representatives of $D_1\Phi$ and
$D_2\Phi$, we find a Borel $V\subset U$ with $\mu(U\setminus V)=0$ and
such that:
\begin{equation}
  \label{eq:renorm_p3}
  \left\|D_1\Phi(x)+D_2\Phi(x)\right\|_{l^2} =
\left\|D_1\Phi(x) \right\|_{l^2} + \left\|D_2\Phi(x)
\right\|_{l^2}\quad(\forall x\in V).
\end{equation}
The strict convexity of the norm on $l^2$ implies that for each $x\in
V$ the vectors $D_1\Phi(x)$ and $D_2\Phi(x)$ are linearly dependent. Let
\begin{align}
  \tilde V_1&=\left\{(x,\lambda)\in V\times[-1,1]: D_1\Phi(x)=\lambda
    D_2\Phi(x)\right\}\\
  \tilde V_2&=\left\{(x,\lambda)\in V\times[-1,1]: D_2\Phi(x)=\lambda D_1\Phi(x)\right\};
\end{align}
then $\tilde V_1$ and $\tilde V_2$ are Borel subsets of
$X\times[-1,1]$ and, denoting by $V_i$ the projection of $\tilde V_i$
on $X$, we have $V=V_1\cup V_2$. Note that for each $x$ the section
$(\tilde V_i)_x$ is compact; in particular, by the Lusin-Novikov
Uniformization Theorem \cite[Thm.~18.10]{kechris_desc} 18.10, the
sets $V_1$ and $V_2$ are Borel and admit Borel uniformizing functions
$\sigma_i:V_i\to[-1,1]$. In particular,
\begin{align}
  \label{eq:renorm_p4}
  \chi_{V_1}D_1\Phi&=\sigma_1 \chi_{V_1}D_2\Phi\\\label{eq:renorm_p5}
  \chi_{V_2}D_2\Phi&=\sigma_2 \chi_{V_2}D_1\Phi;
\end{align}
as the $\{\psi_n\}$ generate $\lipalg X.$, \eqref{eq:renorm_p4} and
\eqref{eq:renorm_p5} imply that
(\ref{eq:str_conv2})--~(\ref{eq:str_conv4}) hold by letting
$\lambda_1=\chi_{V_1}\sigma_1$ and $\lambda_2=\chi_{V_2}\sigma_2$.
\end{proof}
\bibliographystyle{alpha}
\bibliography{curr_alb_biblio}
\end{document}